\documentclass[11pt]{amsart}
\usepackage{amssymb,mathrsfs,graphicx,enumerate,color}
\usepackage{colortbl}
\definecolor{black}{rgb}{0.0, 0.0, 0.0}
\definecolor{red}{rgb}{1.0, 0.5, 0.5}

\title[   ]{Well-posedness of the Riemann problem with two shocks for the isentropic Euler system in a class of vanishing physical viscosity limits}

\author[Kang]{Moon-Jin Kang}
\address[Moon-Jin Kang]
{ Department of Mathematical Sciences, \newline
Korea Advanced Institute of
Science and Technology \\
 Daejeon 34141, Korea}
 \email{moonjinkang@kaist.ac.kr}

\author[Vasseur]{Alexis F. Vasseur}
\address[Alexis F. Vasseur]{\newline Department of Mathematics, \newline The University of Texas at Austin, Austin, TX 78712, USA}
\email{vasseur@math.utexas.edu}

\newtheorem{theorem}{Theorem}[section]
\newtheorem{lemma}{Lemma}[section]

\newtheorem{proposition}{Proposition}[section]
\newtheorem{remark}{Remark}[section]

\newcommand{\bbr}{\mathbb R}

\newcommand{\pt}{p(\tilde{v})}
\newcommand{\deo}{\delta_0}

\numberwithin{figure}{section}

\newcommand{\beq}{\begin{equation}}
\newcommand{\eeq}{\end{equation}}
\newcommand{\bsp}{\begin{split}}
\newcommand{\esp}{\end{split}}

\newcommand{\pa}{{\partial}}

\newcommand{\tiv}{\tilde{v}}
\newcommand{\tih}{\tilde{h}}

\newcommand{\lam}{\lambda}

\newcommand{\sgn}{{\text{\rm sgn}}}

\newcommand{\RR}{{\mathbb R}}

\def\eps{\varepsilon}

\newcommand\adots{\mathinner{\mkern2mu\raise1pt\hbox{.}
\mkern3mu\raise4pt\hbox{.}\mkern1mu\raise7pt\hbox{.}}}

\def\charf {\mbox{{\text 1}\kern-.30em {\text l}}}

\def\lam{\lambda}  

\newcommand \deltaz {{\delta_0}}
\newcommand \deltao {{\delta_1}}

\newcommand \deltat {{\delta_2}}
\newcommand{\eo}{{\varepsilon_1}}
\newcommand{\et}{{\varepsilon_2}}

\newcommand{\bw}{{\bf w}}
\newcommand{\bv}{{\bf v}}

 \newcommand{\bmat}{\begin{pmatrix}} 
  \newcommand{\emat}{\end{pmatrix}}
  
\newcommand{\s}{\sigma}

\newcommand{\vn}{v} 
\newcommand{\un}{u} 

\begin{document}
\bibliographystyle{plain}

\date{\today}


\thanks{\textbf{Acknowledgment.}  The first author was partially supported by the NRF-2019R1C1C1009355.
The second author was partially supported by the NSF grant: DMS 1614918. 
}

\begin{abstract}
We consider the Riemann problem composed of two shocks for the 1D Euler system. We show that the Riemann solution with two shocks is stable and unique in the class of weak inviscid  limits of solutions to the  Navier-Stokes equations with initial data with bounded energy. This work extends to the case of two shocks a previous result of the authors in the case of a single shock. It is based on the method of weighted relative entropy with shifts known as $a$-contraction theory. A major difficulty due to the method is that very little control is available on the shifts. A modification of the construction of the shifts is needed to ensure that the two shock waves are well separated,  at the level of the Navier-Stokes system, even when subjected  to    large perturbations. This work put the foundations needed to  consider a large family of interacting waves. It is a key result in the program to solve the Bianchini-Bressan conjecture, that is the inviscid limit of solutions to the  Navier-Stokes equation to the unique BV solution of the Euler equation, in the case of small BV initial values. 
  \end{abstract}
\maketitle \centerline{\date}

\tableofcontents

\section{Introduction}
\setcounter{equation}{0}

Consider  the one-dimensional barotropic Navier-Stokes system in the Lagrangian coordinates:
\begin{align}
\begin{aligned}\label{inveq}
\left\{ \begin{array}{ll}
    \displaystyle{    \vn_t - \un_x =0,}\\[0.3cm]
        \displaystyle{    \un_t+p(\vn)_x = \nu\Big(\frac{\mu(\vn)}{\vn} \un_x\Big)_x, \qquad t>0, \ x\in \RR,}\end{array} \right.
\end{aligned}
\end{align}
where $v$ denotes the specific volume, $u$ is the fluid velocity, and $p(v)$ is the pressure law. We consider the case of  a polytropic perfect gas where the pressure is given by
\beq\label{pressure}
p(v)= v^{-\gamma},\quad \gamma> 1,
\eeq
with  $\gamma$ the adiabatic constant. Here, $\mu$ denotes the viscosity coefficient given by 
\beq\label{mu-def}
\mu(v) = bv^{-\alpha},\qquad b>0 .
\eeq
We assume the following relating between $\alpha$ and $\gamma$:
\beq\label{ass-ag}
0<\alpha\le \gamma \le \alpha +1.
\eeq
This includes, for instance, the case of the viscous shallow water equation $\gamma=2$, $\alpha =1$ (see \cite{GP}).
\vskip0.3cm
We are interested in studying on the well-posedness of the inviscid limits $\nu\to 0$.
At least formally, the limit system of \eqref{inveq} is given by the isentropic Euler system:
\begin{align}
\begin{aligned} \label{Euler}
\left\{ \begin{array}{ll}
       v_t - u_x =0,\\[0.3cm]
       u_t+p(v)_x =0.\end{array}\ \right.
\end{aligned}
\end{align}
For  small $BV$ initial values, global weak solutions for conservation laws as \eqref{Euler} have been constructed in 1965 by Glimm \cite{Glimm}. The uniqueness of these weak entropy solutions in the class of small $BV$ functions has been proved later by Bressan, Crasta, and Piccoli  in 2000 \cite{Bressan1} (see also Liu and Yang \cite{Liu-Yang},  Bressan Lui and Yang \cite{Bressan2}, and Bressan \cite{Bressan}). 
In 2005, Bianchini and Bressan considered the inviscid limit of a fully parabolic model (with viscosity, and artificial diffusion in the $v$ equation), and showed that the solutions converge to the unique BV solution in the case of small BV initial values \cite{BB}.  In this work, they mention the problem of the inviscid limit of Navier-Stokes as an open problem. As today, the problem is still unsolved. 
\vskip0.3cm
This paper is a key step in  our program to solve this problem. One key difficulty for the problem is that obtaining a uniform BV estimate on the solutions of Navier-Stokes seems unattainable. Our general philosophy is to avoid completely  this step, by working only with the natural energy estimates via our  method of a-contraction with shifts. In \cite{Kang-V-NS17}, we used the method to show the stability, uniformly in $\nu$, of a single shock wave. In \cite{KV-unique19}, we apply the result to show that shock waves of \eqref{Euler} are unique in the class of  weak limits of \eqref{inveq} whose initial values converges to the shock wave. The basic idea is now to combine several such waves. Because the limit equation \eqref{Euler} has a finite speed of propagation, it is expected to be possible to estimate such evolution, and pass into the limits in the number of waves. However, many difficulties stem from  features of the method itself. First, at the level of the Navier-Stokes equation, the waves do interact at long distance, and are not independent anymore. More importantly,  we can obtain only rough  a priori control on the artificial shifts induced by the method. 
Since we do not have a priori control on the solutions, uniform in $L^\infty$  with respect to $\nu$, these shifts can be very oscillatory at the limit. For the  limit problems \eqref{Euler}, due to the separation of shock speeds, a 1-shock can never collide  from the left with   a 2-shock. This allows, in particular, to solve the Riemann problem. Due to the artificial shifts, and the lack of uniform bounds in $\nu$ on the solutions of  \eqref{inveq}, the separation of waves  is far more complicated at the level of Navier-Stokes. 
A key point is to show that shifts associated to two different families of shocks cannot produce artificial collisions.  We need to ensure that  a 1-shock cannot be pushed through the artificial shift, so much that it would collide with  a 2-shock from the left. This is crucial to recover the Riemann problem at the limit, and this is the problem solved  in this paper.
\vskip0.3cm
Consider the Riemann problem for \eqref{Euler} with the Riemann initial data:
\beq\label{rini}
(\bar v, \bar u)(x)=\left\{ \begin{array}{ll}
         (v_-,u_-)  \quad\mbox{if $ x <0$},\\
        (v_+,u_+) \quad \mbox{if $x >0$}.\end{array} \right.
\eeq
Here, the two end states $(v_-,u_-)$ and $(v_+,u_+)$ are given constants such that the following holds: 
there exists an intermediate (constant) state $(v_m , u_m)$ such that $(v_-,u_-)$ connects with $(v_m , u_m)$ by the 1-Hugoniot curve, and $(v_m , u_m)$ connects with $(v_+,u_+)$ by the 2-Hugoniot curve, that is, those satisfy the two Rankine-Hugoniot conditions and Lax entropy conditions, respectively:
\begin{align}
\begin{aligned}\label{end-con} 
&\left\{ \begin{array}{ll}
       -\sigma_1 (v_m-v_-) - (u_m -u_-) =0,\\
       -\sigma_1 (u_m-u_-) +p(v_m)-p(v_-)=0, \end{array} \right. \\
& \mbox{where} \quad   \sigma_1=- \sqrt{-(p(v_m)-p(v_-))/(v_m - v_-)}, \quad v_->v_m ,~ u_->u_m ;\\   
&\left\{ \begin{array}{ll}
       -\sigma_2 (v_+-v_m) - (u_+-u_m) =0,\\
       -\sigma_2 (u_+-u_m) +p(v_+)-p(v_m)=0, \end{array} \right.  \\
& \mbox{where} \quad  \sigma_2= \sqrt{-(p(v_+)-p(v_m))/(v_+-v_m)},\quad  v_m<v_+,~ u_m>u_+ .      
\end{aligned}
\end{align} 
The Riemann solution $(\bar v,\bar u)(t,x)$, corresponding to the above data, is the self-similar entropy solution for \eqref{Euler} consisting of the 1-shock wave and the 2-shock wave as follows:
\beq\label{shock-0}
(\bar v, \bar u)(t,x)=\left\{ \begin{array}{ll}
         (v_-,u_-)  \quad\mbox{if $ x/t<\sigma_1 $},\\
          (v_m,u_m)  \quad\mbox{if $\sigma_1 <x/t<\sigma_2 $},\\
        (v_+,u_+) \quad \mbox{if $x/t> \sigma_2$}.\end{array} \right.
\eeq

In this paper, we show  the stability and uniqueness of the Riemann solution \eqref{shock-0} in the class of inviscid limits of Navier-Stokes. To achieve our goal, the main step is to get the uniform (in $\nu$) stability for large perturbations of a composite viscous wave related to the Riemann solution. More precisely, we first recall the fact (see Matsumura-Wang \cite{MW}) that the system \eqref{inveq} admits the 1-viscous shock wave $(\tilde v^{\nu}_1,\tilde u^{\nu}_1)(x-\sigma_1 t)$ and the  2-viscous shock wave $(\tilde v^{\nu}_2,\tilde u^{\nu}_2)(x-\sigma_2 t)$ as traveling wave solutions:
\begin{align}
\begin{aligned}\label{shock_0} 
\left\{ \begin{array}{ll}
      -\sigma_1 (\tilde v^{\nu}_1)_x - (\tilde u^{\nu}_1)_x =0,\\
     -\sigma_1 (\tilde u^{\nu}_1)_x+p( \tilde v^{\nu}_1)_x= \nu\Big(\frac{\mu( \tilde v^{\nu}_1)}{ \tilde v^{\nu}_1}  (\tilde u^{\nu}_1)_x \Big)_x \\
     \lim_{x\to-\infty}(\tilde v^{\nu}_1,\tilde u^{\nu}_1)(x-\sigma_1 t)=(v_-, u_-),\quad  \lim_{x\to+\infty}(\tilde v^{\nu}_1,\tilde u^{\nu}_1)(x-\sigma_1 t)=(v_m, u_m), \end{array} \right.\\
\left\{ \begin{array}{ll}
      -\sigma_2 (\tilde v^{\nu}_2)_x - (\tilde u^{\nu}_2)_x =0,\\
      -\sigma_2 (\tilde u^{\nu}_2)_x+p( \tilde v^{\nu}_2)_x= \nu\Big(\frac{\mu( \tilde v^{\nu}_2)}{ \tilde v^{\nu}_2}  (\tilde u^{\nu}_2)_x \Big)_x \\
     \lim_{x\to-\infty}(\tilde v^{\nu}_2,\tilde u^{\nu}_2)(x-\sigma_2 t)=(v_m, u_m),\quad  \lim_{x\to+\infty}(\tilde v^{\nu}_2,\tilde u^{\nu}_2)(x-\sigma_2 t)=(v_+, u_+). \end{array} \right.       
\end{aligned}
\end{align} 
Then, as the viscous version of \eqref{shock-0}, we consider the composite wave  consisting of the two viscous shock waves:
\beq\label{comwave}
(\tiv^\nu, \tilde u^\nu) (t,x) := \Big(\tilde v^{\nu}_1(x-\sigma_1 t)+\tilde v^{\nu}_2(x-\sigma_2 t) -v_m ,\tilde u^{\nu}_1(x-\sigma_1 t)+\tilde u^{\nu}_2(x-\sigma_2 t) -u_m\Big).
\eeq

For the global-in-time existence of solutions to \eqref{inveq}, we introduce the function space:
\begin{align*}
\begin{aligned}
\mathcal{X}_T := \{ (v,u)~&|~ v-\underline v, ~u- \underline u \in C (0,T; H^1(\bbr)),\\
&\qquad\qquad u-\underline u \in L^2 (0,T; H^2(\bbr)),~ 0< v^{-1}\in L^\infty((0,T)\times \bbr) \} ,
\end{aligned}
\end{align*}
where $\underline v$ and $\underline u$ are smooth monotone functions such that
\beq\label{sm-end}
\underline v(x) = v_\pm \quad\mbox{and}\quad \underline u(x) = u_\pm\quad\mbox{for } \pm x \ge 1.
\eeq
Define the relative potential energy as 
$$
Q(v|\underline{v})= Q(v)-Q(\underline{v})-Q'(\underline{v})(v-\underline{v}), 
$$
where $Q'=p$.
Thanks to \cite{KVexistence}, for  any initial value $(v_0,u_0)$ with finite relative energy 
\begin{equation}\label{hyp1}
\displaystyle{\int_\RR\left(\frac{|u_0-\underline{u}|^2}{2}+Q(v_0|\underline{v})\right)\,dx <\infty,}
\end{equation}
and such that 
\begin{equation}\label{hyp2}
\begin{array}{l}
\displaystyle{v_0-\underline v, u_0-\underline u \in H^k(\RR), \qquad  \mathrm{for \ some \ }k\geq 4,}\\[0.3cm]
\displaystyle{0< \underline \kappa_0\nu \leq v_0(x)\leq  \overline \kappa_0/\nu, \qquad  \forall x\in \RR, \qquad \mathrm{for \ some \ constants \ }   \underline \kappa_0,  \overline \kappa_0,}\\[0.3cm]
\displaystyle{\partial_x u_0(x)\leq \frac{v_0^{\alpha-\gamma}}{\nu}, \qquad  \forall x\in \RR,}
\end{array}
\end{equation}
there exists a unique global solution to \eqref{inveq}. Moreover, for any time $T>0$, this solution lies in $\mathcal{X}_T$.

\subsection{Main results}

To estimate the stability and uniqueness of the Riemann solution, we use the relative entropy associated to the entropy of \eqref{Euler} as follows:
For any functions $v_1,u_1,v_2,u_2$,
\beq\label{eta_def}
\eta((v_1,u_1)|(v_2,u_2)) :=\frac{|u_1-u_2|^2}{2} +Q(v_1|v_2),
\eeq
where $Q(v_1|v_2)$ is the relative functional associated with the strictly convex function 
\[
Q(v):=\frac{v^{-\gamma+1}}{\gamma-1},\quad v>0.
\]
However, in the inviscid limit, the first components $v_1$ are limit of Navier-Stokes equations, which can be a measure in $t,x$. 
So, we should extend the definition of the relative entropy to the case of measures defined on $\bbr^+\times \bbr$ as in \cite{KV-unique19}.  We will restrict the definition in the case where we compare a measure $d v$ with a simple function $\bar \bv$ only taking three values  $v_-, v_+$ and  $v_m$ satisfying $v_\pm\ge v_m$ (as the values of \eqref{end-con}).
 Let  $v_a$ denote the  Radon-Nikodym derivative  of $dv$ with respect to the Lebesgue measure and $dv_s$  its singular part, i.e., $dv=v_a \,dt\,dx+dv_s$.
The relative potential energy  is then itself a measure defined as
\beq\label{dQ}
dQ(v|\bar \bv)(t,x) := Q\left(v_a|\bar \bv\right) dt dx +  |Q'(\overline V(t,x))| dv_s (t,x) ,
\eeq
where $\overline V$ is defined everywhere on $\bbr^+\times \bbr$ by
\begin{eqnarray*}
\overline{V}(t,x)
=\left\{ \begin{array}{ll}
        v_- \quad &\mathrm{for  \ \ } (t,x)\in \overline{\{\bar \bv = v_-\}} \quad (=\mbox{the closure of } \{\bar \bv = v_-\}),\\
        v_+ \quad &\mathrm{for  \ \ } (t,x)\in  \overline{\{\bar \bv = v_+\}}, \\
        v_m \quad &\mathrm{for  \ \ } (t,x)\in \left(\overline{\{\bar \bv = v_-\}}\cup \overline{\{\bar \bv = v_+\}}\right)^c,
        \end{array} \right.
\end{eqnarray*}
Note that $|Q'(v_\pm)| \le |Q'(v_m)|$. Also, note that if $v\in L^\infty(\bbr^+;L^\infty(\bbr)+\mathcal{M}(\bbr))$, then   $dQ(v|\bar v)$ is defined in $L^\infty(\bbr^+; L^\infty(\bbr)+\mathcal{M}(\bbr))$, where $\mathcal{M}$ denotes the space of bounded Radon measures.

The main result of this paper is the following. 

\begin{theorem}\label{thm_inviscid}
For each $\nu>0$, consider the system \eqref{inveq}-\eqref{mu-def} with the assumption \eqref{ass-ag}. 
For a given constant $U_*:=(v_*,u_*)\in\bbr^+\times\bbr$, there exists a constant $\eps_0>0$ such that the following holds.\\
Let $U_-:=(v_-,u_-), U_m:=(v_m,u_m), U_+:=(v_+,u_+)\in\bbr^+\times\bbr$ be any constant states such that \eqref{end-con} holds true, and $U_-, U_m, U_+ \in B_{\eps_0}(U_*)$.\\
Then for a given initial datum $(v^0,u^0)$ of \eqref{Euler} satisfying 
\beq\label{basic_ini}
\mathcal{E}_0:=\int_{-\infty}^{\infty} \eta\big((v^0,u^0)| (\bar v, \bar u)\big) dx <\infty,
\eeq
as a perturbation of the Riemann initial datum \eqref{rini},
the following is true.\\
(i) (Well-prepared initial data) There exists a sequence of smooth functions $\{(v^{\nu}_0, u^{\nu}_0)\}_{\nu>0}$ such that 
\begin{align}
\begin{aligned}\label{ini_conv}
&\lim_{\nu\to0} v^{\nu}_0 = v^0,\quad \lim_{\nu\to0} u^{\nu}_0 = u^0\quad \mbox{a.e.},\quad v^{\nu}_0>0,\\
&\lim_{\nu\to 0} \int_\bbr  \left[\frac{1}{2}\left(u^{\nu}_0 +\nu\left(p(v^{\nu}_0)^{\frac{\alpha}{\gamma}}\right)_x -\tilde u^{\nu}(0,x) -\nu\left(p(\tilde v^{\nu}(0,x))^{\frac{\alpha}{\gamma}}\right)_x  \right)^2 +Q(v^{\nu}_0|\tilde v^{\nu}(0,x)) \right] dx  = \mathcal{E}_0 ,
\end{aligned}
\end{align}
where $(\tilde v^{\nu}, \tilde u^{\nu})$ is the composite wave \eqref{comwave} of the two viscous shocks \eqref{shock_0}.\\
(ii) For a given $T>0$, let $\{(v^{\nu}, u^{\nu})\}_{\nu>0}$ be a sequence of solutions in $\mathcal{X}_T$ to \eqref{inveq} with the initial datum $(v^{\nu}_0, u^{\nu}_0)$ as above. 
Then there exist limits $v_{\infty}$ and $u_{\infty}$ such that as $\nu\to0$ (up to a subsequence),
\beq\label{wconv}
v^{\nu}\rightharpoonup  v_{\infty},\quad u^{\nu} \rightharpoonup  u_{\infty}  \quad \mbox{in} ~\mathcal{M}_{\mathrm{loc}}((0,T)\times\bbr) ~\mbox{(space of locally bounded Radon measures)} ,
\eeq
where $v_\infty$ lies in $L^\infty(0,T,L^\infty(\bbr)+\mathcal{M}(\bbr))$, and  $u_\infty$ lies in $L^\infty(0,T,L^\infty(\bbr)+L^2(\bbr))$.\\
In addition, there exist shifts $X_1^{\infty}, X_2^{\infty} \in \mbox{BV}((0,T))$ and constant $C>0$ such that for a.e. $t\in (0,T]$,
\beq\label{limX12}
\sigma_1 t +X_1^\infty(t)\le \frac{\s_1}{2} t <0<\frac{\s_2}{2} t\le \sigma_2 t +X_2^\infty(t),
\eeq
and
\beq\label{uni-est}
\int_{\bbr } \frac{|u_\infty(t,x)-\bar u^{X_1^\infty, X_2^\infty}(t,x)|^2}{2}  dx + \int_{x\in \bbr } d Q(v_\infty | \bar v^{X_1^\infty, X_2^\infty} )(t, x)  \  \le C \mathcal{E}_0,
\eeq
where $(\bar v^{X_1^\infty, X_2^\infty}, \bar u^{X_1^\infty, X_2^\infty})$ denotes the shifted Riemann solution, that is,
\[
(\bar v^{X_1^\infty, X_2^\infty},\bar u^{X_1^\infty, X_2^\infty}) (t,x) =\left\{ \begin{array}{ll}
         (v_-,u_-)  \quad\mbox{if $ x<\sigma_1 t + X_1^\infty(t) $},\\
          (v_m,u_m)  \quad\mbox{if $\sigma_1 t +X_1^\infty(t)  <x<\sigma_2 t +X_2^\infty(t) $},\\
        (v_+,u_+) \quad \mbox{if $x> \sigma_2t +X_2^\infty(t) $}.\end{array} \right.
\]
Moreover, for a.e. $t_0>0$, there exists a positive constant $C(t_0)$ such that
\beq\label{X-control}
|X_1^\infty(t)| +|X_2^\infty(t)| \le C t_0 + C(t_0)\Big( \mathcal{E}_0 + (1+t)\sqrt{\mathcal{E}_0} \Big),\quad\mbox{for a.e. } t\in(0,T).
\eeq
Therefore, the Riemann solution \eqref{shock-0} is stable (up to shifts) and unique in the class of weak inviscid limits of solutions to the Navier-Stokes systems \eqref{inveq}-\eqref{ass-ag}.
\end{theorem}

\begin{remark}\label{rem-main}
1. Since the shifts $X_1^{\infty}, X_2^{\infty}$ are of BV on $(0,T)$, we have
\begin{align*}
\begin{aligned}
& \overline{\{\bar v^{X_1^\infty, X_2^\infty} = v_-\}} = \{ x\le \s_1 t + X_1^\infty(t) \},\\
& \overline{\{\bar v^{X_1^\infty, X_2^\infty} = v_+\}} = \{ x\ge \s_2 t + X_2^\infty(t) \},\\
&\left(\overline{\{\bar v^{X_1^\infty, X_2^\infty} = v_-\}}\cup \overline{\{\bar v^{X_1^\infty, X_2^\infty} = v_+\}}\right)^c  = \{\sigma_1 t +X_1^\infty(t)  <x<\sigma_2 t +X_2^\infty(t)\}.
\end{aligned}
\end{align*}
Thus it follows from \eqref{dQ} that the measure $d Q(v_\infty | \bar v^{X_1^\infty, X_2^\infty})$ in the stability estimate \eqref{uni-est} is written as follows: for the decomposition $dv_\infty=v_a \,dt\,dx+dv_s$,
\[
d Q(v_\infty | \bar v^{X_1^\infty, X_2^\infty}) (t, x) := Q\left(v_a| \bar v^{X_1^\infty, X_2^\infty}\right) dt dx +  |Q'(\overline V(t,x))| dv_s (t,x),
\]
where
\begin{align}
\begin{aligned}\label{barvd}
\overline{V}(t,x)
=\left\{ \begin{array}{ll}
        v_- \quad &\mathrm{for  \ \ }  x\le \s_1 t + X_1^\infty(t),\\
        v_+ \quad &\mathrm{for  \ \ } x\ge \s_2 t + X_2^\infty(t), \\
        v_m \quad &\mathrm{for  \ \ } \sigma_1 t +X_1^\infty(t)  <x<\sigma_2 t +X_2^\infty(t).
        \end{array} \right.
\end{aligned}
\end{align}
2. Theorem \ref{thm_inviscid} provides the stability and uniqueness of the Riemann solution \eqref{shock-0}  in the wide class of weak inviscid limits of solutions to the Navier-Stokes system. \\
Indeed, for the uniqueness, if $\mathcal{E}_0=0$, then \eqref{X-control} implies that for a.e. $t_0>0$,
\[
|X_1^\infty(t)| +|X_2^\infty(t)| \le C t_0,\qquad \mbox{for a.e. } t\in(0,T),
\]
and so,
\[
X_1^\infty(t) =0, \quad X_2^\infty(t) = 0, \qquad \mbox{for a.e. } t\in(0,T).
\]
This together with \eqref{uni-est} implies that for a.e. $t\in (0,T]$,
\[
\int_{\bbr } \frac{|u_\infty(t,x)-\bar u(t,x)|^2}{2}  dx + \int_{\bbr }  Q(v_a(t,x) | \bar v (t,x)) dx = 0 ,
\]
where the singular part $v_s$ of $v_\infty$ vanishes.  Therefore, we have
\[
u_\infty(t,x) = \bar u(t,x),\qquad  v_\infty(t,x) =  \bar v (t,x),\quad \mbox{ a.e. } (t,x)\in (0,T]\times\bbr.
\]
3. By \eqref{uni-est}, the limits $v_{\infty}, u_{\infty}$ satisfy $v_{\infty}\in \bar v + L^{\infty}(0,T; L^\infty(\bbr)+\mathcal{M}(\bbr))$ and $u_{\infty} \in \bar u + L^{\infty}(0,T; L^2(\bbr))$. The control in measure of $v_\infty$ is due to the fact that $Q(v|\overline{v})\geq c_2|v-\overline{v}|$ for $v\geq 3v_-$ (see (\ref{rel_Q}) in Lemma \ref{lem-pro}). Especially,  $v_\infty$ may have some measure concentrated at infinity. This corresponds physically to cavitation and appearance of vacuum.  \\
Note that we do not  need to know whether the weak inviscid limits $(u,v)$ are solutions to the system \eqref{Euler}, nor any a priory regualrity. The stability of the Riemann's problem needs only   that the perturbations  are generated through inviscid limits of the Navier-Stokes equation \eqref{inveq}. This is very different in spirit from results obtained via compensated compactness, see for instance Chen and Perepelitsa \cite{ChenPere}. The compensated compactness  method  shows that a certain limit verifies the equation. But it does not provide any information on the stability of these functions. \\
4. In fact, the smallness of amplitude of shocks is not needed for the proof of Theorem \ref{thm_inviscid}. The constraint is due to Theorem \ref{thm_general}.\\
\end{remark}

The starting point for the proof of Theorem \ref{thm_inviscid} is to derive the uniform (in $\nu$) stability of any large perturbations of the composite wave for \eqref{inveq}. It is equivalent to  obtaining the contraction property of any large perturbations of the composite wave to \eqref{inveq} with a fixed $\nu=1$:
\begin{align}
\begin{aligned}\label{main}
\left\{ \begin{array}{ll}
        v_t - u_x =0,\\
       u_t+p(v)_x = \Big(\frac{\mu(v)}{v} u_x\Big)_x. \end{array} \right.
\end{aligned}
\end{align}
As in \cite{KV-unique19}, we consider the following relative functional $E(\cdot|\cdot)$ to measure the stability:
\begin{align}
\begin{aligned}\label{psedo}
&\mbox{for any functions } v_1,u_1,v_2,u_2,\\ 
&E((v_1,u_1)|(v_2,u_2)) :=\frac{1}{2}\left(u_1 +\Big(p(v_1)^{\frac{\alpha}{\gamma}}\Big)_x -u_2 -\Big(p(v_2)^{\frac{\alpha}{\gamma}}\Big)_x  \right)^2 +Q(v_1|v_2),
\end{aligned}
\end{align}
where the constants $\gamma, \alpha$ are in \eqref{pressure} and \eqref{mu-def}. Since $Q(v_1|v_2)$ is positive definite, so is the functional $E(\cdot|\cdot)$, that is, for any functions $(v_1,u_1)$ and $(v_2,u_2)$ we have $E((v_1,u_1)|(v_2,u_2))\ge 0$, and 
\[
\quad E((v_1,u_1)|(v_2,u_2))= 0~\mbox{a.e.} \quad\Leftrightarrow\quad (v_1,u_1)=(v_2,u_2)~\mbox{a.e.}
\]
The functional $E$ is associated to the BD entropy (see \eqref{introh}).
The following main result provides the uniform stability of any large perturbations of the composite wave \eqref{comwave} with $\nu=1$. 

\begin{theorem}\label{thm_general}
Assume $\gamma>1$ and $\alpha>0$ satisfying $\alpha\le \gamma \le \alpha +1$.
For a given constant $U_*:=(v_*,u_*)\in\bbr^+\times\bbr$, there exists constant $\delta_0\in (0,1/2)$ such that the following holds.\\
Let $U_-:=(v_-,u_-), U_m:=(v_m,u_m), U_+:=(v_+,u_+)\in\bbr^+\times\bbr$ be any constant states such that \eqref{end-con} and $U_-, U_m, U_+ \in B_{\delta_0}(U_*)$.
Let $\eps_1:=|p(v_-)-p(v_m)|$ and $\eps_2:=|p(v_m)-p(v_+)|$. For any $\lambda >0$ with $\eps_1/\lambda, \eps_2/\lambda<\delta_0$ and $\lambda<\delta_0$, there exist a constant $C>0$ and smooth monotone functions $a_1,a_2$ with $a_1(x), a_2(x)\in[1-\lambda,1]$ for all $x\in\bbr$ such that the following holds.\\
Let $\tilde U(t,x):=(\tiv,\tilde u)(t,x)$ be the composite wave \eqref{comwave} with $\nu=1$. \\
Let
\beq\label{ab_weight}
a(t,x):= a_1(x-\sigma_1t) +a_2(x-\sigma_2t) -1.
\eeq
For a given $T>0$, let $U:=(v,h)$ be a solution in $\mathcal{X}_T$ to \eqref{main} with a initial datum $U_0:=\bmat{v_0}\\{u_0}\emat$ satisfying $\int_{-\infty}^{\infty} E(U_0(x)| \tilde U(0,x)) dx<\infty$. 
Then, there exist shift functions $X_1,X_2\in W^{1,1}((0,T))$ with $X_1(0)=X_2(0)=0$ such that 
for the shifted composite wave
\[
\tilde U^{X_1,X_2}(t,x):=\bmat{\tilde v^{X_1,X_2}(t,x) }\\{\tilde u^{X_1,X_2}(t,x) }\emat :=\bmat{\tilde v_1(x-\sigma_1t-X_1(t))+\tilde v_2(x-\sigma_2t-X_2(t))-v_m}\\ {\tilde u_1(x-\sigma_1t-X_1(t))+\tilde u_2(x-\sigma_2t-X_2(t))-u_m}\emat ,
\]
and the shifted weight
\[
a^{X_1,X_2}(t,x):= a_1(x-\sigma_1t-X_1(t)) +a_2(x-\sigma_2t-X_2(t)) -1,
\]
we have the uniform stability:
\begin{align}
\begin{aligned}\label{acont_main}
&\int_\bbr  a^{X_1,X_2}(t,x) E\big(U(t,x)| \tilde U^{X_1,X_2} (t,x) \big) dx \\
&\qquad  +\int_{0}^{T}\int_{-\infty}^{\infty} | \partial_x a^{X_1,X_2} (t,x)| Q\left(v(t,x)|\tilde v^{X_1,X_2}(t,x)\right) dx dt \\
&\qquad + \int_{0}^{T}\int_{-\infty}^{\infty} v^{\gamma-\alpha}(t,x) \Big|\partial_x\big(p(v(t,x)))-p(\tilde v^{X_1,X_2}(t,x))\big)\Big|^2dxdt  \\
&\quad\le  C \int_{\bbr} a(0,x) E \big(U_0(x)| \tilde U(0,x) \big) dx +C ,
\end{aligned}
\end{align}
and 
\beq\label{amsepX12}
X_1(t) \le -\frac{\s_1}{2} t,\quad X_2(t) \ge -\frac{\s_2}{2} t,\quad\forall t>0,
\eeq
in addition, for each $i=1,2$,
\begin{align}
\begin{aligned} \label{aest-shift}
&|\dot X_i(t)|\le C \left[ f(t) + \int_{\bbr} \eta(U_0(x)|\tilde U(0,x)) dx + 1 \right]\quad \mbox{ for \textit{a.e.} }t\in[0,T] ,\\
&\mbox{for some positive function $f$ satisfying}\quad\|f\|_{L^1(0,T)} \le C \int_{-\infty}^{\infty} \eta \big(U_0(x)| \tilde U(0,x) \big) dx.
\end{aligned}
\end{align}
\end{theorem}

\begin{remark}
1. The stability of the viscous shock waves for the Navier-Stokes system is a very important issue in both mathematical and physical viewpoints.
Theorem \ref{thm_general} provides the first result on stability, independent of the size of the  perturbation, for composite wave of two viscous shocks of the compressible Navier-Stokes system. 
To the best of our knowledge, all the previous results on stability of composite wave of viscous shocks (even a single shock) for the Navier-Stokes require smallness conditon of initial perturbations (see for instance \cite{HMa,HLZ,LZ,MZ,MN,VY}).\\
2. Notice that the uniform stability \eqref{acont_main} is not a contraction estimate, contrary to the case of a single shock in \cite{KV-unique19}. This is natural because the composite wave \eqref{comwave} is not a solution to \eqref{main} (due to the nonlinearity of \eqref{main}).
\end{remark}

The rest of the paper is as follows. We first explain the background of our problem in Section \ref{sec:back}. 
Section \ref{sec:idea} provides scenario of proofs of the mains results. In Section \ref{sec:pre}, we present a transformation of the system \eqref{main}, and the statement of Theorem \ref{thm_main} as an equivalent version of Theorem \ref{thm_general}.  In Section \ref{section_theo}, we show that proving Theorem \ref{thm_main} boils down to showing Proposition \ref{prop:main}. The proof of the main Proposition \ref{prop:main} is presented in Section \ref{sec:prop}. 
In Section \ref{sec:abs}, as a special section, we provide useful propositions written in an abstract manner.
Finally, Section \ref{sec:main} is dedicated to the proof of Theorem \ref{thm_inviscid}.

\section{Background and previous results} \label{sec:back}
The method of proof is based on the relative entropy. Consider $U$ the state variable for conservation laws, and $\eta$ an associated strictly convex entropy functional. In the case of the Euler system \eqref{Euler} (or Navier-Stokes \eqref{inveq}), it corresponds to $U=(v,u)$, $\eta(U)= u^2/2+Q(v)$. 
The relative entropy of a state $U$ compared to another state $\tilde{U}$ is defined as 
$$
\eta(U|\tilde{U})=\eta(U)-\eta(\tilde{U})-d \eta(\tilde{U})\cdot(U-\tilde{U}).
$$
Note that the formula is not symmetric in $U$ and $\tilde{U}$. Because of convexity of $\eta$, this quantity provides a pseudo-distance between the state $U$ and $\tilde{U}$ (if $U$ and $\tilde{U}$ stay bounded in a compact set, then it is actually equivalent to $|U-\tilde{U}|^2$). This quantity was used by Dafermos  \cite{Dafermos4} to show the weak/strong uniqueness principle, a stability result for smooth solutions among the large class of merely weak solutions. Especially, it shows that smooth solutions 
$\tilde{U}$ are unique among any weak solutions $U$. Our main idea is to extend this principle to the situation where the solution $\tilde{U}$ has some discontinuities, and $U$ is an inviscid limit of \eqref{inveq} (possibly not solution to \eqref{Euler}). 
\vskip0.3cm
This idea is similar to a program developed in parallel, and initiated in \cite{Vasseur_Book}, to show the stability (and so uniqueness) of small BV solutions of \eqref{Euler} in the large class of bounded weak solutions of \eqref{Euler} verifying a strong trace property. This program is inspired by  the early work of   DiPerna \cite{DiPerna}  who showed the uniqueness of shocks among this family of weak solutions (see also Chen-Frid-Li  \cite{chen1} for the extension to the Riemann problem). 
\vskip0.3cm
Let us first describe a bit the history of this second program which deals directly with the hyperbolic equation as \eqref{Euler} (without considering any approximation as \eqref{inveq}).
When $\tilde{U}$ is a a fixed constant, the relative entropy is an affine perturbation of the entropy $\eta$, and so is an entropy on its own right. Therefore, the relative entropy $\eta(\cdot| \tilde{U})$ is a large family of entropies in the spirit of the Kruzkov entropies for the scalar case $|\cdot-\tilde{u}|$ (see \cite{K1}). The Kruzkov theory generates contraction properties of the weak solutions in $L^1$. Such contraction property cannot be obtained in general from the relative entropy. This is because the $L^2$ norm is incompatible with the Rankine-Huguoniot condition which dictates the displacement of singularities, and is based on the conservation of linear combinations of the conserved quantities. To understand this obstruction, we first considered the case where $\tilde{U}$ is a single shock. In \cite{Leger}, Leger showed that, in the scalar case, the relative entropy generates a contraction in $L^2$ up to a shift. This artificial shift  is very sensitive to the weak perturbation $U$, and is not uniquely defined. Typically, it can be obtained by solving a generalized ODE (Fillippov \cite{Filippov}) depending on the left and right values of the weak solution at a single point.  This is why, in the purely hyperbolic case, the theory needs the notion of strong traces. Unfortunately, this property is known to hold only for the scalar case \cite{Vasseur_trace}, or the isentropic system with $\gamma=3$ \cite{Vasseur_gamma3}.  
In \cite{Serre-Vasseur}, it has been showed   that $L^2$-type contraction for shocks, even up to a shift, is not true for most of the systems, including \eqref{Euler}. However, we showed in \cite{KVARMA,Vasseur-2013} that this contraction property can be recovered by weighting the relative entropy, leading to the theory of $a$-contraction with shifts.   More precisely, it has been shown that for any shocks $\tilde{U}=(U_l,U_r,\sigma)$, there exist weights $a_1, a_2>0$ such that for any weak solution $U$ we can build a shift $t\to X(t)$ such that 
$$
a_1\int_{-\infty}^{X(t)}\eta(U(t,x)|U_l)\,dx+a_2 \int_{X(t)}^{+\infty}  \eta(U(t,x)|U_r)\,dx \qquad \mathrm{ is \ non \ \! increasing \ in \ time. }
$$
In \cite{Krupa-V}, it was showed how to generalize this statement to obtain a stability result  for any solution ${U}$ in the scalar case. To extend the result to the system case where ${U}$ is of BV, two subtle  properties are needed. When considering several waves, we need to associate an artificial shift for each of the shock curves. Note that we have very little control on the artificial shifts $X_i(t)$.  A crucial needed property is  that, even with these artificial shifts, waves which are not supposed to collide (like a 1-shock on the left of a 2-shock) will never do collide. This has been achieved by Krupa in \cite{Krupa2}. 
Note that the large family of waves will generate a large family of weights $a_i$. It is then important to control the strength of variations of $a_i-a_{i-1}$. It can be shown that these variations can be chosen proportionally to the strength of the associated shock wave.   Surprisingly, this problem was first solved in the context of Navier-Stokes in \cite{Kang-V-NS17}. The proof for the hyperbolic case is very different, and will be available  in \cite{future1}. The final result of stability of BV solutions in the class of weak bounded solutions with strong trace can then been proved \cite{future2}.   
\vskip0.3cm
The program to solve the Bianchini-Bressan conjecture consists of following the same strategy, but aiming for a function ${U}$ which is an inviscid limit of \eqref{inveq} instead of being a solution of \eqref{Euler}. 
Note that the results in this context, as Theorem \ref{thm_inviscid}, are stronger than the one obtain directly on the hyperbolic system. Indeed no a priori assumption is needed on the inviscid limit (not even $L^\infty$ bounds or strong trace property). This is because the relative entropy calculus is done at the level of the Navier-Stokes with $\nu>0$ where enough regularity on the solutions ensures that all the computations hold true. But the price is a far higher level of sophistication in the proofs.  
\vskip0.3cm
Even in the case of a single shock, the contraction involves a subtle balance between the hyperbolic structure (forcing toward the singularity), and the parabolic one (fighting against it). An important point  is to ensure that estimates are uniform with respect to $\nu$.  However, we will show that it is enough to consider the case $\nu=1$ (in this paper it corresponds to Theorems \ref{thm_general}, \ref{thm_main}), at the price of considering general large perturbations. For this reason,  Theorem \ref{thm_general} (and Theorem \ref{thm_main}) is a far stronger result than a standard stability result, since it does not assume any smallness on the initial perturbation. This shows the strength of replacing the notion of stability, by the notion of contraction (without smallness on the initial value). The scaling argument is as follows. Consider $\tilde{U}^\nu$ a traveling wave (viscous shock) of \eqref{inveq}.  Assume that we want to show that for any solution $U^\nu$ of \eqref{inveq}, 
we have a contraction up to a weight function $a_\nu$ and a shift $X_\nu$:
$$
\int_\RR a_\nu(x-X_\nu(t)) \eta(U^\nu(t,x)|\tilde{U}^\nu(x-X_\nu(t)))\,dx \qquad \mathrm{ is \ non \ \! increasing \ in \ time. }
$$
Then, the function $U(t,x)=U^\nu(\nu t,\nu x)$ is a solution to \eqref{inveq} with $\nu=1$, and $\tilde{U}(x)=\tilde{U}^\nu(\nu x)$ is a corresponding traveling wave. Therefore,  using the change of variable in $x$, it is equivalent to showing that up to a weight function ($a(x)=a_\nu(\nu x)$) and a shift ($X(t)= X_\nu(\nu t)/\nu$), we have
$$
\int_\RR a(x-X(t)) \eta(U(t,x)|\tilde{U}(x-X(t)))\,dx \qquad \mathrm{ is \ non \ \! increasing \ in \ time. }
$$
However, even if the initial perturbation for the $\nu$ problem is small, let say
$$
\int_\RR  \eta(U^\nu(0,x)|\tilde{U}^\nu(x))\,dx=\delta,
$$
The associated initial perturbation for the rescaled problem  is very big as
$$
\int_\RR \eta(U(0,x)|\tilde{U}(x))\,dx=\frac{\delta}{\nu}.
$$
Especially, every method based on linearization will fail. 
\vskip0.3cm
 In the context of viscous models, the method was first introduced for the viscous scalar case (without weight) in \cite{Kang-V-1} (improved in \cite{Kang19}), and for the multi-dimensional scalar case in \cite{KVW}.  In the case of Navier-Stokes, the $a$-contraction with shift  for large initial perturbation was proved in \cite{Kang-V-NS17}. The method was also applied in \cite{CKKV} to the Keller-Segel-type model. It provided the key tool to show the global-in-time existence of solutions for non-homogenous boundary data \cite{CKV}. 
 To obtain the stability of shocks of \eqref{Euler} in the family of inviscid limits of \eqref{inveq}, we need to pass into the limit $\nu$ goes to 0. This step, performed in \cite{KV-unique19} is also delicate due to the lack of uniform bound both on the solutions $U^\nu$, and on the shifts $X_\nu$. Especially, nothing prevents cavitations, which corresponds to concentration in measure of $v$. It is remarkable that the stability result can handle even this effect. 
This paper is dedicated to  the last crucial property needed before considering a large family of waves. We show the $a$-contraction property with shifts for the Riemann problem consisting of two shocks. The important point  is to show that we can construct two shifts (one for each shock) which will never collide. This result is the  counter part of \cite{Krupa2} for Navier-Stokes. As always in this program, the philosophy is similar to  the hyperbolic case, but the techniques and the results are very different and far more technical. 

\section{An overview of the proof} \label{sec:idea}

We here describe the main steps of the proof. 
\vskip0.3cm
\noindent {\bf The Stability for $\nu=1$: Theorem \ref{thm_general}.}
As explained in the previous section, the main results of this paper boil down to the proof of stability of a composite wave, consisting of the superposition of  two viscous  shocks waves,  to the Navier-Stokes equation UNIFORMLY with  respect to the strength of the viscosity. A mentioned, this is equivalent to the stability for the case $\nu=1$, if we consider  LARGE perturbations (Theorem \ref{thm_general}). One difficulty due to considering a composite wave, is that at the level of Navier-Stokes, superpositions of exact shock waves are  not  exact solutions to Navier-Stokes (because of the viscosity term, the waves should interact). This explains the extra constant term on the right-hand side of \eqref{acont_main}. However, after rescaling to obtain the result of small viscosity $\nu$, this term becomes $C\nu$ and so converges to 0 in the inviscid limit. This is consistent with the fact that there are no  interactions of far away waves vanish for the inviscid equation. The proof of the theorem can be split into several steps.

\vskip0.1cm
{\it Step one: Introducing  a new velocity variable: Section \ref{sec:pre}.}  The  growth of the perturbation is partly due to  hyperbolic terms (flux functionals). Thanks to the relative entropy method, the linear fluxes are easier to handle (the relative functional of linear quantity vanishes). Therefore, the main hyperbolic quantities to control are the pressure terms depending only on the specific volume $v$.  At the core of the method, we are using a generalized Poincar\'e inequality Proposition \ref{prop:W}, first proved in \cite{Kang-V-NS17}. The Navier-Stokes system can be seen as a degenerate parabolic system. But the diffusion is in the other variable, the velocity variable $u$. Bresch and Desjardins (see \cite{BDL, BD_06}) showed that compressible Navier-Stokes systems have a natural perturbed velocity quantity associated to the viscosity:
$$
h^\nu=u^\nu+\nu \left(p(v^\nu)^{\frac{\alpha}{\gamma}}\right)_x.
$$
Remarkably, the system in the variables $(v^\nu, h^\nu)$ exhibits a diffusion in the $v$ variable (the Smoluchowski equation), rather than in the velocity variable. 
For this reason, we are working with the natural relative entropy of this system, which corresponds to the usual relative entropy of the associated p-system in the $U^\nu_h=(v^\nu, h^\nu)$ variable:
$$
\eta(U^\nu_h|\tilde{U}^\nu_h)= E_\nu(U^\nu|\tilde{U}^\nu).
$$
To simplify the notation,  we denote now $U=(v,h)$, since we work on the new system, and only with $\nu=1$. The associated shifted composite wave is then denoted by
$\tilde{U}^{X_1,X_2}$. It is the superposition, in the $(v,h)$ variables, of the two shocks $\tilde{U}_i$, each subjected to an artificial shift  $X_i(t)$ to be determined. 

\vskip0.1cm
{\it Step 2: Evolution of the relative entropy: Lemma \ref{lem-rel}.}
The evolution of the relative entropy, modulated by a weight $a$ which has also to be determined, can be roughly represented as 
\begin{eqnarray*}
&& \frac{d}{dt}\int_{-\infty}^{\infty} a(x) \eta\big(U(t,x)|\tilde{U}^{X_1,X_2} (x)\big) dx\\
 &&\qquad=\sum_{i=1}^2\dot X_i(t) Y_i(U(t)) +\mathcal{J}^{bad}(U(t))- \mathcal{J}^{good}(U(t)).
\end{eqnarray*}
The functional  $\mathcal{J}^{good}(U)$ is non-negative  (good term) and can be split into three terms:  
$$
\mathcal{J}^{good}(U)=\mathcal{J}^{good}_1(U)+\mathcal{G}_2(U)+\mathcal{D}(U),
$$
where only $\mathcal{J}^{good}_1(U)$ depends on $h$ (and actually does not depend on $v$). The term $\mathcal{D}(U)$ corresponds to the diffusive term (which depends on $v$ only, thanks to the transformation of the system).

\vskip0.1cm
{\it Step 3: Construction of the shifts and the weight function $a$.}
The shifts $X_i(t)$ together with choice of weight function $a$ produce  the terms $\dot{X}_i(t) Y_i(U)$.  The key idea of the technique is to take advantage of these terms, when the $Y_i(U(t))$  are not too small, by compensating all the other terms via the choice of the velocity of the shift (see (\ref{X-def})). Specifically, we algebraically ensure  that the contraction holds as long as one of the conditions $(-1)^{i-1}Y_i(U(t))\geq\eps_i^2$ holds, while ensuring that the two shifts  keep the two shocks waves apart. 
This last property is crucial to avoid unnatural collisions between the two waves, and is due to Proposition \ref{prop:sm}. 
 The rest of the analysis is to ensure that when both $(-1)^{i-1}Y_i(U(t))\le\eps_i^2$ hold, the uniform stability still holds. 

The conditions  $(-1)^{i-1}Y_i(U(t))\le\eps_i^2$ ensure  smallness conditions that we want to fully exploit. This is where the non-homogeneity of the semi-norm is crucial. In the  case where the function $a$ is constant,
$Y_i(U)$ are linear functional in $U$. The smallness of linear part of $Y_i(U)$ gives only that a certain weighted mean value of $U$ is almost null. However, when $a$ has the right monotonicity, $Y_i(U)$ becomes convex. The condition $(-1)^{i-1}Y_i(U(t))\le\eps_i^2$ implies, for this fixed time $t$, 
 a control in $L^2$ for moderate values of $v$, and in  $L^1$ for big values of $v$,  in the two layer regions ($|\xi-X_i(t)|\lesssim 1/\eps_i$). 

The problem now looks, at first glance, as a typical problem of stability with a smallness condition.
There are, however, three major difficulties:  The bad term $\mathcal{J}^{bad}(U)$ has some terms depending on the variable $h$ for which we do not have diffusion, we have some smallness  in $v$, only for a very weak norm, and only localized in the layer regions. More importantly, the smallness is measured with respect to the smallness of the shocks. It basically says that, considering only the moderate values of $v$:
 the perturbation is not bigger than $\eps/\lambda$ (which is still very big with respect to the size of the biggest shock  $\eps$). Actually, as we will see later, it is not possible to consider only the linearized problem: Third order terms appear in the  expansion using the smallness condition (the energy method involving the linearization would have only second order term in $\eps$).
 
In the argument, for the values of $t$ satisfying both $(-1)^{i-1}Y_i(U(t))\le\eps_i^2$,  we construct the shifts as  solutions to the ODEs:
\[
\dot X_i (t) =
\left\{ \begin{array}{ll}
      Y_i(U(t,\cdot+X_i(t)))/\eps^4,\quad &\mbox{if} ~ 0\le  (-1)^{i-1} Y_i(U)\le \eps_i^2, \\
        \frac{(-1)^{i-1}\s_i }{2 \eps_i^2} Y_i(U(t,\cdot+X_i(t))),\quad &\mbox{if}  ~ -\eps_i^2  \le  (-1)^{i-1} Y_i(U)\le 0 ,\\
       -\frac{1}{2}\s_i , \quad &\mbox{if}  ~ (-1)^{i-1} Y_i(U) \le -\eps_i^2 ,
       \end{array} \right.
\]       
 From this point, we forget that $U=U(t,\xi)$ is a solution to the equation and that  $X_i(t)$ are the shifts. That is, we leave out the  $X_i(t)$ and the $t$-variable of $U$. Then we show that for any function $U$ satisfying both
  $(-1)^{i-1}Y_i(U(t))\le\eps_i^2$ for $i=1,2$, we have 
\begin{align*}
&\sum_{i=1}^2\bigg( -\frac{1}{\eps_i^4}|Y_i(U)|^2{\mathbf 1}_{\{0\le  (-1)^{i-1} Y_i(U) \le \eps_1^2\}} + \frac{(-1)^{i-1}\s_i }{2 \eps_i^2}|Y_i(U)|^2{\mathbf 1}_{\{-\eps_i^2\le  (-1)^{i-1} Y_i(U) \le 0\}} \\
&\quad\qquad -\frac{\s_i }{2}Y_i(U){\mathbf 1}_{\{ (-1)^{i-1} Y_i(U) \le -\eps_i^2\}}  \bigg)
+\mathcal{J}^{bad}(U)-\mathcal{J}^{good}(U)\\
&\quad \le  C \left[  g_1(t) \int_\bbr  \eta(U|\tilde U) dx +g_2(t)\right] \mathbf{1}_{t\ge t_0} + C \mathbf{1}_{t\le t_0}  ,
\end{align*}
where $g_1, g_2$ are some integrable functions.
This is the main Proposition \ref{prop:main} (actually, the proposition is slightly stronger to ensure the control of the shift). This implies clearly the uniform estimate as desired. 
From now on, we are focusing on the proof of this statement.

\vskip0.1cm
{\it Step 4: Maximization in $h$ for $v$ fixed.} We recall that the new system is parabolic only in the variable $v$. 
Therefore, we need to get rid of the dependence on the $h$ variable from the bad parts $\mathcal{J}^{bad}(U)$. 
It is done in two different ways depending of the value of $p(v)-p(\tilde{v})$ with respect to a threshold $\delta_1$ to be determined (and depending on the Poincar\'e inequality).  When  $p(v)-p(\tilde{v})\geq \delta_1$, the bad terms involving $h$ can  be controlled using additional information from the unconditional estimates $(-1)^{i-1}Y_i(U(t))\le\eps_i^2$. However, when  $p(v)-p(\tilde{v})\leq \delta_1$, the idea is to maximize the bad term with respect to $h$ for $v$ fixed:
$$
\mathcal{B}(v)=\sup_{h} \left(\mathcal{J}^{bad}(v,h)-\mathcal{J}^{good}_1(h)\right).
$$
We then have an inequality depending only on $v$ and $\partial_x v$ (through $\mathcal{D}(U)$) for which we can apply the generalized Poincar\'e inequality.

\vskip0.1cm
{\it Step 5: Expansion in $\eps_i$.}  Although we have no control on the supremum of $|p(v)-p(\tilde{v})|$, we can control independently the contribution of the values $|p(v)-p(\tilde{v})|\geq \delta_1$ in Proposition \ref{prop_out} (for the same $\delta_1$ related to the maximization process above. The coefficient $\delta_1$ can be chosen very small, but  independent of $\eps_i$ and of $\eps_i/\lambda$). The last step is to perform an expansion in the size of the shocks $\eps_i$, uniformly in $v$ (but for a fixed small value of $\delta$). The expansion is done for each shock wave separately.  The generalized nonlinear Poincar\'e inequality, Proposition \ref{prop:W} concludes the proof.

\vskip0.5cm
\noindent{\bf The inviscid limit:  Theorem \ref{thm_inviscid}.}  

We  have now a stability result uniform with respect to the viscosity. It is natural to expect a stability result on the corresponding inviscid limit. The result, however, is not immediate. Several difficulties have to be overcome. First, due to the BD representation as above, the stability result for $\nu$ fixed is on the quantities:
$$
U^\nu_h=(v^\nu, h^\nu),\quad h^\nu=u^\nu +\nu\left(p(v^\nu)^{\frac{\alpha}{\gamma}}\right)_x.
$$ 
This is the reason we need a compatibility condition on the family of initial values $U_0^\nu$. This also leads to a very weak convergence (in measure in $(t,x)$ only). The next difficulty is that for small values of $v$,  the relative entropy control only the $L^1$ norm of $Q(v)=1/v^{\gamma-1}$. Therefore the pressure $p(v)=1/v^\gamma$ cannot be controlled at all. Therefore, we do not control the time derivative of $u$ in any distributional sense in $x$.  We need to show that the shifts converge to $\sigma_i t$ when the perturbation converges to 0. This can be obtained, thanks to the convergence of $v$ in $C^0(\bbr^+, W^{-s,1}_{\mathrm{loc}}(\bbr))$. It is interesting to note that the continuity (in time) of $v$ is enough. We do not obtain any such control on $u$ (nor $h$).

\section{Reformulation for Theorem \ref{thm_general}} \label{sec:pre}
\setcounter{equation}{0}

Following \cite{Kang-V-NS17}, the first step consists of making a change of  variables to work on an equivalent system where the diffusion  is in the  $v$ equation (instead of $u$).  This is important, since the nonlinearity of the hyperbolic term of \eqref{main} are in $v$ only (through the pressure).\\
First of all, since the strength of the coefficient $b$ in $\mu(v)$ does not affect our analysis, as in \cite{Kang-V-NS17}, we set $b=\gamma$ (for simplicity). Then, any solution $(v,u)$ to \eqref{inveq}, we consider 
\beq\label{introh}
h:=u + \Big(p(v)^{\frac{\alpha}{\gamma}}\Big)_x.
\eeq
This quantity is the modulated velocity associated to the BD entropy (see Bresch and Desjardins \cite{BD_cmp03}).   
Let $\beta:=\gamma-\alpha$. The new unknown $(v,h)$ is then a solution to the system:
\begin{align}
\begin{aligned}\label{NS_1}
\left\{ \begin{array}{ll}
       v_t - h_x = -\Big( v^\beta p(v)_x\Big)_{x}\\
       h_t+p(v)_x =0. \end{array} \right.
\end{aligned}
\end{align}
Then, the two viscous traveling waves in the variables $(v,h)$ associated to \eqref{shock_0} are as follows:
\begin{align}
\begin{aligned}\label{small_shock1} 
\left\{ \begin{array}{ll}
       -\sigma_i \partial_x\tilde v_i(x-\s_i t) - \partial_x\tilde h_i(x-\s_i t) =-\partial_x\Big( {\tilde v_i(x-\s_i t)}^\beta \partial_x p(\tilde v_i(x-\s_i t)) \Big)\\
       -\sigma_i \partial_x\tilde h_i (x-\s_i t)+ \partial_x p( \tilde v_i(x-\s_i t))=0, \end{array} \right.
\end{aligned}
\end{align} 
together with
\begin{align*}
\begin{aligned}
&\lim_{x\to -\infty}(\tilde v_1, \tilde h_1)(x-\s_1 t)=(v_-, u_-),\qquad \lim_{x\to +\infty}(\tilde v_1, \tilde h_1)(x-\s_1 t)=(v_m, u_m),\\
&\lim_{x\to -\infty}(\tilde v_2, \tilde h_2)(x-\s_2 t)=(v_m, u_m),\qquad \lim_{x\to +\infty}(\tilde v_2, \tilde h_2)(x-\s_2 t)=(v_+, u_+).
\end{aligned}
\end{align*} 

Note from the space $\mathcal{X}_T$ that the global solutions to \eqref{NS_1} are in the following  function space: (see Remark \ref{rem-ht})
\begin{align}
\begin{aligned}\label{sol-HT}
       \mathcal{H}_T := \{ (v,h)~ &|~ v-\underline v \in C (0,T; H^1(\bbr)), \\
       &\qquad\qquad ~h-\underline u \in C (0,T; L^2(\bbr)),~ 0< v^{-1}\in L^\infty((0,T)\times \bbr) \} ,
\end{aligned}
\end{align}
where $\underline v$ and $\underline u$ are as in \eqref{sm-end}.

Theorem \ref{thm_general} is an immediate consequence of the following theorem. 

\begin{theorem}\label{thm_main}
Assume $\gamma>1$ and $\alpha>0$ satisfying $\alpha\le \gamma \le \alpha +1$.
For a given constant $U_*:=(v_*,u_*)\in\bbr^+\times\bbr$, there exists constant $\delta_0\in (0,1/2)$ such that the following holds.\\
Let $U_-:=(v_-,u_-), U_m:=(v_m,u_m), U_+:=(v_+,u_+)\in\bbr^+\times\bbr$ be any constant states such that \eqref{end-con} and $U_-, U_m, U_+ \in B_{\delta_0}(U_*)$.
Let $\eps_1:=|p(v_-)-p(v_m)|$ and $\eps_1:=|p(v_m)-p(v_+)|$. For any $\lambda >0$ with $\eps_1/\lambda, \eps_2/\lambda<\delta_0$ and $\lambda<\delta_0$, there exist a constant $C>0$ and smooth monotone functions $a_1,a_2$ with $a_1(x), a_2(x)\in[1-\lambda,1]$ for all $x\in\bbr$ such that the following holds.\\
Let $\tilde U$ be the composite wave connecting  $U_-$ and $U_+$ as:
\beq\label{b_wave}
\tilde U(t,x):=\bmat{\tilde v(t,x) }\\{\tilde h(t,x) }\emat :=\bmat{\tilde v_1(x-\sigma_1t)+\tilde v_2(x-\sigma_2t)-v_m}\\ {\tilde h_1(x-\sigma_1t)+\tilde h_2(x-\sigma_2t)-u_m}\emat .
\eeq
Let 
\beq\label{b_weight}
a(t,x):= a_1(x-\sigma_1t) +a_2(x-\sigma_2t) -1.
\eeq
For a given $T>0$, let $U:=(v,h)$ be a solution in $\mathcal{H}_T$ to \eqref{NS_1} with a initial datum $U_0:=\bmat{v_0}\\{u_0}\emat$ satisfying $\int_{-\infty}^{\infty} \eta(U_0(x)| \tilde U(0,x)) dx<\infty$. 
Then, there exist shift functions $X_1,X_2\in W^{1,1}((0,T))$ with $X_1(0)=X_2(0)=0$ such that 
for the shifted composite wave
\[
\tilde U^{X_1,X_2}(t,x):=\bmat{\tilde v^{X_1,X_2}(t,x) }\\{\tilde h^{X_1,X_2}(t,x) }\emat :=\bmat{\tilde v_1(x-\sigma_1t-X_1(t))+\tilde v_2(x-\sigma_2t-X_2(t))-v_m}\\ {\tilde h_1(x-\sigma_1t-X_1(t))+\tilde h_2(x-\sigma_2t-X_2(t))-u_m}\emat ,
\]
and the shifted weight
\[
a^{X_1,X_2}(t,x):= a_1(x-\sigma_1t-X_1(t)) +a_2(x-\sigma_2t-X_2(t)) -1,
\]
we have the uniform stability:
\begin{align}
\begin{aligned}\label{cont_main2}
&\int_\bbr  a^{X_1,X_2}(t,x) \eta\big(U(t,x)| \tilde U^{X_1,X_2} (t,x) \big) dx \\
&\qquad  +\int_{0}^{T}\int_{-\infty}^{\infty} | \partial_x a^{X_1,X_2} (t,x)| Q\left(v(t,x)|\tilde v^{X_1,X_2}(t,x)\right) dx dt \\
&\qquad + \int_{0}^{T}\int_{-\infty}^{\infty} v^{\gamma-\alpha}(t,x) \Big|\partial_x\big(p(v(t,x)))-p(\tilde v^{X_1,X_2}(t,x))\big)\Big|^2dxdt  \\
&\quad\le  C \int_{\bbr} a(0,x) \eta \big(U_0(x)| \tilde U(0,x) \big) dx +C ,
\end{aligned}
\end{align}
and 
\beq\label{msepX12}
X_1(t) \le -\frac{\s_1}{2} t,\quad X_2(t) \ge -\frac{\s_2}{2} t,\quad\forall t>0,
\eeq
in addition, for each $i=1,2$,
\begin{align}
\begin{aligned} \label{est-shift1}
&|\dot X_i(t)|\le C \left[ f(t) + \int_{\bbr} \eta(U_0(x)|\tilde U(0,x)) dx + 1 \right]\quad \mbox{ for \textit{a.e.} }t\in[0,T] ,\\
&\mbox{for some positive function $f$ satisfying}\quad\|f\|_{L^1(0,T)} \le C \int_{-\infty}^{\infty} \eta \big(U_0(x)| \tilde U(0,x) \big) dx.
\end{aligned}
\end{align}
\end{theorem}

\begin{remark}\label{rem-ht}
1. Theorem \ref{thm_main} provides the uniform stability for composite waves with suitably small amplitude parametrized by $|p(v_-)-p(v_m)|=\eps_1$ and $|p(v_m)-p(v_+)|=\eps_2$. This smallness together with \eqref{end-con} implies that 
\beq\label{uvsm}
|v_--v_m|=\mathcal{O}(\eps_1),\quad |u_--u_m|=\mathcal{O}(\eps_1),\quad |v_+ -v_m|=\mathcal{O}(\eps_2),\quad |u_+-u_m|=\mathcal{O}(\eps_2).
\eeq
2. If we consider the solution $(v,u)\in \mathcal{X}_T$ to \eqref{main}. Then, \eqref{NS_1} admits the solution $(v,h)$ in $\mathcal{H}_T$. Indeed,
since $v_t=u_x \in L^2(0,T; H^1(\bbr))$ by $\eqref{main}_1$, we have $v-\underline v \in C (0,T; H^1(\bbr))$. 
To show $h-\underline u \in C (0,T; L^2(\bbr))$, we first find that for $(v,u)\in \mathcal{X}_T$,
\[
h-\underline u=u-\underline u + \frac{\alpha}{\gamma} p(v)^{\frac{\alpha}{\gamma}-1} v_x \in L^\infty (0,T; L^2(\bbr)).
\]
Moreover, together with the fact that $v\in L^\infty((0,T)\times \bbr)$ by Sobolev embedding, we find that
\begin{align*}
\begin{aligned} 
& u_t = -p'(v) v_x + \frac{d}{dv}\Big(\frac{\mu(v)}{v}\Big)v_x u_x + \frac{\mu(v)}{v} u_{xx} \in L^2 (0,T; L^2(\bbr)),\\
&\Big(p(v)^{\frac{\alpha}{\gamma}-1} v_x \Big)_t = (\frac{\alpha}{\gamma}-1)p(v)^{\frac{\alpha}{\gamma}-2} v_t v_x + p(v)^{\frac{\alpha}{\gamma}-1} v_{xt}  \in L^2 (0,T; L^2(\bbr)),
\end{aligned}
\end{align*}
which implies $h_t \in L^2 (0,T; L^2(\bbr))$, and therefore $h-\underline u \in C (0,T; L^2(\bbr))$.
\end{remark}

$\bullet$ {\bf Notation:} In what follows, $C$ denotes a positive constant which may change from line to line, but which is independent of $\eps_1, \eps_2$ (the sizes of shocks) and $\lambda$ (the total variation of the weights $a_i$).


\vspace{1cm}

\section{Proof of Theorem \ref{thm_main}}\label{section_theo}
\setcounter{equation}{0}

\subsection{Properties of small shock waves}
In this subsection, we  present  useful properties of the $i$-shock waves $(\tilde v_i,\tilde h_i)$ with small amplitude $\eps_i$. In the sequel, we assume that the 1-shock (resp. 2-shock) satisfy $\tilde v_1(0)=\frac{v_-+v_m}{2}$ (resp. $\tilde v_2(0)=\frac{v_m+v_+}{2}$) without loss of generality (by the translation invarience). 

\begin{lemma}\label{lem:vp}
For a given constant $U_*:=(v_*,u_*)\in\bbr^+\times\bbr$, there exist positive constants $\eps_0, C, C_1, C_2$ such that the following holds.\\
Let $U_-:=(v_-,u_-), U_m:=(v_m,u_m), U_+:=(v_+,u_+)\in\bbr^+\times\bbr$ be any constant such that \eqref{end-con} and $U_-, U_m, U_+ \in B_{\eps_0}(U_*)$, and 
$|p(v_-)-p(v_m)|=:\eps_1<\eps_0$ and $|p(v_m)-p(v_+)|=: \eps_2<\eps_0$.
Let $(\tilde v_1,\tilde h_1)$ and $(\tilde v_2,\tilde h_2)$ be the 1- and 2-shocks respectively connecting from $U_-$ to $U_m$, and from $U_m$ to $U_+$ such that $\tilde v_1(0)=\frac{v_-+v_m}{2}$ and $\tilde v_2(0)=\frac{v_m+v_+}{2}$. Then, the following estimates hold.
\begin{align}
\label{tv1}
&C^{-1}\eps_1 e^{-C_1 \eps_1 |x-\sigma_1 t|} \le \tilde v_1(x-\sigma_1 t) -v_m \le C\eps_1 e^{-C_2 \eps_1 |x-\sigma_1 t|},\quad x\ge \s_1 t,\\
\label{tdv1}
&-C^{-1}\eps_1^2 e^{-C_1 \eps_1 |x-\sigma_1 t|} \le \partial_x\tilde v_1(x-\sigma_1 t) \le -C\eps_1^2 e^{-C_2 \eps_1 |x-\sigma_1 t|},\quad x\in\bbr, ~ t>0,
\end{align}
and
\begin{align}
\label{tv2}
&C^{-1}\eps_2 e^{-C_1 \eps_2 |x-\sigma_2 t|} \le \tilde v_2(x-\sigma_2 t) -v_m \le C\eps_2 e^{-C_2 \eps_2 |x-\sigma_2 t|},\quad x\le \s_2 t,\\
\label{tdv2}
&C^{-1}\eps_2^2 e^{-C_1 \eps_2 |x-\sigma_2 t|} \le \partial_x\tilde v_2(x-\sigma_2 t) \le C\eps_2^2 e^{-C_2 \eps_2 |x-\sigma_2 t|},\quad x\in\bbr, ~ t>0.
\end{align}
As a consequence, for each $i=1, 2$,
\beq\label{lower-v}
\inf_{\left[-\eps_i^{-1}, \,\eps_i^{-1}\right]}| \partial_x\tilde v_i|\ge C\eps_i^2.
\eeq
In addition, for each $i=1, 2$,
\beq\label{twodv}
| (\tilde v_i)_{xx}(x-\sigma_i t)| \le C\eps_i | (\tilde v_i)_x (x-\sigma_i t)|,  \quad x\in\bbr, ~ t>0.
\eeq
\end{lemma}
\begin{proof} 
Since $v_*/2<\tilde v_i<2v_*$ with choosing $\eps_0$ small enough, the proofs of \eqref{tv1}-\eqref{lower-v} use the same computations as in the proof of \cite[Lemma 2.1]{Kang-V-NS17}. Therefore, we omit the details.\\
To show \eqref{twodv}, we first observe from \eqref{small_shock1} that
\begin{align*}
&\left( \frac{\s_i^2}{p'(\tiv_i)} +1 \right) \frac{p'(\tiv_i) \pa_x\tiv_i}{\s_i} =
   \sigma_i \partial_x\tilde v_i + \partial_x\tilde h_i  =\partial_x\Big( {\tilde v_i}^\beta \partial_x p(\tilde v_i) \Big) \\
&\quad =\beta \tiv_i^{\beta-1} p'(\tiv_i) |(\tiv_i)_x|^2 + \tiv_i^{\beta} p''(\tiv_i) |(\tiv_i)_x|^2 +  \tiv_i^{\beta} p'(\tiv_i) (\tiv_i)_{xx},
\end{align*}
where note that the above waves are evaluated at $x-\s_i t$.\\
Then, using \eqref{tdv1}, \eqref{tdv2} and $v_*/2 \le \tiv_i\le 2v_*$, we find
\[
| (\tilde v_i)_{xx}| \le C\eps_i^2 | (\tilde v_i)_x | + \left| \frac{\s_i^2}{p'(\tiv_i)} +1 \right| | (\tilde v_i)_x |.
\]
Using Taylor theorem together with \eqref{uvsm} and \eqref{end-con}, we have
\[
\s_1= -\sqrt{-p'(v_-)}+\mathcal{O}(\eps_1), \qquad \s_2= -\sqrt{-p'(v_m)}+\mathcal{O}(\eps_2),
\]
and 
\[
p'(\tilde v_1)^{-1}= p'(v_-)^{-1} +\mathcal{O}(\eps_1),\qquad p'(\tilde v_2)^{-1}= p'(v_m)^{-1} +\mathcal{O}(\eps_2).
\]
In addition, since $|v_--v_*|\le\eps_0$ and $|v_m-v_*|\le\eps_0$, we have
\[
\left| \frac{\s_i^2}{p'(\tiv_i)} +1 \right| \le C \eps_i\quad \mbox{for each } i=1, 2.
\]
Therefore we have \eqref{twodv}.
\end{proof}

\begin{remark}\label{rem:vh}
Notice that Lemma \ref{lem:vp} also holds for $\tilde h_i$, since 
\beq\label{vhre}
C^{-1} |\partial_x\tilde v_i|\le |\partial_x\tilde h_i|\le C|\partial_x\tilde v_i|,
\eeq 
which comes from the fact that $\partial_x\tilde h_i=\frac{p'(\tilde v_i)}{\sigma_i} \partial_x\tilde v_i$ and $\frac{1}{2}\frac{p'(v_*)}{\sigma_*} \le\frac{p'(\tilde v_i)}{\sigma_i}\le 2 \frac{p'(v_*)}{\sigma_*}$ by \eqref{end-con} and \eqref{uvsm} with $\eps_1, \eps_2<\eps_0\ll 1$.
Especially, notice that $\tiv_i$ and $\tih_i$ are monotone by \eqref{tdv1}, \eqref{tdv2} and $\partial_x\tilde h_i=\frac{p'(\tilde v_i)}{\sigma_i} \partial_x\tilde v_i$.
\end{remark}

\subsection{Relative entropy method}
A starting point of our analysis is to use the relative entropy method. The method  is purely nonlinear, and allows to handle rough and large perturbations, and so it has been extensively used in studying the stability of shock (or contact discontinuity) (see \cite{AKV,CV,Kang19,Kang,Kang-V-NS17,KVARMA,Kang-V-1,KVW,KVW3D,Leger,Serre-Vasseur,SV_16,SV_16dcds,Vasseur-2013,VW}).

To use the relative entropy method, we rewrite \eqref{NS_1} into the  viscous hyperbolic system of conservation laws:
\beq\label{system-0}
\partial_t U +\partial_x A(U)= { -\partial_{\xi}\big(v^\beta \partial_{\xi}p(v)\big) \choose 0},
\eeq
where 
\[
U:={v \choose h},\quad A(U):={-h \choose p(v)}.
\]
The system \eqref{system-0} has a convex entropy $\eta(U):=\frac{h^2}{2}+Q(v)$, where $Q(v)=\frac{v^{-\gamma+1}}{\gamma-1}$, i.e., $Q'(v)=-p(v)$.\\
Using the derivative of the entropy as 
\beq\label{nablae}
\nabla\eta(U)={-p(v)\choose h},
\eeq
the above system \eqref{system-0} can be rewritten as
\beq\label{system}
\partial_t U +\partial_x A(U)= \partial_x\Big(M(U) \partial_x\nabla\eta(U) \Big),
\eeq
where $M(U)={v^\beta \quad 0 \choose 0\quad 0}$. \\
Also, for each wave 
\[
\tilde U_i(x-\s_i t):=\bmat{\tilde v_i(x-\sigma_i t)}\\ {\tilde h_i(x-\sigma_i t)}\emat,
\]
the system \eqref{small_shock1} can be rewritten as:
\beq\label{re_shock}
-\s_i \partial_x\tilde U_i + \partial_x A(\tilde U_i)= \partial_x\Big(M(\tilde U_i) \partial_x\nabla\eta(\tilde U_i) \Big).
\eeq
Thus, the composite wave
\[
\tilde U(t,x)= \tilde U_1(x-\s_1 t) + \tilde U_2(x-\s_2 t) - \bmat{v_m}\\ {u_m}\emat 
\]
satisfies
\beq\label{cowave-eq}
\partial_t \tilde U= \sum_{i=1}^2 \left( -\partial_x A(\tilde U_i)+\partial_x\Big(M(\tilde U_i) \partial_x\nabla\eta(\tilde U_i) \Big) \right).
\eeq

We define the relative entropy function by
\[
\eta(U|V)=\eta(U)-\eta(V) -\nabla\eta(V) (U-V),
\]
and the relative flux by
\[
A(U|V)=A(U)-A(V) -\nabla A(V) (U-V).
\] 
Let $G(\cdot;\cdot)$ be the flux of the relative entropy defined by
\[
G(U;V) = G(U)-G(V) -\nabla \eta(V) (A(U)-A(V)),
\]
where $G$ is the entropy flux of $\eta$, i.e., $\partial_{i}  G (U) = \sum_{k=1}^{2}\partial_{k} \eta(U) \partial_{i}  A_{k} (U),\quad 1\le i\le 2$.\\
Then, for our system \eqref{system-0}, we have
\begin{align}
\begin{aligned}\label{relative_e}
&\eta(U|\tilde U)=\frac{|h-\tilde h |^2}{2} + Q(v|\tilde v),\\
& A(U|\tilde U)={0 \choose p(v|\tilde v)},\\
&G(U;\tilde U)=(p(v)-p(\tilde v)) (h-\tilde h),
\end{aligned}
\end{align}
where the relative pressure is defined as
\begin{equation}\label{pressure-relative}
p(v|w)=p(v)-p(w)-p'(w)(v-w).
\end{equation}

We will consider a weighted relative entropy between the solution $U$ of \eqref{system} and the shifted composite wave $\tilde U^{X_1,X_2}$ as follows: for shifts $X_1, X_2$ (to be determined),
\[
a^{X_1,X_2}(t,x)\eta\big(U(t,x)|\tilde U^{X_1,X_2}(t,x) \big) ,
\]
where the wave $\tilde U^{X_1,X_2}$ and the weight $a^{X_1,X_2}$ have the form of 
\beq\label{xwave}
\tilde U^{X_1,X_2}(t,x):=\bmat{\tilde v^{X_1,X_2}(t,x) }\\{\tilde h^{X_1,X_2}(t,x) }\emat :=\bmat{\tilde v_1(x-\sigma_1t-X_1(t))+\tilde v_2(x-\sigma_2t-X_2(t))-v_m}\\ {\tilde h_1(x-\sigma_1t-X_1(t))+\tilde h_2(x-\sigma_2t-X_2(t))-u_m}\emat ,
\eeq
and 
\beq\label{xa}
a^{X_1,X_2}(t,x):= a_1(x-\sigma_1t-X_1(t)) +a_2(x-\sigma_2t-X_2(t)) -1 .
\eeq

In Lemma \ref{lem-rel}, we will derive a quadratic structure on 
$$
\frac{d}{dt}\int_{\bbr} a^{X_1,X_2}(t,x)\eta\big(U(t,x)|\tilde U^{X_1,X_2}(t,x) \big) dx.
$$
For notational simplicity, we will use the following notations: for the waves $\tilde U_1, \tilde U_2$ and the functions $a_1, a_2$, and any shifts $X_1, X_2$, 
\[
\tilde U_i^{X_i}:=\tilde U_i (x -\s_i t - X_i(t)), \quad a_i^{X_i}:=a_i(x -\s_i t - X_i(t))
\]

\begin{lemma}\label{lem-rel}
Let $a_1$ and $a_2$ be any positive smooth bounded functions whose derivative is bounded and integrable. Let $U\in  \mathcal{H}_T$ be a solution to \eqref{system}, and $X_1, X_2$ be any absolutely continuous functions on $[0,T]$. 
Let $\tilde U^{X_1,X_2}$ and $a^{X_1,X_2}$ as in \eqref{xwave} and \eqref{xa}. Then,
\begin{align}
\begin{aligned}\label{ineq-0}
\frac{d}{dt}\int_{\bbr} a^{X_1,X_2}(t,x) \eta(U(t,x)|\tilde U^{X_1,X_2}(t,x) ) dx =\sum_{i=1}^2\big(\dot X_i(t) Y_i(U) \big)+\mathcal{J}^{bad}(U) - \mathcal{J}^{good}(U),
\end{aligned}
\end{align}
where
\begin{align}
\begin{aligned}\label{ybg-first}
&Y_i(U):= - \int_{\bbr} (a_i)_x^{X_i}  \eta(U|\tilde U^{X_1,X_2} ) dx +\int_{\bbr} a^{X_1,X_2} (\tilde U_i)_x^{X_i}  \nabla^2\eta(\tilde U^{X_1,X_2}) (U-\tilde U^{X_1,X_2}) dx ,\\
&\mathcal{J}^{bad}(U):=\sum_{i=1}^2\bigg[  \int_\bbr (a_i)_x^{X_i}  \big(p(v)-p(\tilde v^{X_1,X_2} )\big) \big(h-\tilde h^{X_1,X_2} \big) dx +  \s_i \int_\bbr a^{X_1,X_2} (\tilde v_i)_x^{X_i} p(v|\tilde v^{X_1,X_2}) dx\\
 &\quad\quad\quad -\int_\bbr (a_i)_x^{X_i}  v^\beta \big(p(v)-p(\tilde v^{X_1,X_2} )\big)\partial_x \big(p(v)-p(\tilde v^{X_1,X_2} )\big)  dx\\
&\quad\quad\quad -\int_\bbr (a_i)_x^{X_i}  \big(p(v)-p(\tilde v^{X_1,X_2})\big) (v^\beta - (\tilde v^\beta)^{X_1,X_2} ) \partial_{x} p(\tilde v^{X_1,X_2})  dx \bigg] \\
&\quad\quad\quad -\int_\bbr a^{X_1,X_2}  \partial_x \big(p(v)-p(\tilde v^{X_1,X_2})\big) (v^\beta - (\tilde v^\beta)^{X_1,X_2} ) \partial_{x} p(\tilde v^{X_1,X_2})  dx\\
&\quad\quad\quad + \int_\bbr a^{X_1,X_2} \big(p(v)-p(\tilde v^{X_1,X_2}) \big)  \tilde E_1 dx -  \int_\bbr a^{X_1,X_2} (h-\tilde h^{X_1,X_2})  \tilde E_2 dx ,\\
&\mathcal{J}^{good}(U):= \sum_{i=1}^2\bigg( \frac{\sigma_i}{2}\int_\bbr (a_i)_x^{X_i} \left| h-\tilde h^{X_1,X_2} \right|^2 dx  +\sigma_i  \int_\bbr  (a_i)_x^{X_i}  Q(v|\tilde v^{X_1,X_2})  dx\bigg)\\
&\quad\quad\quad +\int_\bbr a^{X_1,X_2}  v^\beta |\partial_x \big(p(v)-p(\tilde v^{X_1,X_2} )\big)|^2 dx,
\end{aligned}
\end{align}
where 
\begin{align*}
&\tilde E_1:=  \partial_x\big( (\tilde v^\beta)^{X_1,X_2} \partial_x p(\tilde v^{X_1,X_2} )\big) -\partial_x\big( (\tilde v_1^\beta)^{X_1} \partial_x p(\tilde v_1^{X_1} )\big) -\partial_x\big( (\tilde v_2^\beta)^{X_2} \partial_x p(\tilde v_2^{X_2} )\big) ,\\
&\tilde E_2:=\partial_x p(\tilde v^{X_1,X_2}) -\partial_x p(\tilde v_1^{X_1} ) -\partial_x p(\tilde v_2^{X_2} ).
\end{align*}
\end{lemma}

\begin{remark}
In what follows, we will define the weight function $a$ such that $\sigma_\eps a' >0$. Therefore, $-\mathcal{J}^{good}$ consists of three good terms, while $\mathcal{J}^{bad}$ consists of bad terms. 
\end{remark}

\begin{proof}
By the definition of the relative entropy, we first have
\begin{align*}
&\frac{d}{dt}\int_{\bbr} a^{X_1,X_2}(t,x) \eta(U(t,x)|\tilde U^{X_1,X_2}(t,x) ) dx 
= \int_{\bbr} \partial_t  a^{X_1,X_2} \eta(U|\tilde U^{X_1,X_2} ) dx \\
&\qquad\quad+\int_\bbr \!\! a^{X_1,X_2}\bigg[\Big(\nabla\eta(U)-\nabla\eta(\tilde U^{X_1,X_2})\Big) \partial_t U  -\nabla^2\eta(\tilde U^{X_1,X_2}) (U-\tilde U^{X_1,X_2}) \partial_t \tilde U^{X_1,X_2}  \bigg] dx.
\end{align*}
Since
\[
\partial_t a^{X_1,X_2}(t,x) = -\sum_{i=1}^2 (\s_i +\dot{X}_i ) (a_i)_x^{X_i} ,
\]
and it follows from \eqref{cowave-eq} that
\[
\partial_t \tilde U^{X_1,X_2} (t,x)= \sum_{i=1}^2\left( -\dot{X}_i ( \tilde U_i )_x^{X_i}   -\partial_x A(\tilde U_i^{X_i} )+\partial_x\Big(M(\tilde U_i^{X_i} ) \partial_x\nabla\eta(\tilde U_i^{X_i} ) \Big) \right),
\]
we have
\begin{align*}
&\frac{d}{dt}\int_{\bbr} a^{X_1,X_2}(t,x) \eta(U(t,x)|\tilde U^{X_1,X_2}(t,x) ) dx \\
&=\sum_{i=1}^2 (\dot{X}_i Y_i ) -\sum_{i=1}^2 \s_i  \int_{\bbr} (a_i)_x^{X_i} \eta(U|\tilde U^{X_1,X_2} ) dx \\
&\qquad +\int_\bbr \!\! a^{X_1,X_2} \bigg[\Big(\nabla\eta(U)-\nabla\eta(\tilde U^{X_1,X_2})\Big)\!\Big(\!\!\!-\partial_x A(U)+ \partial_x\Big(M(U)\partial_x\nabla\eta(U) \Big) \Big)\\
&\qquad -\nabla^2\eta(\tilde U^{X_1,X_2}) (U-\tilde U^{X_1,X_2}) \underbrace{ \sum_{i=1}^2  \Big( -\partial_x A(\tilde U_i^{X_i})+\partial_x\Big(M(\tilde U_i^{X_i}) \partial_x\nabla\eta(\tilde U_i^{X_i}) \Big)  \Big) }_{=:J} \bigg] dx ,
\end{align*}
where 
\[
Y_i:= - \int_{\bbr} (a_i)_x^{X_i}  \eta(U|\tilde U^{X_1,X_2} ) dx +\int_{\bbr} a^{X_1,X_2} (\tilde U_i)_x^{X_i} \nabla^2\eta(\tilde U^{X_1,X_2}) (U-\tilde U^{X_1,X_2}) dx.
\]
Since
\[
J=-\partial_x A(\tilde U^{X_1,X_2})+\partial_x\Big(M(\tilde U^{X_1,X_2}) \partial_x\nabla\eta(\tilde U^{X_1,X_2}) \Big) +{\tilde E_1\choose \tilde E_2} ,
\]
where
\[
{\tilde E_1\choose \tilde E_2}:= { \partial_x\big( (\tilde v^\beta)^{X_1,X_2} \partial_x p(\tilde v^{X_1,X_2} )\big) -\partial_x\big( (\tilde v_1^\beta)^{X_1} \partial_x p(\tilde v_1^{X_1} )\big) -\partial_x\big( (\tilde v_2^\beta)^{X_2} \partial_x p(\tilde v_2^{X_2} )\big) \choose \partial_x p(\tilde v^{X_1,X_2}) -\partial_x p(\tilde v_1^{X_1} ) -\partial_x p(\tilde v_2^{X_2} ) }.
\]
Using the same computation in \cite[Lemma 2.3]{Kang-V-NS17} (see also \cite[Lemma 4]{Vasseur_Book}), we have
\[
\frac{d}{dt}\int_{\bbr} a^{X_1,X_2}(t,x) \eta(U(t,x)|\tilde U^{X_1,X_2}(t,x) ) dx=\sum_{i=1}^2 (\dot{X}_i Y_i ) +\sum_{i=1}^6 I_i,
\]
where 
\begin{align*}
\begin{aligned}
&I_1:=  -\sum_{i=1}^2 \s_i  \int_{\bbr} (a_i)_x^{X_i} \eta(U|\tilde U^{X_1,X_2} ) d x -\int_\bbr  a^{X_1,X_2} \partial_x G(U;\tilde U^{X_1,X_2}) dx,\\
&I_2:=- \int_\bbr a^{X_1,X_2} \partial_x \nabla\eta(\tilde U^{X_1,X_2}) A(U|\tilde U^{X_1,X_2}) dx ,\\
&I_3:=\int_\bbr a^{X_1,X_2} \Big( \nabla\eta(U)-\nabla\eta(\tilde U^{X_1,X_2})\Big) \partial_x \Big(M(U) \partial_x \big(\nabla\eta(U)-\nabla\eta(\tilde U^{X_1,X_2})\big) \Big)  dx, \\
&I_4:=\int_\bbr a^{X_1,X_2} \Big( \nabla\eta(U)-\nabla\eta(\tilde U^{X_1,X_2})\Big) \partial_x \Big(\big(M(U)-M(\tilde U^{X_1,X_2})\big) \partial_x \nabla\eta(\tilde U^{X_1,X_2}) \Big)  dx, \\
&I_5:=\int_\bbr a^{X_1,X_2}(\nabla\eta)(U|\tilde U^{X_1,X_2})\partial_x \Big(M(\tilde U^{X_1,X_2}) \partial_x \nabla\eta(\tilde U^{X_1,X_2}) \Big)  dx,\\
&I_6:=-\int_\bbr a^{X_1,X_2} \nabla^2\eta(\tilde U^{X_1,X_2}) (U-\tilde U^{X_1,X_2}) {\tilde E_1\choose \tilde E_2} \,dx .
\end{aligned}
\end{align*}
Using $\partial_x a^{X_1,X_2} = (a_1)_x^{X_1}+(a_2)_x^{X_2}$ and \eqref{relative_e} with \eqref{nablae}, we have
\begin{align*}
\begin{aligned}
I_1&=\sum_{i=1}^2\left( -\s_i  \int_{\bbr} (a_i)_x^{X_i} \eta(U|\tilde U^{X_1,X_2} ) d x + \int_\bbr (a_i)_x^{X_i}  \big(p(v)-p(\tilde v^{X_1,X_2} )\big) \big(h-\tilde h^{X_1,X_2} \big) dx\right) ,
\end{aligned}
\end{align*}
and 
\[
I_2=-\sum_{i=1}^2 \int_\bbr a^{X_1,X_2} (\tilde h_i)_x^{X_i} p(v|\tilde v^{X_1,X_2}) dx.
\]
By integration by parts, we have
\begin{align*}
I_3&=\int_\bbr a^{X_1,X_2}  \big(p(v)-p(\tilde v^{X_1,X_2} )\big)\partial_{x}\Big(v^\beta \partial_{x}\big(p(v)-p(\tilde v^{X_1,X_2} )\big) \Big) dx \\
&=-\int_\bbr a^{X_1,X_2}  v^\beta |\partial_x \big(p(v)-p(\tilde v^{X_1,X_2} )\big)|^2 dx \\
&\quad -\sum_{i=1}^2 \int_\bbr (a_i)_x^{X_i}  v^\beta \big(p(v)-p(\tilde v^{X_1,X_2} )\big)\partial_x \big(p(v)-p(\tilde v^{X_1,X_2} )\big)  dx,
\end{align*}
and
\begin{align*}
I_4&= \int_\bbr a^{X_1,X_2} \big(p(v)-p(\tilde v^{X_1,X_2})\big)\partial_{x}\Big((v^\beta - (\tilde v^\beta)^{X_1,X_2} ) \partial_{x} p(\tilde v^{X_1,X_2}) \Big) dx \\
&=-\sum_{i=1}^2 \int_\bbr (a_i)_x^{X_i}  \big(p(v)-p(\tilde v^{X_1,X_2})\big) (v^\beta - (\tilde v^\beta)^{X_1,X_2} ) \partial_{x} p(\tilde v^{X_1,X_2})  dx \\
&\qquad -\int_\bbr a^{X_1,X_2}  \partial_x \big(p(v)-p(\tilde v^{X_1,X_2})\big) (v^\beta - (\tilde v^\beta)^{X_1,X_2} ) \partial_{x} p(\tilde v^{X_1,X_2})  dx.
\end{align*}
It follows from \eqref{re_shock} and \eqref{nablae} that
\begin{align*}
\begin{aligned}
I_5&=\int_\bbr a^{X_1,X_2}(\nabla\eta)(U|\tilde U^{X_1,X_2})
 \bigg[ \sum_{i=1}^2  \Big( -\s_i  (\tilde U_i)_x^{X_i} +\partial_x A(\tilde U_i^{X_i})  \\
&\quad -\partial_x\Big(M(\tilde U_i^{X_i}) \partial_x\nabla\eta(\tilde U_i^{X_i}) \Big)  \Big)+ \partial_x \Big(M(\tilde U^{X_1,X_2}) \partial_x \nabla\eta(\tilde U^{X_1,X_2}) \Big) \bigg] dx \\
&=\sum_{i=1}^2\bigg[ \s_i \int_\bbr a^{X_1,X_2} (\tilde v_i)_x^{X_i} p(v|\tilde v^{X_1,X_2}) dx +\int_\bbr a^{X_1,X_2} (\tilde h_i)_x^{X_i} p(v|\tilde v^{X_1,X_2}) dx\bigg]\\
&\quad + \int_\bbr a^{X_1,X_2} p(v|\tilde v^{X_1,X_2})  \tilde E_1 dx .
\end{aligned}
\end{align*}
Since
\beq\label{2dere}
 \nabla^2\eta(\tilde U^{X_1,X_2}) = \bmat -p'(\tilde v^{X_1,X_2}) &0\\0&1\emat ,
\eeq
we have
\begin{align*}
\begin{aligned}
I_6&=  \int_\bbr a^{X_1,X_2} p'(\tilde v^{X_1,X_2}) (v-\tilde v^{X_1,X_2})  \tilde E_1 dx
-  \int_\bbr a^{X_1,X_2} (h-\tilde h^{X_1,X_2})  \tilde E_2 dx .
\end{aligned}
\end{align*}
Thus we have some cancellation
\begin{align*}
I_2+I_5+I_6 &=\sum_{i=1}^2\bigg[ \s_i \int_\bbr a^{X_1,X_2} (\tilde v_i)_x^{X_i} p(v|\tilde v^{X_1,X_2}) dx \bigg]\\
&\quad + \int_\bbr a^{X_1,X_2} \big(p(v)-p(\tilde v^{X_1,X_2}) \big)  \tilde E_1 dx
-  \int_\bbr a^{X_1,X_2} (h-\tilde h^{X_1,X_2})  \tilde E_2 dx .
\end{align*}
Therefore, we have the desired representation.
\end{proof}

\subsection{Construction of the weight function}\label{subsec:a}
We define the weight function $a$ by
\begin{align}
\begin{aligned} \label{weight-a}
&a (t,x)=a_1(x-\sigma_1t) +a_2(x-\sigma_2t) -1,\\
&a_1(x-\sigma_1t):= 1-\lambda \frac{p(\tilde v_1(x-\sigma_1t))-p(v_-)}{\eps_1},\\
&a_2(x-\sigma_2t) :=1-\lambda \frac{p(\tilde v_2(x-\sigma_2t))-p(v_+)}{\eps_2}.
\end{aligned}
\end{align}
That is,
\[
a (t,x)= 1-\lambda \frac{p(\tilde v_1(x-\sigma_1t))-p(v_-)}{\eps_1} -\lambda \frac{p(\tilde v_2(x-\sigma_2t))-p(v_+)}{\eps_2}.
\]

We briefly present some useful properties on the weight $a$.\\
First of all, $a_1$ decreases from $1$ to $1-\lambda$ on $\bbr$, but $a_2$ increases from $1-\lambda$ to $1$ on $\bbr$. Thus,
the weight function $a$ satisfies $1-\lambda\le a\le 1$, and
\beq\label{massa}
\int_\bbr |(a_i)_x| dx =\lam,\quad\mbox{for } i=1, 2.
\eeq
Note that for $\eps_1 ,\eps_2<\delta_0\ll1$, $p'(v_*/2)\le p'(\tilde v_i)\le p'(2 v_*)<0$ for all $i=1,2$. Thus, for each $i=1,2$, $\sigma_i (a_i)_x >0$, and
\beq\label{der-a}
(a_i)_x =- \frac{\lam}{\eps_i}  p'(\tilde v_i) (\tilde v_i)_x ,
\eeq
we have 
\beq\label{d-weight}
 C^{-1} \frac{\lambda}{\eps_i}|(\tilde v_i)_x| \le |(a_i)_x| \le C \frac{\lambda}{\eps_i}|(\tilde v_i)_x| .
\eeq

\subsection{Maximization in terms of $h-\tilde h$}\label{sec:mini}

In order to estimate the right-hand side of \eqref{ineq-0}, we will use Proposition \ref{prop:main3}, i.e., a sharp estimate with respect to $p(v)-p(\tilde v)$ near $p(\tiv)$. For that, we need to rewrite $\mathcal{J}^{bad}$ on the right-hand side of \eqref{ineq-0} only in terms of $p(v)$ near $p(\tilde v)$, by separating $h-\tilde h$ from the first term of $\mathcal{J}^{bad}$. Therefore, we will rewrite $\mathcal{J}^{bad}$ into the maximized representation in terms of $h-\tilde h$ in the following lemma. However, we will keep all terms of $\mathcal{J}^{bad}$ in a region $\{p(v)-\pt >\delta\}$ for small values of $v$.

\begin{lemma}\label{lem-max}
Let $\tilde U^{X_1,X_2}$ and $a^{X_1,X_2}$ as in \eqref{xwave} and \eqref{xa}.  Let $\delta$ be any positive constant.
Then, for any $U\in \mathcal{H}_T$, 
\begin{align}
\begin{aligned}\label{ineq-1}
\mathcal{J}^{bad} (U) -\mathcal{J}^{good} (U)=B_\delta(U)- G_\delta(U),
\end{aligned}
\end{align}
where
\begin{align}
\begin{aligned}\label{bad}
&{B}_\delta(U):= \sum_{i=1}^2\bigg[   \s_i \int_\bbr a^{X_1,X_2} (\tilde v_i)_x^{X_i} p(v|\tilde v^{X_1,X_2}) dx\\
 &\quad\quad\quad  + \frac{1}{2\s_i} \int_\bbr (a_i)_x^{X_i}|p(v)-p(\tilde v^{X_1,X_2}) |^2 {\mathbf 1}_{\{p(v)- p(\tilde v^{X_1,X_2}) \leq\delta\}}  dx \\
&\quad\quad\quad + \int_\bbr (a_i)_x^{X_i}  \big(p(v)-p(\tilde v^{X_1,X_2} )\big) \big(h-\tilde h^{X_1,X_2} \big) {\mathbf 1}_{\{p(v)- p(\tilde v^{X_1,X_2})  > \delta\}}  dx\\
 &\quad\quad\quad -\int_\bbr (a_i)_x^{X_i}  v^\beta \big(p(v)-p(\tilde v^{X_1,X_2} )\big)\partial_x \big(p(v)-p(\tilde v^{X_1,X_2} )\big)  dx\\
&\quad\quad\quad -\int_\bbr (a_i)_x^{X_i}  \big(p(v)-p(\tilde v^{X_1,X_2})\big) (v^\beta - (\tilde v^\beta)^{X_1,X_2} ) \partial_{x} p(\tilde v^{X_1,X_2})  dx \bigg] \\
&\quad\quad\quad -\int_\bbr a^{X_1,X_2}  \partial_x \big(p(v)-p(\tilde v^{X_1,X_2})\big) (v^\beta - (\tilde v^\beta)^{X_1,X_2} ) \partial_{x} p(\tilde v^{X_1,X_2})  dx\\
&\quad\quad\quad + \int_\bbr a^{X_1,X_2} \big(p(v)-p(\tilde v^{X_1,X_2}) \big)  \tilde E_1 dx -  \int_\bbr a^{X_1,X_2} (h-\tilde h^{X_1,X_2})  \tilde E_2 dx ,
\end{aligned}
\end{align}
and
\begin{align}
\begin{aligned}\label{good}
{G}_\delta(U) &:=\sum_{i=1}^2\bigg( \frac{\sigma_i}{2}\int_\bbr (a_i)_x^{X_i}  \left|  h-\tilde h^{X_1,X_2} -\frac{p(v)-p(\tilde v^{X_1,X_2})}{\s_i}\right|^2  {\mathbf 1}_{\{p(v)-p(\tilde v^{X_1,X_2})  \leq\delta\}}  dx \\
&+\frac{\sigma_i}{2}\int_\bbr (a_i)_x^{X_i} \left| h-\tilde h^{X_1,X_2} \right|^2  {\mathbf 1}_{\{p(v)-p(\tilde v^{X_1,X_2})  >\delta\}}  dx  +\sigma_i  \int_\bbr  (a_i)_x^{X_i}  Q(v|\tilde v^{X_1,X_2})  dx\bigg)\\
&+\int_\bbr a^{X_1,X_2}  v^\beta |\partial_x \big(p(v)-p(\tilde v^{X_1,X_2} )\big)|^2 dx,
\end{aligned}
\end{align}

\begin{remark}\label{rem:0}
Since $\sigma_i (a_i)_x >0$ and $a>0$, $-{G}_\delta$ consists of good terms. 
\end{remark}

\end{lemma}
\begin{proof}
For a given $\delta>0$, we split the first terms of $\mathcal{J}^{bad}$ and $-\mathcal{J}^{good}$ as follows:
\begin{align*}
\begin{aligned}
&\sum_{i=1}^2 \int_\bbr (a_i)_x^{X_i}  \big(p(v)-p(\tilde v^{X_1,X_2} )\big) \big(h-\tilde h^{X_1,X_2} \big)   dx\\
&=\sum_{i=1}^2\bigg[ \underbrace{ \int_\bbr (a_i)_x^{X_i}  \big(p(v)-p(\tilde v^{X_1,X_2} )\big) \big(h-\tilde h^{X_1,X_2} \big) {\mathbf 1}_{\{p(v)-p(\tilde v^{X_1,X_2})  \leq\delta\}}  dx}_{=: J_i} \\
&\quad + \int_\bbr (a_i)_x^{X_i}  \big(p(v)-p(\tilde v^{X_1,X_2} )\big) \big(h-\tilde h^{X_1,X_2} \big) {\mathbf 1}_{\{p(v)- p(\tilde v^{X_1,X_2})  > \delta\}}  dx \bigg]
\end{aligned}
\end{align*}
and 
\begin{align*}
\begin{aligned}
&- \sum_{i=1}^2\frac{\sigma_i}{2}\int_\bbr (a_i)_x^{X_i} \left| h-\tilde h^{X_1,X_2} \right|^2 dx\\
&=-\sum_{i=1}^2\bigg[ \underbrace{ \sum_{i=1}^2\frac{\sigma_i}{2}\int_\bbr (a_i)_x^{X_i} \left| h-\tilde h^{X_1,X_2} \right|^2 {\mathbf 1}_{\{p(v)- p(\tilde v^{X_1,X_2})  \leq\delta\}}  dx }_{=: K_i} \\
&\quad -\sum_{i=1}^2\frac{\sigma_i}{2}\int_\bbr (a_i)_x^{X_i} \left| h-\tilde h^{X_1,X_2} \right|^2 {\mathbf 1}_{\{p(v)-p(\tilde v^{X_1,X_2})  >\delta\}}  dx \bigg]
\end{aligned}
\end{align*}
Applying the quadratic identity $\alpha x^2+ \beta x =\alpha(x+\frac{\beta}{2\alpha})^2-\frac{\beta^2}{4\alpha}$ with $x:=h-\tilde h^{X_1,X_2} $ to the integrands of $J_i+K_i$, we find
\begin{align*}
\begin{aligned}
&- \frac{\sigma_i}{2} \left| h-\tilde h^{X_1,X_2}\right|^2 + \big(p(v)-p(\tilde v^{X_1,X_2}) \big) \big(h-\tilde h^{X_1,X_2} \big)  \\
&\quad = - \frac{\s_i}{2} \left|  h-\tilde h^{X_1,X_2} -\frac{p(v)-p(\tilde v^{X_1,X_2})}{\s_i}\right|^2+ \frac{1}{2\s_i}|p(v)-p(\tilde v^{X_1,X_2})|^2.
\end{aligned}
\end{align*}
Therefore, we have the desired representation.
\end{proof}

\subsection{Construction of shifts}
For a given $\eps>0$, we consider a continuous function $\Phi_\eps$ defined by
\beq\label{Phi-d}
\Phi_\eps (y)=
\left\{ \begin{array}{ll}
      0,\quad &\mbox{if}~ y\le 0, \\
      -\frac{1}{\eps^4}y,\quad &\mbox{if} ~ 0\le y\le \eps^2, \\
       -\frac{1}{\eps^2},\quad &\mbox{if}  ~y\ge \eps^2. \end{array} \right.
\eeq
\beq\label{Psi-d}
\Psi_\eps (y)=
\left\{ \begin{array}{ll}
      1,\quad &\mbox{if}~ y\le -\eps^2, \\
      -\frac{1}{\eps^2}y,\quad &\mbox{if} ~ -\eps^2\le y\le 0, \\
       0,\quad &\mbox{if}  ~y\ge 0. \end{array} \right.
\eeq
For any fixed $\eps_1, \eps_2>0$, and $U\in \mathcal{H}_T$,  we define a pair of shift functions ${X_1\choose X_2}$ as a solution to the system of nonlinear ODEs:
\beq\label{X-def}
\left\{ \begin{array}{ll}
       \dot X_1(t) = \Phi_{\eps_1} (Y_1(U)) \Big(2|\mathcal{J}^{bad}(U)|+1 \Big) -\frac{\s_1}{2}  \Psi_{\eps_1} (Y_1(U)) ,\\
        \dot X_2(t) = -\Phi_{\eps_2} (-Y_2(U)) \Big(2|\mathcal{J}^{bad}(U)|+1 \Big) -\frac{\s_2}{2}  \Psi_{\eps_2} (-Y_2(U)) ,\\
       X_1(0)=X_2(0)=0, \end{array} \right.
\eeq
where $Y_1, Y_2$ and $\mathcal{J}^{bad}$ are as in \eqref{ybg-first}.\\
Then, it is shown in Appendix \ref{app-shift} that the system \eqref{X-def} has a unique absolutely continuous solution ${X_1\choose X_2}$ on $[0,T]$. \\

Since it follows from \eqref{X-def} that for each $i=1, 2$,
\beq\label{explicit-X}
\dot X_i (t) =
\left\{ \begin{array}{ll}
      (-1)^i \eps_i^{-2} \big(2|\mathcal{J}^{bad}(U)|+1 \big),\quad &\mbox{if}~ (-1)^{i-1} Y_i(U) \ge \eps_i^2, \\
  -\eps_i^{-4} Y_i(U) \big(2|\mathcal{J}^{bad}(U)|+1 \big),\quad &\mbox{if} ~ 0\le  (-1)^{i-1} Y_i(U)\le \eps_i^2, \\
       (-1)^{i-1} \frac{1}{2} \s_i \eps_i^{-2}  Y_i(U),\quad &\mbox{if}  ~ -\eps_i^2  \le  (-1)^{i-1} Y_i(U)\le 0 ,\\
       -\frac{1}{2}\s_i , \quad &\mbox{if}  ~ (-1)^{i-1} Y_i(U) \le -\eps_i^2 ,
       \end{array} \right.
\eeq
the shifts satisfy the bounds:
\[
\dot X_1(t) \le -\frac{\s_1}{2},\quad \dot X_2(t) \ge -\frac{\s_2}{2},\quad\forall t>0.
\]
Thus,
\[
X_1(t) \le -\frac{\s_1}{2} t,\quad X_2(t) \ge -\frac{\s_2}{2} t,\quad\forall t>0,
\]
which gives \eqref{msepX12}.\\
Especially, we have
\beq\label{sx12}
X_1(t)+\s_1 t \le \frac{\s_1}{2} t<0,\quad X_2(t)+\s_2 t \ge \frac{\s_2}{2} t>0,\quad\forall t>0.
\eeq

\subsection{Main proposition}
The main proposition for the proof of Theorem \ref{thm_main} is the following.

\begin{proposition}\label{prop:main}
For the given constant $U_*:=(v_*,u_*)\in\bbr^+\times\bbr$, there exist positive constants $\delta_0,\delta_1 \in(0,1/2)$ such that the following holds.
Let $\tilde U^{X_1,X_2}$ be the composite wave for the given constant states $U_-, U_m, U_+ \in B_{\delta_0}(U_*)$, where $\eps_1=|p(v_-)-p(v_m)|$ and $\eps_2=|p(v_m)-p(v_+)|$.
Then, for any $\lambda >0$ with $\eps_1/\lambda, \eps_2/\lambda<\delta_0$ and $\lambda<\delta_0$,  there exist positive constants $C_{\eps,\deltaz}$ and $t_0$  such that the following holds.\\
For any $U\in\mathcal{H}_T$, let ${X_1\choose X_2}$ be the solution to \eqref{X-def}.
Then,  for all $U\in \mathcal{H}_T$ satisfying  $Y_1(U)\le\eps_1^2$ and $Y_2(U) \ge -\eps_2^2$, 
\begin{align}
\begin{aligned}\label{prop:est}
&\mathcal{R}(U):= -\frac{1}{\eps_1^4} |Y_1(U)|^2 {\mathbf 1}_{\{0\le Y_1(U) \le \eps_1^2\}} 
+\frac{\s_1}{2\eps_1^2} |Y_1(U)|^2 {\mathbf 1}_{\{ -\eps_1^2 \le Y_1(U) \le 0 \}} 
-\frac{\s_1}{2} Y_1(U) {\mathbf 1}_{\{Y_1(U) \le -\eps_1^2 \}} \\
&\quad\quad\quad-\frac{1}{\eps_2^4} |Y_2(U)|^2 {\mathbf 1}_{\{-\eps_2^2\le Y_2(U) \le 0 \}} 
-\frac{\s_2}{2\eps_2^2}  |Y_2(U)|^2 {\mathbf 1}_{\{0\le Y_2(U) \le\eps_2^2 \}}
-\frac{\s_2}{2} Y_2(U) {\mathbf 1}_{\{Y_2(U) \ge \eps_2^2 \}}\\
&\quad\quad\quad+  {B}_{\deltao}(U)+\delta_0\frac{\min(\eo,\et)}{\lambda} |{B}_{\deltao}(U)|
-{G}_{11}^{-}(U)-{G}_{11}^{+}(U) -{G}_{12}^{-}(U)-{G}_{12}^{+}(U) \\
&\quad\quad\quad -\left(1-\delta_0\frac{\eps_1}{\lambda}\right){G}_{21}(U) -\left(1-\delta_0\frac{\eps_2}{\lambda}\right){G}_{22}(U) -(1-\delta_0){D}(U) \\
&\quad \le  C_{\eps,\deltaz}\left[   \frac{1}{t^2} \int_\bbr  \eta(U|\tilde U) dx + \left(\exp\big( -C_\eps t\big) + \frac{1}{t^4} \right) \right] \mathbf{1}_{t\ge t_0} + C \mathbf{1}_{t\le t_0} ,
\end{aligned}
\end{align}
where $Y_i$ and ${B}_{\deltao}$ are as in \eqref{ybg-first} and \eqref{bad}, and ${G}_{11}^{-},{G}_{11}^{+}, {G}_{12}^{-},{G}_{12}^{+}, {G}_{21},{G}_{22}, {D}$ denote the good terms of ${G}_{\deltao}$ in  \eqref{good} as follows: for each $i=1, 2$,
\begin{align}
\begin{aligned}\label{ggd}
&{G}_{1i}^{-}(U):=\frac{\sigma_i}{2}\int_\bbr (a_i)_x^{X_i} \left| h-\tilde h^{X_1,X_2} \right|^2  {\mathbf 1}_{\{p(v)-p(\tilde v^{X_1,X_2})  >\deltao\}}  dx,\\
&{G}_{1i}^{+}(U):= \frac{\sigma_i}{2}\int_\bbr (a_i)_x^{X_i}  \left|  h-\tilde h^{X_1,X_2} -\frac{p(v)-p(\tilde v^{X_1,X_2})}{\s_i}\right|^2  {\mathbf 1}_{\{p(v)-p(\tilde v^{X_1,X_2})  \leq\deltao\}}  dx,\\
&{G}_{2i}(U):=\sigma_i  \int_\bbr  (a_i)_x^{X_i}  Q(v|\tilde v^{X_1,X_2})  dx,\\
&{D}(U):= \int_\bbr a^{X_1,X_2}  v^\beta |\partial_x \big(p(v)-p(\tilde v^{X_1,X_2} )\big)|^2 dx .
\end{aligned}
\end{align} 
\end{proposition}

\subsection{Proof of Theorem~\ref{thm_main} from the main Proposition}
We here show how Proposition \ref{prop:main} implies Theorem~\ref{thm_main}.\\

To prove the contraction estimate \eqref{cont_main2}, following \eqref{ineq-0}, we may estimate 
\begin{align}
\begin{aligned}\label{contem0}
\mathcal{F}(U)&:=\sum_{i=1}^2\big( \dot X_i(t) Y_i(U) \big)  +\mathcal{J}^{bad} (U) -\mathcal{J}^{good} (U).
\end{aligned}
\end{align} 
First, it follows from \eqref{explicit-X} that for each $i=1, 2$,
\[
 \dot X_i(t) Y_i(U) \le
\left\{ \begin{array}{ll}
    -2|\mathcal{J}^{bad}(U)|,\quad &\mbox{if}~ (-1)^{i-1} Y_i(U) \ge \eps_i^2, \\
  -\eps_i^{-4} |Y_i(U)|^2,\quad &\mbox{if} ~ 0\le  (-1)^{i-1} Y_i(U)\le \eps_i^2, \\
     (-1)^{i-1}\frac{1}{2} \s_i \eps_i^{-2} |Y_i(U)|^2,\quad &\mbox{if}  ~ -\eps_i^2  \le  (-1)^{i-1} Y_i(U)\le 0 ,\\
       -\frac{1}{2}\s_i Y_i(U) , \quad &\mbox{if}  ~ (-1)^{i-1} Y_i(U) \le -\eps_i^2 .
       \end{array} \right.
\]
Then we first find that for all $U\in \mathcal{H}_T$ satisfying  $Y_1(U)\ge \eps_1^2$ or $Y_2(U) \le -\eps_2^2$, 
$$
\mathcal{F}(U)  \le -|\mathcal{J}^{bad}(U)|-\mathcal{J}^{good}(U)  \le 0.
$$
Since \eqref{ineq-1} with $\delta=\delta_1$ yields that for all $U\in \mathcal{H}_T$ satisfying  $Y_1(U)\le \eps_1^2$ and $Y_2(U) \ge -\eps_2^2$, 
\begin{align*}
\begin{aligned}
\mathcal{F}(U)
& \le -\frac{1}{\eps_1^4} |Y_1(U)|^2 {\mathbf 1}_{\{0\le Y_1(U) \le \eps_1^2\}} 
+\frac{\s_1}{2\eps_1^2} |Y_1(U)|^2 {\mathbf 1}_{\{ -\eps_1^2 \le Y_1(U) \le 0 \}} 
-\frac{\s_1}{2} Y_1(U) {\mathbf 1}_{\{Y_1(U) \le -\eps_1^2 \}} \\
&\quad-\frac{1}{\eps_2^4} |Y_2(U)|^2 {\mathbf 1}_{\{-\eps_2^2\le Y_2(U) \le 0 \}} 
-\frac{\s_2}{2\eps_2^2}  |Y_2(U)|^2 {\mathbf 1}_{\{0\le Y_2(U) \le\eps_2^2 \}}
-\frac{\s_2}{2} Y_2(U) {\mathbf 1}_{\{Y_2(U) \ge \eps_2^2 \}}\\
&\quad+ {B}_\deltao(U)- {G}_\deltao(U) ,
\end{aligned}
\end{align*}
Proposition \ref{prop:main} implies that for all $U\in \mathcal{H}_T$ satisfying  $Y_1(U)\le \eps_1^2$ and $Y_2(U) \ge -\eps_2^2$, 
\begin{align*}
\begin{aligned}
&\mathcal{F}(U) +
 \delta_0\frac{\min(\eo,\et)}{\lambda} |{B}_{\delta_1}(U)| 
+\delta_0\frac{\eps_1}{\lambda} {G}_{21}(U) +\delta_0\frac{\eps_2}{\lambda}{G}_{22}(U) - \delta_0{D}(U) \\
&\le C_{\eps,\deltaz}\left[   \frac{1}{t^2} \int_\bbr  \eta(U|\tilde U) dx + \left(\exp\big( -C_\eps t\big) + \frac{1}{t^4} \right) \right] \mathbf{1}_{t\ge t_0} + C \mathbf{1}_{t\le t_0}.
\end{aligned}
\end{align*}
Thus, using the above estimates together with $\eps_i/\lam <\delta_0 <1$ and the definition of $\mathcal{J}^{good}$, we find that for all $U\in \mathcal{H}_T$,
 \begin{align*}
\begin{aligned}
&\mathcal{F}(U) +\delta_0\frac{\eps_1}{\lambda} {G}_{21}(U)+\delta_0\frac{\eps_2}{\lambda}{G}_{22}(U) + \delta_0{D}(U)\\
&\quad+ |\mathcal{J}^{bad}(U)| {\mathbf 1}_{\{Y_1(U) \ge \eps_1^2\}\cup\{Y_2(U) \le -\eps_1^2\} } + \delta_0\frac{\min(\eo,\et)}{\lambda} |{B}_{\delta_1}(U)|  {\mathbf 1}_{\{Y_1(U) \le \eps_1^2\}\cap\{Y_2(U) \ge -\eps_1^2\} }  \\
& \le C_{\eps,\deltaz}\left[   \frac{1}{t^2} \int_\bbr  \eta(U|\tilde U) dx + \left(\exp\big( -C_\eps t\big) + \frac{1}{t^4} \right) \right] \mathbf{1}_{t\ge t_0} + C \mathbf{1}_{t\le t_0}.
\end{aligned}
\end{align*}
This together with \eqref{ineq-0}, \eqref{contem0} and $1/2\le a^{X_1,X_2}\le 1$ implies that for a.e. $t>0$,
\begin{align*}
\begin{aligned}
&\frac{d}{dt}\int_{\bbr} a^{X_1,X_2}\eta(U|\tilde U^{X_1,X_2}) dx  +\delta_0\frac{\eps_1}{\lambda} {G}_{21}(U)+\delta_0\frac{\eps_2}{\lambda}{G}_{22}(U) + \delta_0{D}(U)\\
&\quad+ |\mathcal{J}^{bad}(U)| {\mathbf 1}_{\{Y_1(U) \ge \eps_1^2\}\cup\{Y_2(U) \le -\eps_1^2\} } + \delta_0\frac{\min(\eo,\et)}{\lambda} |{B}_{\delta_1}(U)|  {\mathbf 1}_{\{Y_1(U) \le \eps_1^2\}\cap\{Y_2(U) \ge -\eps_1^2\} }  \\
&=\mathcal{F}(U) +\delta_0\frac{\eps_1}{\lambda} {G}_{21}(U)+\delta_0\frac{\eps_2}{\lambda}{G}_{22}(U) + \delta_0{D}(U)\\
&\quad+ |\mathcal{J}^{bad}(U)| {\mathbf 1}_{\{Y_1(U) \ge \eps_1^2\}\cup\{Y_2(U) \le -\eps_1^2\} } + \delta_0\frac{\min(\eo,\et)}{\lambda} |{B}_{\delta_1}(U)|  {\mathbf 1}_{\{Y_1(U) \le \eps_1^2\}\cap\{Y_2(U) \ge -\eps_1^2\} }  \\
&\le C_{\eps,\deltaz}\left[   \frac{1}{t^2} \int_\bbr  \eta(U|\tilde U) dx + \left(\exp\big( -C_\eps t\big) + \frac{1}{t^4} \right) \right] \mathbf{1}_{t\ge t_0} + C \mathbf{1}_{t\le t_0}.
\end{aligned}
\end{align*}
Since $t^{-2}$ and $\exp\big( -C_\eps t \big) +t^{-4}$ are integrable on $[t_0,\infty)$,  
Gr$\ddot{\rm o}$nwall's inequality implies that there exists a positive constant $C(\eps_1,\eps_2,\lambda,\deltaz)$ such that
\begin{align}
\begin{aligned}\label{cont-pre}
&\int_{\bbr} a^{X_1,X_2}\eta(U^X|\tilde U^{X_1,X_2}) dx  +  \int_0^t {G}_{21}(U)ds+ \int_0^t {G}_{22}(U) ds+ \delta_0 \int_0^t {D}(U) ds \\
&\quad + \int_0^t  |\mathcal{J}^{bad}(U)| {\mathbf 1}_{\{Y_1(U) \ge \eps_1^2\}\cup\{Y_2(U) \le -\eps_1^2\} } ds+  \int_0^t  |{B}_{\delta_1}(U)|  {\mathbf 1}_{\{Y_1(U) \le \eps_1^2\}\cap\{Y_2(U) \ge -\eps_1^2\} }ds \\
&\le C(\eps_1,\eps_2,\lambda,\deltaz)\bigg[ \int_{\bbr} a(0,x)\eta(U_0(x)|\tilde U(0,x)) dx +1 \bigg],
\end{aligned}
\end{align}
which completes \eqref{cont_main2}.\\

To estimate $|\dot X_i|$, we first observe that \eqref{explicit-X} implies that for each $i=1, 2$,
\beq\label{contx}
|\dot X_i(t)| \le \max\bigg( \frac{1}{\eps_i^2} (2|\mathcal{J}^{bad}(U)| +1) , \frac{|\s_i|}{2}\bigg) , \qquad\mbox{for a.e. } t\in(0,T).
\eeq
Notice that it follows from \eqref{cont-pre} that
\begin{align}
\begin{aligned}\label{jb-cont}
\int_0^{T} f(t) dt  \le   C(\eps_1,\eps_2,\lambda,\deltaz)\int_{\bbr} \eta(U_0|\tilde U(0,x)) dx,
\end{aligned}
\end{align}
where 
\[
f(t):= |\mathcal{J}^{bad}(U)| {\mathbf 1}_{\{Y_1(U) \ge \eps_1^2\}\cup\{Y_2(U) \le -\eps_1^2\}}  +| \mathcal{B}_{\delta_1}(U)|{\mathbf 1}_{\{Y_1(U) \le \eps_1^2\}\cap\{Y_2(U) \ge -\eps_1^2\}} .
\]
To estimate $|\mathcal{J}^{bad}(U)|$ globally in time, using  \eqref{ineq-1} with the definitions  of $\mathcal{J}^{good}$ and ${G}_{\deltao}$,
we find that
\begin{align*}
\begin{aligned}
&|\mathcal{J}^{bad}(U)|\\
&\le |\mathcal{J}^{bad}(U)| {\mathbf 1}_{\{Y_1(U) \ge \eps_1^2\}\cup\{Y_2(U) \le -\eps_1^2\}} +|\mathcal{J}^{bad}(U)| {\mathbf 1}_{\{Y_1(U) \le \eps_1^2\}\cap\{Y_2(U) \ge -\eps_1^2\}}  \\
&=|\mathcal{J}^{bad}(U)| {\mathbf 1}_{\{Y_1(U) \ge \eps_1^2\}\cup\{Y_2(U) \le -\eps_1^2\}} +| \mathcal{J}^{good}(U) + {B}_{\delta_1}(U)-{G}_{\delta_1}(U)| {\mathbf 1}_{\{Y_1(U) \le \eps_1^2\}\cap\{Y_2(U) \ge -\eps_1^2\}}\\
&\le |\mathcal{J}^{bad}(U)| {\mathbf 1}_{\{Y_1(U) \ge \eps_1^2\}\cup\{Y_2(U) \le -\eps_1^2\}}  +| \mathcal{B}_{\delta_1}(U)|{\mathbf 1}_{\{Y_1(U) \le \eps_1^2\}\cap\{Y_2(U) \ge -\eps_1^2\}} \\
&+ \frac{|\sigma_i| }{2}\int_\bbr |(a_i)_x^{X_i}| \left| \big(h-\tilde h^{X_1,X_2} \big)^2 - \left(h-\tilde h^{X_1,X_2} -\frac{p(v)-p(\tilde v^{X_1,X_2})}{\s_i} \right)^2 \right| {\mathbf 1}_{\{p(v)-p(\tilde v^{X_1,X_2})  \leq\deltao\}} dx\\
&\le f(t) +C \int_\bbr |(a_i)_x^{X_i}|  \Big(\big(h-\tilde h^{X_1,X_2} \big)^2 + \big(p(v)-p(\tilde v^{X_1,X_2}) \big)^2 \Big) {\mathbf 1}_{\{p(v)-p(\tilde v^{X_1,X_2})  \leq\deltao\}} dx .
\end{aligned}
\end{align*}
Since for any $v$ satisfying $p(v)-\pt \leq\deltao$, there exists a positive constant $c_*$ such that $v>c_*^{-1}$ and $|p(v)-\pt|\le c_*$, we use \eqref{pressure2} and \eqref{rel_Q} to have
\begin{align*}
\begin{aligned}
& \int_\bbr  |(a_i)_x^{X_i}| \big(p(v)-p(\tilde v^{X_1,X_2}) \big)^2 {\mathbf 1}_{\{p(v)-p(\tilde v^{X_1,X_2})  \leq\deltao\}} dx \\
&\quad  \le  c_* \int_{v>c_*^{-1}}  \big|p(v)-p(\tilde v^{X_1,X_2}) \big| {\mathbf 1}_{\{v\ge 3v_- \}}  dx + \int_{v>c_*^{-1}} |(a_i)_x^{X_i}| \big|p(v)-p(\tilde v^{X_1,X_2}) \big|^2 {\mathbf 1}_{\{v\le 3v_- \}}  dx  \\
&\quad \le C\int_{v>c_*^{-1}} |(a_i)_x^{X_i}| \Big( |v-\tilde v^{X_1,X_2}| {\mathbf 1}_{\{v\ge 3v_- \}} + |v-\tilde v^{X_1,X_2}|^2 {\mathbf 1}_{\{v\le 3v_- \}} \Big) dx \\
&\quad \le C\int_\bbr |(a_i)_x^{X_i}| Q(v| \tilde v^{X_1,X_2}) dx .
\end{aligned}
\end{align*}
Therefore, using $ |(a_i)_x^{X_i}| \leq C \delta_0$ and $\delta_0\le \frac{1}{2}\le  a^{X_1,X_2}$, we have
\begin{align*}
|\mathcal{J}^{bad}(U)| \le f(t) + C \int_\bbr  a^{X_1,X_2} \eta(U|\tilde U^{X_1,X_2} ) dx,  
\end{align*}
which together with \eqref{cont-pre} and \eqref{contx} implies that
\[
|\dot X_i(t)|\le C(\eps_1,\eps_2,\lambda,\deltaz) \left[ f(t) + \int_{\bbr} \eta(U_0(x)|\tilde U(0,x)) dx + 1 \right] .
\]
This and \eqref{jb-cont} give \eqref{est-shift1}.\\

\section{Abstract propositions for a general setting}\label{sec:abs}
\setcounter{equation}{0}
This section provides some useful propositions that will be all used in Section \ref{sec:prop} for the proof of Proposition \ref{prop:main}. The propositions are generalizations of \cite[Propositions 4.2 and 4.3 and Lemmas 4.4-4.8]{KV-unique19}. Those are stated below in an abstract setting for future application in a general context, for example, for studies on various composite waves.

\subsection{Sharp estimate near a shock wave}\label{sec:poin}

The following proposition is a generalization of \cite[Proposition 4.2]{KV-unique19}.

\begin{proposition}\label{prop:main3}
For any constant $C_2>0$ and any constant $U_*:=(v_*,u_*)\in \bbr^+\times\bbr$, 
there exists $\deltao>0$ such that for any $\eps>0$ and any $\lam, \delta\in(0,\deltao]$  satisfying $\eps/\lam\le \deltao$,
 the following holds.\\ 
Let $U_l:=(v_l,u_l)$ and $U_r:=(v_r,u_r)$ be any constants satisfying $U_l, U_r\in B_{\deltao}(U_*)$ and $\eps= |p(v_l)-p(v_r)|$, and one of the two conditions \eqref{end-con} with the velocity $\s_0$. Let   
 $\tilde U_0:=(\tilde v_0, \tilde h_0)$ be the viscous shock as a solution to the equation \eqref{re_shock} connecting the left end state $U_l$ and the right end state $U_r$.\\
Let $a_0$ be a function such that
\[
\partial_x a_0 = -\lam \frac{\partial_x p(\tilde v_0)}{\eps}. 
\]
Let $a$ be any positive function such that $\|a-1\|_{L^\infty(\bbr)}\le 2\lam$.\\
 Let $\phi$ be any Lipschitz function, and let
 \begin{align}
\begin{aligned}\label{ftn_poin}
&\mathcal{Y}^g(v):=-\frac{1}{2\sigma_0^2}\int_\bbr (a_0)_x \phi^2(x) |p(v)-p(\tilde v_0)|^2 dx -\int_\bbr (a_0)_x \phi^2(x) Q(v|\tilde v_0) dx\\
&\qquad\qquad -\int_\bbr a \partial_x p(\tilde v_0) \phi(x) (v-\tilde v_0)dx+\frac{1}{\sigma_0}\int_\bbr a  (\tilde h_0)_x \phi(x) \big(p(v)-p(\tilde v_0)\big)dx,\\
&\mathcal{I}_{1}(v):= \sigma_0\int_\bbr a (\tilde v_0)_x \phi^2(x)  p(v|\tilde v_0) dx,\\
&\mathcal{I}_{2}(v):= \frac{1}{2\sigma_0} \int_\bbr (a_0)_x  \phi^2(x) |p(v)-p(\tilde v_0)|^2dx,\\
&\mathcal{G}_{2} (v):=\sigma_0  \int_\bbr  (a_0)_x \bigg( \frac{1}{2\gamma}  p(\tilde v_0)^{-\frac{1}{\gamma}-1} \phi^2(x) \big(p(v)-p(\tilde v_0)\big)^2\\
&\qquad\qquad - \frac{1+\gamma}{3\gamma^2} p(\tilde v_0)^{-\frac{1}{\gamma}-2} \phi^3(x)\big(p(v)-p(\tilde v_0)\big)^3 \bigg) dx, \\
&\mathcal{D} (v):=\int_\bbr a \, v^\beta |\partial_x \big(\phi(x) (p(v)-p(\tilde v_0))\big)|^2 dx,
\end{aligned}
\end{align}
For any function $v:\bbr\to \bbr^+$ such that $\mathcal{D}(v)+\mathcal{G}_2(v)$ is finite, if
\beq\label{assYp}
|\mathcal{Y}^g(v)|\leq C_2 \frac{\eps^2}{\lambda},\qquad  \|p(v)-p(\tilde v_0)\|_{L^\infty(\bbr)}\leq 2\delta_{1},
\eeq
then
\begin{align}
\begin{aligned}\label{redelta}
\mathcal{R}_{\delta}(v)&:=-\frac{1}{\eps\delta}|\mathcal{Y}^g(v)|^2 +\mathcal{I}_{1}(v)+\delta|\mathcal{I}_{1}(v)|\\
&\quad\quad+\mathcal{I}_{2}(v)+\delta\left(\frac{\eps}{\lambda}\right)|\mathcal{I}_{2}(v)|-\left(1-\delta\left(\frac{\eps}{\lambda}\right)\right)\mathcal{G}_{2}(v)-(1-\delta)\mathcal{D}(v)\\
&\le 0.
\end{aligned}
\end{align}
\end{proposition}

\begin{proof}
The proof is essentially the same as that of \cite[Proposition 4.2 and Appendix A]{KV-unique19} through a generalization. 
Notice that the functionals $\mathcal{Y}^g, \mathcal{I}_{1}, \mathcal{I}_{2}, \mathcal{D}$, by putting $\phi \equiv 1$ in their integrands, are respectively the same as $Y_g, \mathcal{I}_{1}, \mathcal{I}_{2},  \mathcal{D}$ in \cite[Proposition 4.2]{KV-unique19}. 
On the other hand, the integrand of the functional $\mathcal{G}_{2}$ is the same as the approximation for $Q(v|\tilde v)$ of $\mathcal{G}_{2}$ in \cite[Proposition 4.2]{KV-unique19} by \eqref{Q-est11} of Lemma \ref{lem:local} together with $\phi\equiv 1$. \\
For completeness, the main parts of the proof are given in Appendix \ref{app-exp}.
\end{proof}

\subsection{Smallness of the weighted relative entropy}\label{sec:sm}
The following proposition is a generalization of \cite[Lemma 4.4]{KV-unique19}.

\begin{proposition}\label{prop:sm}
For a given constant $U_*:=(v_*,u_*)\in \bbr^+\times\bbr$, there exist positive constants $\delta_0, C, C_0$ such that for any $\eps, \lam>0$ satisfying $\eps/\lambda< \delta_0$ and $\lambda<\delta_0$, the following holds.\\
Let $\tilde U_0:=(\tiv_0,\tih_0): \bbr\to  \bbr^+\times\bbr$, $\bw:\bbr\to\bbr$, $\bv :=(\bv_1,\bv_2) : \bbr\to\bbr^2$ and ${\bf a}:\bbr\to\bbr$ be any functions such that
\begin{align}
&|\tilde U_0(x)-U_*|\le C\delta_0,\quad \tiv_0(x)\ge C^{-1},  \quad  |\bv(x)| \le C \frac{\eps}{\lam}|\bw(x)|, \quad\forall x\in\bbr, \label{asunif} \\
&\bw \mbox{ is either positive or negative globally on } \bbr,  \label{asposi} \\
& \bv_1, \bv_2 \in L^1(\bbr),\quad \int_\bbr |\bw| dx =\lam ,\quad \|{\bf a}\|_{L^\infty(\bbr)} \le 1 \label{asinte}.
\end{align}
Let  $U:=(v,h): \bbr\to  \bbr^+\times\bbr$ be any function such that $\sgn(\bw){\bf Y}(U)\le \eps^2$, where
\[
{\bf Y}(U):=  \int_{\bbr} \bw \eta(U|\tilde U_0 ) dx + \int_{\bbr} {\bf a} \bv^T \nabla^2\eta(\tilde U_0) (U-\tilde U_0) dx .
\]
Then,
\beq\label{genl1}
 \int_\bbr |\bw| |h-\tih_0|^2\,dx + \int_\bbr |\bw| Q(v|\tiv_0)\,dx \le C\frac{\eps^2}{\lambda}, 
 \eeq
and
\beq\label{genyabs}
|{\bf Y}(U)| \le C_0\frac{\eps^2}{\lambda}.
\eeq
\end{proposition}
\begin{proof}
$\bullet$ {\it Proof of \eqref{genl1}} :
We first use \eqref{rel_Q} to have
\begin{align}
\begin{aligned}\label{aest-1}
&\int_\bbr  |\bw|\eta(U|\tilde U_0) \ge  \int_\bbr  |\bw| \frac{|h-\tilde h_0 |^2}{2} 
 +c_1\int_{v\le 3v_*}|\bw|  |v-\tilde v_0|^2 + c_2\int_{v> 3v_*}|\bw|  |v-\tilde v_0|.
\end{aligned}
\end{align}
Using \eqref{asposi} and
\[
 \int_{\bbr}  \bw \eta(U|\tilde U_0 ) dx ={\bf Y}(U) - \int_{\bbr} {\bf a} \bv^T \nabla^2\eta(\tilde U_0) (U-\tilde U_0) dx,
\]
we have
\begin{align*}
 \int_{\bbr} | \bw| \eta(U|\tilde U_0 ) dx &= \sgn(\bw) \int_{\bbr} \bw\eta(U|\tilde U_0 )  dx \\
 &= \sgn(\bw) {\bf Y}(U) - \sgn(\bw) \int_{\bbr} {\bf a} \bv^T \nabla^2\eta(\tilde U_0) (U-\tilde U_0) dx.
\end{align*}

Thus, for all $U$ satisfying $\sgn(\bw){\bf Y}(U)\le \eps^2$, 
\[
 \int_{\bbr} | \bw| \eta(U|\tilde U_0 ) dx \le \eps^2 + \int_{\bbr} |{\bf a}|  |\bv^T \nabla^2\eta(\tilde U_0) (U-\tilde U_0)| dx.
\]
Note that the assumptions \eqref{asunif} and \eqref{asinte} with \eqref{2dere} implies
\[
 \int_{\bbr} | \bw| \eta(U|\tilde U_0 ) dx \le  \eps^2 +C\frac{\eps}{\lambda}\int_\bbr   | \bw|  (|v-\tilde v_0| +|h-\tilde h_0| ) dx .
\]
Then using \eqref{asinte} and Young's inequality, we have
\begin{align}
\begin{aligned}\label{aest-2}
 \int_{\bbr} | \bw| \eta(U|\tilde U_0 ) dx &\le \eps^2 +C\frac{\eps}{\lambda}\int_{v> 3v_*}  |\bw| |v-\tilde v_0| dx\\
&\,+C\frac{\eps}{\lambda}\Big(\int_{v\le 3v_*}  |\bw| |v-\tilde v_0|^2 dx+\int_{\bbr}  |\bw| |h-\tilde h_0|^2 dx \Big)^{1/2}\Big(\int_\bbr  |\bw| dx \Big)^{1/2} \\
&\le \eps^2 +C\frac{\eps}{\lambda}\int_{v> 3v_*} |\bw| |v-\tilde v_0| dx\\
&\quad+\frac{c_1}{2}\int_{v\le 3v_*} |\bw| |v-\tilde v_0|^2 dx+\frac{1}{4}\int_{\bbr} |\bw| |h-\tilde h_0|^2 dx +C \frac{\eps^2}{\lambda}.
\end{aligned}
\end{align}
Now, choosing $\deo$ small enough such that $C \eps/\lam < c_2/2$, and then combining the two estimates \eqref{aest-1} and \eqref{aest-2}, we have 
\begin{align}
\begin{aligned}\label{aesth}
\int_\bbr  |\bw| |h-\tilde h_0 |^2 
 +\int_{v\le 3v_*}|\bw|  |v-\tilde v_0|^2 + \int_{v> 3v_*}|\bw|  |v-\tilde v_0|\leq C \frac{\eps^2}{\lambda}.
\end{aligned}
\end{align}
Applying \eqref{aesth} to \eqref{aest-2}, we have \eqref{genl1}. 
\vskip0.2cm

$\bullet$ {\it Proof of \eqref{genyabs}} :
We first use \eqref{2dere} and the definition of $\eta(\cdot|\cdot)$ to rewrite ${\bf Y}(U)$ as
\begin{align*}
{\bf Y} (U)= \underbrace{ \int_\bbr \bw \Big(\frac{|h-\tilde h_0|^2}{2} + Q(v|\tilde v_0) \Big) dx }_{=:J_1} + \underbrace{ \int_\bbr {\bf a} \Big(-\bv_1 p'(\tilde v_0)(v-\tilde v_0) + \bv_2 (h-\tilde h_0) \Big) dx}_{=:J_2}.
\end{align*}
It follows from \eqref{genl1} that
\[
|J_1| \le C \frac{\eps^2}{\lambda}.
\]
As done before, we have
\begin{align*}
\begin{aligned}
|J_2| &\le  C\frac{\eps}{\lambda}\int_\bbr  |\bw| (|v-\tilde v_0| +|h-\tilde h_0 | ) dx\\
&\le C\frac{\eps}{\lambda}\int_{v> 3v_*}  |\bw| |v-\tilde v_0 | dx\\
&\quad+C\frac{\eps}{\lambda}\Big(\int_{v\le 3v_*}  |\bw| |v-\tilde v_0|^2 dx+\int_{\bbr}  |\bw| |h-\tilde h_0|^2 dx \Big)^{1/2}\Big(\int_\bbr  |\bw| dx \Big)^{1/2}\\
&\le C\frac{\eps}{\lambda}\int_\bbr  |\bw| Q(v|\tilde v_0) dx
+C\frac{\eps}{\sqrt\lambda}\Big(\int_\bbr  |\bw|\big(Q(v|\tilde v_0) + |h-\tilde h_0|^2 \big) dx \Big)^{1/2}
\le C \frac{\eps^2}{\lambda}.
\end{aligned}
\end{align*}
\end{proof}

\subsection{Estimates outside truncation}\label{sec:abo}
The following proposition is a generalization of \cite[Lemmas 4.5-4.8]{KV-unique19}.

\begin{proposition}\label{aprop:out}
For a given constant $U_*:=(v_*,u_*)\in \bbr^+\times\bbr$, there exist positive constants $\delta_0,\delta_1$ and $C$ (depending on $\delta_1$) such that for any $\eps, \lam>0$ satisfying $\eps/\lambda< \delta_0$ and $\lambda<\delta_0$, the following holds.\\
Let $\tilde U_0:=(\tiv_0,\tih_0): \bbr\to  \bbr^+\times\bbr$ and $\bw:\bbr\to\bbr$ be any functions such that
\begin{align}
&|\tilde U_0(x)-U_*|\le C\delta_0,\quad \tiv_0(x)\ge C^{-1}, \quad\forall x\in\bbr, \label{aasunif} \\
& |(\tiv_0)_x|\le \delta_0^2, \quad  |{\bf w}(x)|\le C \eps\lam \exp\big(-C^{-1} \eps |x|\big), \quad\forall x\in\bbr, \label{aasup} \\
&  \inf_{-\eps^{-1}\le x\le \eps^{-1}} |\bw(x)| \ge C\eps\lam,\quad \int_\bbr |\bw| dx =\lam. \label{aasinte}
\end{align}
Let  $U:=(v,h): \bbr\to  \bbr^+\times\bbr$ be any function such that 
\beq\label{agenl1}
 \int_\bbr |\bw| |h-\tih_0|^2\,dx + \int_\bbr |\bw| Q(v|\tiv_0)\,dx \le C\frac{\eps^2}{\lambda}.
 \eeq
Let $\bar \bv$ be a $\deltao$-truncation of $v$ defined  by (well-defined since the function $p$ is one to one)
\beq\label{atrunc-def}
p(\bar \bv)-p(\tilde v_0)=\bar \psi\big(p(v)-p(\tilde v_0)\big),\qquad\mbox{where}\quad \bar \psi (y)=\inf\left(\deltao,\sup(-\deltao,y)\right).
\eeq
Let ${\bf\Omega}:=\{x~|~ p(v)-p(\tilde v_0)  \le \deltao \}$, and
\begin{align*}
\begin{aligned}
&{\bf G}_1^{-}(U):=\int_{{\bf\Omega}^c} |\bw| \left| h-\tilde h_0 \right|^2  dx,\\
& {\bf G}_2(U):= \int_\bbr  |\bw|   Q(v|\tilde v_0)  dx,\\
&{\bf \tilde G}_2(U):= \int_\bbr  |(\tiv_0)_x|   Q(v|\tilde v_0)  dx,\\
&{\bf D}(U):= \int_\bbr   v^\beta |\partial_x \big(p(v)-p(\tilde v_0 )\big)|^2 dx .
\end{aligned}
\end{align*}  
Then, the following estimates hold, where $\bar U:= (\bar\bv, h)$.
\begin{eqnarray}
\label{abig1}
&& \int_{{\bf\Omega}} |\bw| \big| p(v)-p(\bar \bv) \big|^2 dx + \int_\bbr|\bw| \big| p(v)-p(\bar \bv) \big| dx  \le \sqrt{\frac{\eps}{\lambda}}{\bf D}(U),\\
\label{abig2}
&&  \int_{{\bf\Omega}} |\bw| \Big | |p(v)-p(\tiv_0)|^2-  |p(\bar\bv)-p(\tiv_0)|^2\Big|  \,dx \le \sqrt{\frac{\eps}{\lambda}}{\bf D}(U),
\end{eqnarray}
\begin{eqnarray}
\label{al3}
&& \int_\bbr |\bw|^2  v^\beta  |p(v)-p(\bar\bv)|^2 dx + \int_\bbr  |\bw|^2 v^\beta  |p(v)-p(\bar\bv)| \,dx \\
\nonumber
&&\qquad\le C\lambda^2 \left( {\bf D}(U) +\delta_0 {\bf \tilde G}_2(U) \right) +C\eps\lam {\bf G}_2(U),\\
\label{al4}
&&  \int_\bbr|\bw|^2 \Big |v^\beta |p(v)-p(\tiv_0)|^2- \bar\bv^\beta |p(\bar\bv)-p(\tiv_0)|^2\Big| \,dx\\
\nonumber
&&\qquad \le C\lambda^2 \left( {\bf D}(U) +\delta_0  {\bf \tilde G}_2(U) \right)+ C\eps\lam {\bf G}_2(U),
\end{eqnarray}
\begin{eqnarray}
\label{al5}
&&  \int_\bbr|\bw| \left|p(v|\tilde{v}_0)-p(\bar\bv|\tilde{v}_0)\right | \,dx\\
\nonumber
&&\qquad \leq C\sqrt{\frac{\eps}{\lambda}} \left({\bf D}(U) +\delta_0  {\bf \tilde G}_2(U) \right) +C\left({\bf G}_2(U)-{\bf G}_{2}(\bar U) \right),\\
\label{al50}
&&   \int_\bbr|\bw| \left|Q(v|\tilde{v}_0)-Q(\bar\bv|\tilde{v}_0)\right | \,dx+\int_\bbr|\bw| |v-\bar\bv| \,dx \leq  C\left({\bf G}_{2}(U)-{\bf G}_{2}(\bar U) \right),
\end{eqnarray}
\begin{eqnarray}
\label{newq}
&&\int_{{\bf\Omega}^c} |\bw| \big|p(v)-p(\bar \bv) \big|^2 dx \le C\frac{\eps^{2-q}}{\lam} \left({\bf D}(U) +\delta_0  {\bf \tilde G}_2(U) \right)^q, \quad q:=\frac{2\gamma}{\gamma +\alpha}, \\
\label{ans1}
&& \int_{{\bf\Omega}^c} |\bw| \big |p(v)-p(\tiv_0)\big| |h-\tilde h_0| dx  \le \delta_0 \left({\bf D}(U) +\delta_0 {\bf \tilde G}_2(U) \right) + (\deltao+C\deltaz)  {\bf G}_{1}^-(U),\\
\label{ans2}
&&\int_{{\bf\Omega}^c} |\bw| \left( Q(\bar\bv|\tilde v_0) +|\bar\bv -\tilde v_0 | \right)  dx \le C \sqrt{\frac{\eps}{\lambda}} \left({\bf D}(U) +\delta_0 {\bf \tilde G}_2(U) \right),
\end{eqnarray}
\begin{eqnarray}
\label{al7}
&&  \int_\bbr|\bw|^2  \frac{|v^\beta-\bar\bv^\beta|^2}{v^\beta} \,dx\le C\lambda  \bigg({\bf D}(U) +\delta_0 {\bf \tilde G}_2(U) \bigg) +C\lam\left({\bf G}_{2}(U)-{\bf G}_{2}(\bar U)\right) ,\\
\label{al8}
&&  \int_\bbr|\bw|^2  \Big| \frac{|v^\beta-\tilde v_0^\beta|^2}{v^\beta} - \frac{|\bar\bv^\beta-\tilde v_0^\beta|^2}{\bar\bv^\beta} \Big| \,dx\\
\nonumber
&&\qquad \le C \lambda  \bigg({\bf D}(U) +\delta_0 {\bf \tilde G}_2(U) \bigg) +C\lam\left({\bf G}_{2}(U)-{\bf G}_{2}(\bar U)\right)  .
\end{eqnarray}
\end{proposition}

\begin{proof}
First, let $\bar\bv_s$ and $\bar \bv_b$ be one-sided truncations of $v$ defined by
\beq\label{schi}
p(\bar \bv_s)- p(\tilde v_0) :=\bar \psi_s \big(p(v)- p(\tilde v_0)\big),\qquad\mbox{where}\quad \bar\psi_s (y)=\inf (\deltao,y),
\eeq
and
\beq\label{bchi}
p(\bar \bv_b)- p(\tilde v_0) :=\bar \psi_b \big(p(v)- p(\tilde v_0)\big),\qquad\mbox{where}\quad \bar\psi_b (y)=\sup (-\deltao,y).
\eeq
Notice that the function $\bar \bv_s$ (resp. $\bar \bv_b$) represents the truncation of small (resp. big) values of $v$ corresponding to $|p(v)-p(\tilde v_0)|\ge\deltao$.\\
By comparing the definitions of \eqref{atrunc-def}, \eqref{schi} and \eqref{bchi}, we see
\beq\label{acompare1}
\bar\bv = \left\{ \begin{array}{ll}
    \bar\bv_b,\quad &\mbox{on}~ {\bf\Omega}, \\
    \bar\bv_s,\quad &\mbox{on}~ {\bf\Omega}^c ,
       \end{array} \right.
\eeq
and
\begin{align*}
\begin{aligned}
& \left( p(\bar \bv_s)-p(\tilde v_0) \right) {\mathbf 1}_{\{p(v)-p(\tilde v_0) \ge -\deltao\}} = \left( p(\bar \bv)-p(\tilde v_0) \right) {\mathbf 1}_{\{p(v)-p(\tilde v_0) \ge -\deltao\}}  ,\\
& \left( p(\bar \bv_b)-p(\tilde v_0) \right) {\mathbf 1}_{\{p(v)-p(\tilde v_0) \le \deltao\}} = \left( p(\bar \bv)-p(\tilde v_0)\right) {\mathbf 1}_{\{p(v)-p(\tilde v_0) \le \deltao\}}  .
\end{aligned}
\end{align*}
We also note that 
\begin{equation}\label{adef-bar}
\begin{array}{rl}
p(v)-p(\bar \bv_s)=& (p(v)-p(\tilde v_0))+(p(\tilde v_0)-p(\bar \bv_s)) \\[0.2cm]
=&\left(I-\bar\psi_s\right)\left(p(v)-p(\tilde v_0)\right) \\[0.2cm]
=& \left(\left(p(v)-p(\tilde v_0) \right)-\deltao\right)_+,\\[0.2cm]
p(\bar \bv_b)-p(v)=& (p(\bar \bv_b)-p(\tilde v_0))+(p(\tilde v_0)-p(v)) \\[0.2cm]
=&\left(\bar\psi_b-I\right)\left(p(v)-p(\tilde v_0)\right) \\[0.2cm]
=& \left(-\left(p(v)-p(\tilde v_0)\right)-\deltao\right)_+,\\
|p(v)-p(\bar \bv)|=& |(p(v)-p(\tilde v_0))+(p(\tilde v_0)-p(\bar \bv))|\\[0.2cm]
=&|(I-\bar\psi)(p(v)-p(\tilde v_0))|\\[0.2cm]
=& (|p(v)-p(\tilde v_0)|-\deltao)_+.
\end{array}
\end{equation}
Therefore, using \eqref{atrunc-def}, \eqref{bchi}, \eqref{schi} and \eqref{adef-bar}, we have
\begin{align}
\begin{aligned}\label{aeq_D}
{\bf D}(U)&=\int_\bbr  v^{\beta}  |\partial_x (p(v)-p(\tilde v_0))|^2 dx\\
&=\int_\bbr  v^{\beta} |\partial_x (p(v)-p(\tilde v_0))|^2 ( {\mathbf 1}_{\{|p(v)-p(\tilde v_0) |\leq\deltao\}} + {\mathbf 1}_{\{p(v)-p(\tilde v_0) >\deltao\}}+ {\mathbf 1}_{\{p(v)-p(\tilde v_0) <-\deltao\}}  )dx\\
&={\bf D}(\bar U)+\int_\bbr  v^{\beta} |\partial_x (p(v)-p(\bar \bv_s))|^2 dx+\int_\bbr  v^{\beta}  |\partial_x (p(v)-p(\bar \bv_b))|^2 dx\\
&\ge \int_\bbr  v^{\beta} |\partial_x (p(v)-p(\bar \bv_s))|^2 dx+\int_\bbr v^{\beta}  |\partial_x (p(v)-p(\bar \bv_b))|^2 dx,
\end{aligned}
\end{align}
which also yields 
\beq\label{amonod}
{\bf D}(U) \ge {\bf D}(\bar U).
\eeq
On the other hand, since $Q(v|\tilde{v}_0)\geq  Q(\bar v|\tilde{v}_0)$,
we have
\[
 {\bf G}_{2}(U)-{\bf G}_{2}(\bar U)= \int_\bbr |\bw|\left(Q(v|\tilde{v}_0)-Q(\bar v|\tilde{v}_0)\right)\,dx\geq 0,
\]
which together with  \eqref{genl1} yields
\beq\label{al2}
0\leq {\bf G}_{2}(U)-{\bf G}_{2}(\bar U)\leq  {\bf G}_{2}(U)\leq C \int_\bbr |\bw|Q(v|\tilde{v}_0)\,dx \le
C\frac{\eps^2}{\lambda}.
\eeq

To get the desired estimates, we use the same computations as in the proofs of \cite[Lemmas 4.5, 4.6, 4.7, 4.8]{KV-unique19} by considering $\bw$ the spatial derivative of each weight. 
Indeed, the main ideas for the proofs of \cite[Lemmas 4.5, 4.6, 4.7, 4.8]{KV-unique19} are based on the smallness of the weighted relative entropy \eqref{agenl1}, and the following point-wise estimates \eqref{aest6} and \eqref{abeta1}:\\
Since \eqref{agenl1} and \eqref{aasinte} imply
\begin{eqnarray*}
2\eps\int_{-1/\eps}^{1/\eps} Q(v|\tilde{v}_0)\, dx &\leq& \frac{2\eps}{\inf_{[-\eps^{-1},\eps^{-1}]}|\bw|}\int_\bbr|\bw|Q(v|\tilde{v}_0)\,dx \\
&\leq& C \frac{\eps}{\lambda\eps}\frac{\eps^2}{\lambda}=C\left(\frac{\eps}{\lambda}\right)^2 ,
\end{eqnarray*}
there exists $x_0 \in [-\eps^{-1},\eps^{-1}]$ such that $Q(v|\tilde{v}_0) (x_0) \leq C(\eps/\lambda)^2$.
For $\deo$ small enough, and using  \eqref{pQ-equi0}, we have 
$$
|(p(v)-p(\tiv_0))(x_0)|\leq C\frac{\eps}{\lambda}.
$$
Thus, if $\deo$ is small enough such that $C\eps/\lambda\leq \deltao/2$, then we find from \eqref{atrunc-def} that
\[
|(p(v)-p(\bar \bv))(x_0)| =0,
\]
which together with \eqref{adef-bar} implies
$$
|(p(v)-p(\bar \bv_b))(x_0)|=0,\qquad |(p(v)-p(\bar \bv_s))(x_0)|=0.
$$
Therefore, using \eqref{aeq_D}, we find
\begin{align}
\begin{aligned}\label{aest6}
\forall x \in \bbr,\quad |(p(v)-p(\bar \bv_b))(x)|  &\le \int_{x_0}^x \left|\partial_y \big( p(v)-p(\bar \bv_b)\big) \right|  {\mathbf 1}_{\{p(v)-p(\tiv_0) < -\deltao\}} \,dy \\
 &\le C\int_{x_0}^x v^{\beta/2}\left|\partial_y \big( p(v)-p(\bar \bv_b)\big) \right| {\mathbf 1}_{\{p(v)-p(\tiv_0) < -\deltao\}}  \,dy \\
 &\le C\sqrt{|x|+\frac{1}{\eps}}\sqrt{{\bf D}(U)}.
 \end{aligned}
\end{align}
To get a point-wise estimate for $|v^{\beta/2} (p(v)-p(\bar \bv_s))(x)|$, we use
\[
|v^{\beta/2}(p(v)-p(\bar \bv_s))(x)|=\left|\int_{x_0}^x \partial_y \big(v^{\beta/2}(p(v)-p(\bar \bv_s))\big)\,dy\right|.
\]
To control the right-hand side by the good terms, we observe that 
since $v^{\beta/2}=p(v)^{-(\gamma-\alpha)/2\gamma}$, 
we have
\begin{eqnarray*}
&&\partial_y \big(v^{\beta/2}(p(v)-p(\bar \bv_s))\big)= \partial_y \big( p(v)^{-(\gamma-\alpha)/2\gamma} (p(v)-p(\bar \bv_s)) \big)\\
&&\qquad =  p(v)^{-(\gamma-\alpha)/2\gamma}\partial_y (p(v)-p(\bar \bv_s))\\
&&\qquad\quad -\frac{\gamma-\alpha}{2\gamma}p(v)^{-(\gamma-\alpha)/2\gamma} \frac{p(v)-p(\bar \bv_s)}{p(v)}\partial_y \big[ (p(v)-p(\tiv_0))+p(\tiv_0) \big]\\
&&\qquad = v^{\beta/2} \partial_y (p(v)-p(\bar \bv_s)) -\frac{\gamma-\alpha}{2\gamma}v^{\beta/2} \underbrace{ \frac{p(v)-p(\bar \bv_s)}{p(v)}  \partial_y (p(v)-p(\tiv_0))}_{=:K}\\
&&\qquad\quad -\frac{\gamma-\alpha}{2\gamma}v^{\beta/2}  \frac{p(v)-p(\bar \bv_s)}{p(v)} \partial_y p(\tilde v_0).
\end{eqnarray*}
Using the fact that 
\[
p(\bar \bv_s)=p(\tiv_0) + \deltao\quad\mbox{and}\quad  \frac{p(v)-p(\bar \bv_s)}{p(v)} \le C\quad \mbox{on } \{p(v)-p(\tiv_0)>\deltao\}, 
\]
we have
\[
K= \frac{p(v)-p(\bar \bv_s)}{p(v)} {\mathbf 1}_{\{p(v)-p(\tiv_0)>\deltao\}}  \partial_y (p(v)-p(\tiv_0))= \frac{p(v)-p(\bar \bv_s)}{p(v)}  \partial_y (p(v)-p(\bar \bv_s)),
\]
and so,
\[
|K|\le C | \partial_y (p(v)-p(\bar \bv_s))|.
\]
In addition, using $ |\partial_y p(\tiv_0)| = |p'(\tiv_0)| |(\tiv_0)_y | \le C|(\tiv_0)_y |$, we have
\[
|\partial_y \big(v^{\beta/2}(p(v)-p(\bar \bv_s))\big)|\le C v^{\beta/2}(|\partial_y (p(v)-p(\bar \bv_s))| +|(\tiv_0)_y | ).
\]
Therefore, using \eqref{aeq_D}, we have that for any $x\in \bbr$,
\begin{align*}
\begin{aligned}
 |v^{\beta/2} (p(v)-p(\bar \bv_s))(x)|  &=\left|\int_{x_0}^x \partial_y \big(v^{\beta/2} (p(v)-p(\bar \bv_s))\big)\,dy \right| \\
&\le \int_{x_0}^x \big|\partial_y \big(v^{\beta/2} (p(v)-p(\bar \bv_s))\big)\big| {\mathbf 1}_{\{p(v)-p(\tiv_0) > \delta_1\}} dy\\
&\le C \int_{x_0}^x  v^{\beta/2}(|\partial_y ((p(v)-p(\bar \bv_s))| +|(\tiv_0)_y|) {\mathbf 1}_{\{p(v)-p(\tiv_0)  > \delta_1\}} dy\\
 &\le C\sqrt{|x|+\frac{1}{\eps}}\left(\sqrt{{\bf D}(U)} + \sqrt{ \int_\bbr  |(\tiv_0)_y|^2 v^{\beta}{\mathbf 1}_{\{p(v)-p(\tiv_0) > \deltao\}} dy } \right).
\end{aligned}
\end{align*}
Using the condition $\beta=\gamma-\alpha>0$, we have
\begin{align*}
\begin{aligned}
 \int_\bbr |(\tiv_0)_x|^2 v^{\beta}{\mathbf 1}_{\{p(v)-p(\tiv_0)> \delta_1\}} dx  &= \int_\bbr |(\tiv_0)_x|^2 \frac{v^{\beta}}{|v-\tilde v_0|^2}  |v-\tilde v_0|^2 {\mathbf 1}_{\{p(v)-p(\tiv_0) > \delta_1\}} dx\\
&\le  C \int_\bbr |(\tiv_0)_x|^2  |v-\tilde v_0|^2 {\mathbf 1}_{\{p(v)-p(\tiv_0) > \delta_1\}} dx.
\end{aligned}
\end{align*}
In addition, since $\{p(v)-p(\tiv_0) > \delta_1\} =\{v\le C\}$ for some constant $C$, \eqref{aasup} and \eqref{rel_Q} yield
\[
 \int_\bbr  |(\tiv_0)_x|^2 |v-\tilde v_0|^2{\mathbf 1}_{\{p(v)-p(\tiv_0) > \delta_1\}} d\xi\le C \delta_0^2  \int_\bbr  |(\tiv_0)_x| Q(v|\tilde v_0) dx = C \delta_0^2{\bf \tilde G}_2(U) .
\]
Therefore we obtain that  
\begin{align}
\begin{aligned}\label{abeta1}
\forall i=1, 2,\, \forall x \in \bbr,\quad  |v^{\beta/2} (p(v)-p(\bar \bv_s))(x)| \le C\sqrt{|x|+\frac{1}{\eps}}\bigg(\sqrt{{\bf D}(U)} + \delta_0 \sqrt{{\bf \tilde G}_2(U)} \bigg).
\end{aligned}
\end{align}

 The remaining parts use the same computations as in the proofs of \cite[Lemmas 4.5, 4.6, 4.7, 4.8]{KV-unique19}.
 Especially, for the estimate \eqref{newq}, note from \eqref{acompare1} that 
 \[
 \int_{{\bf\Omega}^c} |\bw| \big|p(v)-p(\bar \bv) \big|^2 dx =\int_{{\bf\Omega}^c} |\bw| \big|p(v)-p(\bar \bv_s) \big|^2 dx,
 \]
 and so, its proof follows from the proof of \cite[Lemmas 4.7]{KV-unique19}.
  We omit those details.
\end{proof}

The following proposition is a generalization of \cite[Proposition 4.3]{KV-unique19}.

\begin{proposition}\label{aprop_out}
For a given constant $U_*:=(v_*,u_*)\in \bbr^+\times\bbr$, there exist positive constants $\delta_0,\delta_1, \s, C^*$ and $C$  (in particular, $C$ depends on the constant $\delta_1$, but $C$ and $C^*$ are independent of $\deo$)  such that for any $\eps, \lam>0$ satisfying $\eps/\lambda< \delta_0$ and $\lambda<\delta_0$, the same hypotheses as in Proposition \ref{aprop:out} hold. In addition, let $\bv_1, \bv_2 :\bbr\to\bbr$ be any functions such that
\beq\label{cweight}
|\bv_i|\le C \frac{\eps}{\lam}|\bw(x)|, \quad\forall x\in\bbr,\quad i=1,2,
\eeq
and assume
\beq\label{g2a}
{\bf\tilde G}_{2}(U) \le \frac{\eps^2}{\lam}.
\eeq
Consider the following functionals:
\begin{align*}
\begin{aligned}
&{\bf B}_{1}(U):= \s\int_\bbr \bv_1 p(v|\tilde v_0) dx,\\
&{\bf B}_{2}^- (U):=\int_{{\bf\Omega}^c} \bw  \big(p(v)-p(\tilde v_0 )\big) \big(h-\tilde h_0 \big)dx , \qquad {\bf B}_{2}^+(U) := \frac{1}{2\s}\int_{\bf\Omega} \bw |p(v)-p(\tilde v_0) |^2  dx,\\
&{\bf B}_{3}(U):=- \int_\bbr \bw  v^\beta \big(p(v)-p(\tilde v_0 )\big)\partial_x \big(p(v)-p(\tilde v_0 )\big)  dx,\\
&{\bf B}_{4}(U):=\int_\bbr |\bw| |\bv_1| \big|p(v)-p(\tilde v_0)\big| \big|v^\beta - \tilde v_0^\beta\big|  dx,\\
&{\bf B}_5(U):= \int_\bbr  |\bv_1| \big| \partial_x \big(p(v)-p(\tilde v_0)\big) \big| \big|v^\beta - \tilde v_0^\beta \big| dx,
\end{aligned}
\end{align*}
\begin{align*}
\begin{aligned}
{\bf Y}^g (U)&:= -\frac{1}{2\s^2}\int_{\bf \Omega} \bw |p(v)-p(\tilde v_0)|^2 dx -\int_{\bf\Omega} \bw Q(v|\tilde v_0) dx \\
&\qquad -\int_{\bf\Omega}  \bv_1 p'(\tiv) (v-\tilde v_0) dx +\frac{1}{\s}\int_{\bf\Omega} \bv_2 \big(p(v)-p(\tilde v_0)\big)dx,\\
{\bf Y}^s (U)&:=-\int_{{\bf\Omega}^c} \bw Q(v|\tilde v_0) dx -\int_{{\bf\Omega}^c} \bv_1 p'(\tiv_0)(v-\tilde v_0)dx 
 -\int_{{\bf\Omega}^c} \bw \frac{|h-\tilde h_0|^2}{2} dx+\int_{{\bf\Omega}^c} \bv_2 (h-\tilde h_0) dx,\\
 {\bf Y}^b(U)&:= -\frac{1}{2}\int_{\bf\Omega}  \bw  \Big(h-\tilde h_0-\frac{p(v)-p(\tilde v_0)}{\sigma}\Big)^2 dx -\frac{1}{\s} \int_{\bf\Omega}  \bw  \big(p(v)-p(\tilde v_0)\big)\Big(h-\tilde h_0-\frac{p(v)-p(\tilde v_0)}{\sigma}\Big) dx,\\
{\bf Y}^l(U) &:=\int_{\bf\Omega} \bv_2 \Big(h-\tilde h_0-\frac{p(v)-p(\tilde v_0)}{\sigma}\Big)dx,
\end{aligned}
\end{align*}
and 
\[
{\bf G}_{1}^{+}(U):= \int_{\bf\Omega} |\bw|  \left|  h-\tilde h_0 -\frac{p(v)-p(\tilde v_0)}{\s}\right|^2 dx .
\]
Then, the following estimates hold:
\begin{align}
\begin{aligned} \label{an1}
&|{\bf B}_{1}(U)-{\bf B}_{1}(\bar U)| \leq C\deo \left( {\bf D}(U) +  {\bf \tilde G}_2(U) + \left({\bf G}_{2}(U)-{\bf G}_{2}(\bar U) \right)\right) ,\\
&|{\bf B}_{2}^-(U)|  \le\deo \left({\bf D}(U) + {\bf \tilde G}_2(U)   \right) + (\deltao+C\delta_0)  {\bf G}_{1}^-(U),\\
&|{\bf B}_{2}^+(U)-{\bf B}_{2}^+(\bar U)|   \le \sqrt{\deo}{\bf D}(U),
\end{aligned}
\end{align}
\beq\label{aib12}
|{\bf B}_{1}(\bar U)|+|{\bf B}_{2}^+(\bar U)| \le C\int_\bbr |\bw | Q(\bar \bv|\tilde{v}_0)\,dx \le C^*\frac{\eps^2}{\lambda},
\eeq
\begin{align}
\begin{aligned} \label{an3}
&|{\bf B}_{3}(U)|+|{\bf B}_{4}(U)| + |{\bf B}_{5}(U)|  \le C\deo \left( {\bf D}(U) +  {\bf \tilde G}_2(U) + \left({\bf G}_{2}(U)-{\bf G}_{2}(\bar U) \right)+\frac{\eps}{\lambda} {\bf G}_{2}(\bar U) \right).
\end{aligned}
\end{align}
If, in addition, ${\bf D}(U)\leq \frac{4C^*}{\sqrt\deltaz} \frac{\eps^2}{\lambda}$, then
\begin{align}
\begin{aligned}\label{am1}
& |{\bf Y}^g(U)-{\bf Y}^g(\bar U)|^2 +|{\bf Y}^b(U)|^2 +|{\bf Y}^l(U)|^2+|{\bf Y}^s(U)|^2  \\
&\le C\frac{\eps^2}{\lambda} \bigg( \sqrt{\deo}{\bf D}(U)+\left({\bf G}_{2}(U)- {\bf G}_{2} (\bar U) \right) +\deo{\bf\tilde G}_{2}(U) +\left(\frac{\eps}{\lambda}\right)^{1/4} {\bf G}_{2}(\bar U) \\
&\qquad\qquad+{\bf G}_{1}^-(U) + \left(\frac{\lambda}{\eps}\right)^{1/4} {\bf G}_{1}^+(U) \bigg).
\end{aligned}
\end{align}
\end{proposition}

\begin{proof}
\noindent$\bullet$ {\it Proof of \eqref{an1}} :
First, using \eqref{al5} with \eqref{cweight}, and  \eqref{ans1}, we have
\begin{align*}
\begin{aligned} 
&|{\bf B}_{1}(U)-{\bf B}_{1}(\bar U)| \leq C\frac{\eps}{\lambda} \left( {\bf D}(U) +  {\bf \tilde G}_2(U) + \left({\bf G}_{2}(U)-{\bf G}_{2}(\bar U) \right)  \right) ,\\
&|{\bf B}_{2}^-(U)|  \le\deo \left( {\bf D}(U) +  {\bf \tilde G}_2(U)  \right) + (\deltao+C\delta_0)  {\bf G}_{1}^-(U).
\end{aligned}
\end{align*}
Using  \eqref{abig2}, we have
\[
|{\bf B}_{2}^+(U) - {\bf B}_{2}^+(\bar U)|  =\int_{\bf\Omega} |\bw| \left|| p(v)-p(\tiv_0) |^2 - |p(\bar \bv)-p(\tiv_0)|^2 \right| dx  \le \sqrt\frac{\eps}{\lambda} {\bf D}(U).
\]
Thus, for any $\eps, \lam$ satisfying $\eps/\lam<\delta_0$, we have the desired estimates.\\

\noindent$\bullet$ {\it Proof of \eqref{aib12}}:
Using \eqref{p-est1}, \eqref{pQ-equi0} and \eqref{agenl1}, we have
\beq\label{ab2ip}
|{\bf B}_{1}(\bar U)|+|{\bf B}_{2}^+(\bar U)| \leq C\int_\bbr |\bw | Q(\bar \bv|\tilde{v}_0)\,dx \leq C\int_\bbr |\bw | Q( v|\tilde{v}_0)\,dx  \le C\frac{\eps^2}{\lambda}.
\eeq

\noindent$\bullet$ {\it Proof of \eqref{an3}}:
For ${\bf B}_{3}$, Young's inequality implies
\[
|{\bf B}_{3}(U)|\le \delta_0 {\bf D}(U) + \frac{C}{ \delta_0} \underbrace{ \int_\bbr|\bw|^2 v^\beta  |p(v)-p(\tilde v_0)|^2 \,dx}_{=:{J}_{1}(U)}. 
\]
Note from \eqref{al4} that
\[
|{J}_{1}(U)-{J}_{1} (\bar U)|\le C\lambda^2 \left( {\bf D}(U) +\delta_0 {\bf \tilde G}_2(U) \right) +C\eps\lam {\bf G}_2(U) .
\]
For ${\bf B}_{4}$ and ${\bf B}_{5}$, we first have
\begin{align*}
\begin{aligned}
&|{\bf B}_{4}(U)|\le C{J}_{1}(U) + C
 \underbrace{\int_\bbr  |\bv_1|^2 \frac{|v^\beta-\tilde v_0^\beta|^2}{v^\beta} \,d\xi}_{=: J_2(U)} ,\\
&|{\bf B}_5(U)|\le \delta_0 {\bf D}(U) +   \frac{C}{\delta_0} J_2(U)  . 
\end{aligned}
\end{align*}
Note that \eqref{cweight} and \eqref{al8} yield
\begin{align*}
|{J}_{2}(U)-{J}_{2} (\bar U)| &\le \left(\frac{\eps}{\lam}\right)^2 \int_\bbr |\bw|^2  \Big| \frac{|v^\beta-\tilde v_0^\beta|^2}{v^\beta} - \frac{|\bar\bv^\beta-\tilde v_0^\beta|^2}{\bar\bv^\beta} \Big| \,dx\\
& \le C\frac{\eps^2}{\lam}  \bigg({\bf D}(U) +\delta_0 {\bf \tilde G}_2(U)  +\left({\bf G}_{2}(U)-{\bf G}_{2}(\bar U)\right)\bigg)   .
\end{align*}
Therefore, using $\lambda<\delta_0$ and $\eps/\lam<\deo$, we have
\begin{align*}
\begin{aligned}
 & |{\bf B}_{3}(U)|+|{\bf B}_{4}(U)| +|{\bf B}_{5}(U)| \\
&\quad \le    \frac{C}{ \delta_0} \left(|{J}_1 (\bar U)|+  |{J}_2 (\bar U)| \right)    
+C\deo \left( {\bf D}(U) +  {\bf \tilde G}_2(U) + \left({\bf G}_{2}(U)-{\bf G}_{2}(\bar U) \right)\right).
\end{aligned}
\end{align*}
For the remaining terms, we first use the assumptions \eqref{aasup} and \eqref{cweight} to have
\begin{align*}
 \frac{C}{ \delta_0} \left(|{J}_1 (\bar U)|+  |{J}_2 (\bar U)| \right)  &\le 
\frac{C}{\delta_0} \bigg[  \eps\lambda \int_\bbr|\bw| \bar \bv^\beta  |p(\bar \bv)-p(\tilde v_0)|^2 \,dx+\frac{\eps^3}{\lambda} \int_\bbr|\bw| \frac{|\bar \bv^\beta-\tilde v_0^\beta|^2}{\bar \bv^\beta} \,dx  \bigg] \\
&\le C \delta_0\frac{\eps}{\lambda}  \bigg[  \int_\bbr|\bw| \bar \bv^\beta  |p(\bar \bv)-p(\tilde v_0)|^2 \,dx+ \int_\bbr|\bw| \frac{|\bar \bv^\beta-\tilde v_0^\beta|^2}{\bar \bv^\beta} \,dx  \bigg] .
\end{align*}
Using $C^{-1}\le \bar \bv\le C$ and $Q(\bar \bv|\tilde v_0)\ge C|\bar \bv-\tiv_0|^2\ge C|\bar \bv^\beta-\tilde v_0^\beta|^2$, we have
\begin{align*}
\begin{aligned}
 \int_\bbr|\bw| \frac{|\bar \bv^\beta-\tilde v_0^\beta|^2}{\bar \bv^\beta} \,dx
 \le C \int_\bbr|\bw| Q(\bar\bv|\tilde v_0) \,dx \le C {\bf G}_{2}(\bar U),
\end{aligned}
\end{align*}
Using \eqref{pQ-equi0} with $ |p(\bar \bv)-p(\tilde v_0)|\le\delta_1$, we have
\beq\label{b2ii}
 \int_\bbr|\bw| \bar \bv^\beta  |p(\bar \bv)-p(\tilde v_0)|^2 \,dx \le  C \int_\bbr|\bw| Q(\bar\bv|\tilde v_0) \,dx \le C {\bf G}_{2}(\bar U). 
\eeq
Hence we have 
\begin{align*}
\begin{aligned}
 & |{\bf B}_{3}(U)|+|{\bf B}_{4}(U)| +|{\bf B}_{5}(U)| \\
&\quad \le C\deo \left( {\bf D}(U) +  {\bf \tilde G}_2(U) + \left({\bf G}_{2}(U)-{\bf G}_{2}(\bar U) \right)+\frac{\eps}{\lambda} {\bf G}_{2}(\bar U) \right).
\end{aligned}
\end{align*}

\noindent$\bullet$ {\it Proof of \eqref{am1}}: We split the proof into three steps.
\vskip0.2cm
\noindent{\it Step 1:} 
First of all, we will use notations $Y_{1}^s, Y_{2}^s, Y_{3}^s$ and $Y_{4}^s$ for the terms of ${\bf Y}^s $ as follows : set ${\bf Y}^s =Y_{1}^s +Y_{2}^s + Y_{3}^s +Y_{4}^s$ where
\begin{align*}
&Y_{1}^s:=-\int_{{\bf \Omega}^c} \bw Q(v|\tilde v_0) dx,\qquad Y_{2}^s := -\int_{{\bf \Omega}^c} \bv_1p'(\tiv_0) (v-\tilde v_0) dx,\\
&Y_{3}^s := -\int_{{\bf\Omega}^c} \bw \frac{|h-\tilde h_0|^2}{2} dx,\qquad Y_{4}^s:= \int_{{\bf\Omega}^c} \bv_2 (h-\tilde h_0) dx.
\end{align*}
Using \eqref{abig1}, \eqref{al50} with \eqref{aasunif} and \eqref{cweight}, and \eqref{abig2}, we find that
\begin{align}
\begin{aligned}\label{aYgd}
&|{\bf Y}^g(U)-{\bf Y}^g(\bar U)| +|Y_{1}^s(U)-Y_{1}^s(\bar U)| +|Y_{2}^s(U)-Y_{2}^s(\bar U)|\\
&\quad \le C \int_{\bf\Omega} |\bw| \big| |p(v)-p(\tiv_0)|^2-|p(\bar \bv)-p(\tiv_0)|^2\big| dx + C \int_{\bf\Omega} |\bw| |p(v)-p(\bar \bv)| dx \\
&\quad\quad +C \int_\bbr |\bw|  \Big(\big|Q(v|\tilde{v}_0)-Q(\bar \bv|\tilde{v}_0)\big| +|v-\bar \bv|\Big)\,dx \\
&\quad\leq C\sqrt{\frac{\eps}{\lambda}}{\bf D}(U)+C\left({\bf G}_{2}(U)-{\bf G}_{2}(\bar U) \right).
\end{aligned}
\end{align}
Note that \eqref{ans2} with \eqref{aasunif} and \eqref{cweight} yields
\beq\label{aY12}
|Y_{1}^s(\bar U)|+|Y_{2}^s(\bar U)|\le C \int_{{\bf\Omega}^c} |\bw| \left( Q(\bar \bv|\tilde v_0) +|\bar \bv -\tilde v_0 | \right)  dx \le C \sqrt{\frac{\eps}{\lambda}} \left({\bf D}(U) +\deo{\bf\tilde G}_{2}(U) \right) .
\eeq
Next, since
\begin{align*}
\begin{aligned}
|Y_{3}^s (U)|+|{\bf Y}^b(U)| &\le C\int_{{\bf\Omega}^c} |\bw| |h-\tih_0|^2\,dx + C\int_{\bf\Omega} |\bw| \left(|h-\tih_0|^2+|p(v)-p(\tiv_0)|^2\right)\,dx\\
&\le C\int_\bbr |\bw| |h-\tih_0|^2\,dx +C |{\bf B}_{2}^+(U)| ,
\end{aligned}
\end{align*}
it follows from \eqref{an1} and \eqref{aib12} that
\[
|Y_{3}^s (U)|+|{\bf Y}^b(U)| \le C \int_\bbr |\bw| |h-\tih_0|^2\,dx + C\sqrt{\deo} {\bf D}(U) + C \frac{\eps^2}{\lambda} .
\]
Therefore, using \eqref{agenl1}, \eqref{al2} and the assumptions \eqref{g2a} and ${\bf D}(U)\leq 2\frac{C^*}{\sqrt\deltaz} \frac{\eps^2}{\lambda}$, together with combining \eqref{aYgd}, \eqref{aY12} and the above estimates, we have
\[
 |{\bf Y}^g(U)-{\bf Y}^g(\bar U)| +|Y_{1}^s(U)|  +|Y_{2}^s(U)| +|Y_{3}^s (U)|+|{\bf Y}^b(U)|  \le C \frac{\eps^2}{\lambda}.
\]
\vskip0.2cm
\noindent{\it Step 2:} 
First of all, using Young's inequality, \eqref{an1} and \eqref{ab2ip}, we have
\begin{align*}
\begin{aligned}
&\left|  \int_{\bf\Omega} \bw \big(p(v)-p(\tilde v_0)\big)\Big(h-\tilde h_0 -\frac{p(v)-p(\tilde v_0)}{\sigma}\Big) dx \right|\\
&\quad \le\left(\frac{\lambda}{\eps}\right)^{1/4}{\bf G}_{1}^+(U)+C\left(\frac{\eps}{\lambda}\right)^{1/4}\int_{\bf\Omega}  |\bw|  |p(v)-p(\tiv_0)|^2\,dx \\
&\quad \le\left(\frac{\lambda}{\eps}\right)^{1/4}{\bf G}_{1}^+(U)+C\left(\frac{\eps}{\lambda}\right)^{1/4}\left({\bf B}_{2}^+(\bar U) + \left( {\bf B}_{2}^+(U)-{\bf B}_{2}^+(\bar U) \right) \right)\\
&\quad \le\left(\frac{\lambda}{\eps}\right)^{1/4}{\bf G}_{1}^+(U)+C\left(\frac{\eps}{\lambda}\right)^{1/4}\left({\bf G}_{2}(\bar U) +\sqrt\deo {\bf D}(U) \right),
\end{aligned}
\end{align*}
which yields
\[
|{\bf Y}^b(U)| \le C\left(\frac{\lambda}{\eps}\right)^{1/4}{\bf G}_{1}^+(U)+C\left(\frac{\eps}{\lambda}\right)^{1/4}\left({\bf G}_{2}(\bar U) +\sqrt\deo {\bf D}(U) \right).
\]
Thus, this and \eqref{aYgd}-\eqref{aY12} together with $|Y_{3}^s (U)|\le C {\bf G}_{1}^-(U)$ imply
\begin{align*}
\begin{aligned}
& |{\bf Y}^g(U)-{\bf Y}^g(\bar U)| +|Y_{1}^s(U)|  +|Y_{2}^s(U)| +|Y_{3}^s (U)|+|{\bf Y}^b(U)|   \\
&\le C \bigg( \sqrt{\deo}{\bf D}(U)+\left({\bf G}_{2}(U)- {\bf G}_{2} (\bar U) \right) +\deo{\bf\tilde G}_{2}(U) +\left(\frac{\eps}{\lambda}\right)^{1/4} {\bf G}_{2}(\bar U) \\
&\qquad\qquad+{\bf G}_{1}^-(U) + \left(\frac{\lambda}{\eps}\right)^{1/4} {\bf G}_{1}^+(U) \bigg).
\end{aligned}
\end{align*}
\vskip0.2cm
\noindent{\it Step 3:} 
For the remaining terms, using H\"older's inequality together with \eqref{cweight} and \eqref{aasinte}, we have
\begin{align*}
\begin{aligned}
&|Y_{4}^s (U)|^2\leq C\left(\frac{\eps}{\lambda}\right)^2 \left(\int_\bbr |\bw| \,dx \right)\int_{{\bf\Omega}^c} |\bw|  |h-\tih_0|^2\,dx \leq C\frac{\eps^2}{\lambda} {\bf G}_{1}^-(U),\\
&|{\bf Y}^l(U)|^2\leq C\left(\frac{\eps}{\lambda}\right)^2 \left(\int_\bbr|\bw| \,dx \right)\int_{\bf\Omega}|\bw|  \left(h-\tih_0 -\frac{p(\bar v)-p(\tiv_0)}{\sigma} \right)^2\,dx \leq C\frac{\eps^2}{\lambda} {\bf G}_{1}^+(U).
\end{aligned}
\end{align*}
Therefore, this together with Step1 and Step2 yields
\begin{align*}
\begin{aligned}
& |{\bf Y}^g(U)-{\bf Y}^g(\bar U)|^2 +|{\bf Y}^b(U)|^2 +|{\bf Y}^l(U)|^2+|{\bf Y}^s(U)|^2  \\
&\le 2 \left( |{\bf Y}^g(U)-{\bf Y}^g(\bar U)| +|Y_{1}^s(U)|  +|Y_{2}^s(U)| +|Y_{3}^s (U)|+|{\bf Y}^b(U)|  \right)^2 + 2 |Y_{4}^s(U)|^2 +|{\bf Y}^l(U)|^2  \\
&\le C\frac{\eps^2}{\lambda} \bigg( \sqrt{\deo}{\bf D}(U)+\left({\bf G}_{2}(U)- {\bf G}_{2} (\bar U) \right) +\deo{\bf\tilde G}_{2}(U) +\left(\frac{\eps}{\lambda}\right)^{1/4} {\bf G}_{2}(\bar U) \\
&\qquad\qquad+{\bf G}_{1}^-(U) + \left(\frac{\lambda}{\eps}\right)^{1/4} {\bf G}_{1}^+(U) \bigg).
\end{aligned}
\end{align*}

\end{proof}

\section{Proof of Proposition \ref{prop:main}}\label{sec:prop}
\setcounter{equation}{0}

This section is dedicated to the proof of Proposition \ref{prop:main}.

\subsection{Smallness of the localized relative entropy}\label{sec:small}
In order to use Proposition \ref{prop:main3} in the proof of Proposition \ref{prop:main}, we need to show that all bad terms on the region $\{|p(v)-p(\tilde{v}^{X_1, X_2})|\geq\delta_1\}$ are absorbed by a very small portion of the good terms. For that, we will crucially use the following lemma on smallness of the relative entropy localized by the space-derivative of each weight $a_i$, under an assumption that the functionals $Y_i$ of \eqref{ybg-first} are bounded below or above. 
The following lemma is analogous to \cite[Lemma 3.2]{Kang-V-NS17}. Notice that the below assumption $(-1)^{i-1} Y_i(U)\le \eps_i^2$ is weaker than the condition $|Y_i(U)|\le \eps_i^2$ of \cite[Lemma 3.2]{Kang-V-NS17}.

\begin{lemma}\label{lemmeC2}
For the given constant $U_*:=(v_*,u_*)\in\bbr^+\times\bbr$, there exist positive constants $\delta_0, C, C_0$ such that the following holds.
Let $\tilde U^{X_1,X_2}$ be the composite wave for the given constant states $U_-, U_m, U_+ \in B_{\delta_0}(U_*)$, where $\eps_1=|p(v_-)-p(v_m)|$ and $\eps_1=|p(v_m)-p(v_+)|$.
Then, for any $\lambda >0$ with $\eps_1/\lambda, \eps_2/\lambda<\delta_0$ and $\lambda<\delta_0$,  the following estimates hold.\\
For each $i=1, 2$, and for all $U\in \mathcal{H}_T$ satisfying $(-1)^{i-1} Y_i(U)\le \eps_i^2$, 
\beq\label{l1}
 \int_\bbr|(a_i)_x^{X_i} ||h-\tilde{h}^{X_1, X_2}|^2\,dx + \int_\bbr|(a_i)_x^{X_i} | Q(v|\tilde{v}^{X_1, X_2})\,dx \le C\frac{\eps_i^2}{\lambda},  \quad \forall t>0,
 \eeq
and
\beq\label{yabs}
|Y_i(U)| \le C_0\frac{\eps_i^2}{\lambda},  \quad \forall t>0.
\eeq
\end{lemma}
\begin{proof}
We will apply Proposition \ref{prop:sm} to the two cases where for each $i=1, 2$, $\bw= -(a_i)_x^{X_i}, \bv_1 = (\tiv_i)_x^{X_i}, \bv_2 =(\tih_i)_x^{X_i}$, ${\bf a}=a^{X_1,X_2}$, $\eps=\eps_i$,  and  $\tilde U_0 = \tilde U^{X_1,X_2}$ for the composite wave $\tilde U^{X_1,X_2}$.\\
First, since $U_-, U_m, U_+ \in B_{\delta_0}(U_*)$, and $\tiv_i, \tih_i$ are monotone (see Remark \ref{rem:vh}), we find that for each $i=1, 2$,
\[
\|\tiv_i^{X_i} -v_*\|_{L^\infty(\bbr)} < \delta_0,\quad \|\tih_i^{X_i} -u_*\|_{L^\infty(\bbr)} < \delta_0,
\]
which implies
\[
\|\tiv^{X_1,X_2} -v_*\|_{L^\infty(\bbr)} \le \|\tiv_1^{X_1} -v_m\|_{L^\infty(\bbr)} +\|\tiv_2^{X_2} -v_*\|_{L^\infty(\bbr)} \le C\eps_1 + \delta_0 \le C\delta_0,
\]
and similarly, 
\[
\|\tih^{X_1,X_2} -u_*\|_{L^\infty(\bbr)} \le C\delta_0 .
\]
Moreover, note that
\[
\tiv^{X_1,X_2}(x) \ge \frac{v_*}{2}, \quad \forall x\in\bbr.
\]
Thus, the assumptions of \eqref{asunif} hold by the above facts and the first inequality of \eqref{d-weight}.
By the properties of the weights (see Section \ref{subsec:a}) and of the waves (see Lemma \ref{lem:vp} and Remark \ref{rem:vh}), the remaining assumptions \eqref{asposi}-\eqref{asinte} also hold.\\
For any fixed $i=1, 2$, notice that
\[
{\bf Y}(U) =- \int_{\bbr} (a_i)_x^{X_i}  \eta(U|\tilde U^{X_1,X_2} ) dx +\int_{\bbr} a^{X_1,X_2} (\tilde U_i)_x^{X_i}  \nabla^2\eta(\tilde U^{X_1,X_2}) (U-\tilde U^{X_1,X_2}) dx = Y_i(U),
\]
and $\sgn(\bw)=-\sgn\big((a_i)_x^{X_i}\big)=(-1)^{i-1}$, and thus,
\[
\sgn(\bw) {\bf Y}(U) =(-1)^{i-1} Y_i(U) .
\]
Hence, Proposition \ref{prop:sm} implies the desired estimates.
\end{proof}

\subsection{Estimates for fixing the size of truncation}
We will apply Proposition \ref{prop:main3} to each of the waves $\tilde v_1$ and $\tilde v_2$ later on. More precisely,
for the weights $a, a_1, a_2$ as in \eqref{weight-a}, and for the Lipschitz functions $\phi_{1,t}$ and $\phi_{2, t}$ as in \eqref{1phi} and \eqref{2phi},  we consider the following functionals: for each $i=1, 2$,
\begin{align}
\begin{aligned}\label{note-in}
&\mathcal{Y}^g_i(v):=-\frac{1}{2\sigma_i^2}\int_\bbr (a_i)_x \phi_{i,t}^2(x) |p(v)-p(\tilde v_i)|^2 dx -\int_\bbr (a_i)_x \phi_{i,t}^2(x) Q(v|\tilde v_i) dx\\
&\qquad\qquad -\int_\bbr a \partial_x p(\tilde v_i) \phi_{i,t}(x) (v-\tilde v_i)dx+\frac{1}{\sigma_i}\int_\bbr a  (\tilde h_i)_x \phi_{i,t}(x) \big(p(v)-p(\tilde v_i)\big)dx,\\
&\mathcal{I}_{1i}(v):= \sigma_i\int_\bbr a (\tilde v_i)_x \phi_{i,t}^2(x)  p(v|\tilde v_i) dx,\\
&\mathcal{I}_{2i}(v):= \frac{1}{2\sigma_i} \int_\bbr (a_i)_x  \phi_{i,t}^2(x) |p(v)-p(\tilde v_i)|^2dx,\\
&\mathcal{G}_{2i} (v):=\sigma_i  \int_\bbr  (a_i)_x \bigg( \frac{1}{2\gamma}  p(\tilde v_i)^{-\frac{1}{\gamma}-1} \phi_{i,t}^2(x) \big(p(v)-p(\tilde v_i)\big)^2\\
&\qquad\qquad - \frac{1+\gamma}{3\gamma^2} p(\tilde v_i)^{-\frac{1}{\gamma}-2} \phi_{i,t}^3(x)\big(p(v)-p(\tilde v_i)\big)^3 \bigg) dx, \\
&\mathcal{D}_i (v):=\int_\bbr a\, v^\beta |\partial_x \big(\phi_{i,t}(x) (p(v)-p(\tilde v_i))\big)|^2 dx,
\end{aligned}
\end{align}
where the functions $a, (a_i)_x, \tilde v_i, (\tilde v_i)_x, (\tilde h_i)_x$ are evaluated at $x-\s_i t - X_i(t)$.\\

However, since the value of $\delta_1$ is itself conditioned to the constant $C_2$ of Proposition \ref{prop:main3}, we should find the bound of $\mathcal{Y}^g_i$ on the unconditional level (for the assumption \eqref{assYp}).

For that, we will first define a truncation on $|p(v)-p(\tilde v^{X_1, X_2})|$ with any $k>0$, and then  the special case  $k=\delta_1$ as in Proposition \ref{prop:main3}. But for now, we consider a general case $k$ to estimate the constant $C_2$.  For that, let $\psi_k$ be a continuous function on $\bbr$ defined by
\beq\label{psi}
 \psi_k(y)=\inf\left(k,\sup(-k,y)\right),\quad k>0.
\eeq
We then define the function $\bar{v}_k$ uniquely (since the function $p$ is one to one) by
\beq\label{trunc-def}
p(\bar v_k)-p(\tilde v^{X_1, X_2})=\psi_k\big(p(v)-p(\tilde v^{X_1, X_2})\big).
\eeq
Notice that $\|p(\bar v_k)-p(\tilde v^{X_1, X_2})\|_\infty\le k$.\\

We have the following lemma from Lemma \ref{lemmeC2}.

\begin{lemma}\label{lem:essy}
For the given constant $U_*:=(v_*,u_*)\in\bbr^+\times\bbr$, there exists positive constants $\delta_0, C_2, k_0$ such that the following holds.
For any $\eps_1, \eps_2, \lam>0$ satisfying $\eps_1/\lambda, \eps_2/\lambda< \delta_0$ and $\lambda<\delta_0$, there exists $t_0>0$ such that the following estimates hold.
\beq\label{lbis}
 |\mathcal{Y}^g_i(\bar v_k)|\leq C_2\frac{\eps_i^2}{\lambda}, \qquad \forall k\leq k_0,\quad \forall t\ge t_0. 
\eeq
\end{lemma}
\begin{proof}
For simplicity, we omit the dependence of the wave and weight on shifts, that is, $\tilde v:=\tilde v^{X_1, X_2}, \tilde v_i:=\tilde v_i^{X_i}$ and $(a_i)_x:=(a_i)_x^{X_i}$, etc.

To get the desired estimate, we will use \eqref{l1}. For that, we need to replace each wave $\tiv_i$ by the composite wave $\tiv$ as follows.\\
First, using \eqref{d-weight}, \eqref{vhre} and $|\phi_{i,t}|\le 1$, we have
\begin{eqnarray*}
&&|\mathcal{Y}^g_i(\bar v_k)|= \bigg|-\frac{1}{2\sigma_i^2}\int_\bbr (a_i)_x \phi_{i,t}^2 |p(\bar v_k)-p(\tilde v_i)|^2 dx -\int_\bbr (a_i)_x \phi_{i,t}^2 Q(\bar v_k|\tilde v_i) dx\\
&&\qquad\qquad\qquad -\int_\bbr a \partial_x p(\tilde v_i) \phi_{i,t} (\bar v_k-\tilde v_i)dx+\frac{1}{\sigma_i}\int_\bbr a  (\tilde h_i)_x \phi_{i,t} \big(p(\bar v_k)-p(\tilde v_i)\big)dx \bigg| \\
&&\qquad\quad\quad \leq C(I_1+I_2+\cdots + I_5),
\end{eqnarray*}
where
\begin{align*}
&I_1:=  \int_\bbr|(a_i)_x| |p(\bar v_k)-p(\tilde v)|^2\,dx,\\
& I_2:= \int_\bbr \frac{\lam}{\eps_i} |(\tiv_i)_x| |p(\tiv)-p(\tiv_i)|^2\,dx,\\
& I_3:=   \int_\bbr|(a_i)_x| Q(\bar v_k|\tilde{v}_i) \,dx, \\
&I_4:= \int_\bbr \frac{\eps_i}{\lambda}|(a_i)_x|\Big(|\bar v_k-\tilde v| + |p(\bar v_k)-p(\tilde v)| \Big)dx,\\
&I_5:= \int_\bbr |(\tiv_i)_x|\Big(|\tiv-\tilde v_i| + |p(\tiv)-p(\tilde v_i)| \Big)dx.
\end{align*}
Choose $k_0\le\delta_*/2$ for $\delta_*$ of Lemma \ref{lem:local}. \\
Then, for any $k\leq k_0$, we have $|p(\bar v_k)-p(\tilde{v})|\leq k\le \frac{\delta_*}{2}$.\\
Thus, using \eqref{pQ-equi0} with $\delta_0$ small enough, we have
\[
I_1\le C \int_\bbr |(a_i)_x| Q(\bar v_k|\tilde{v})\,dx.
\]
For $I_3$, we use Lemma \ref{lem:tri} to have
\[
I_3 =  \int_\bbr|(a_i)_x| \Big( Q(\bar v_k|\tilde{v}) - Q(\tiv_i|\tiv) + \big(Q'(\tiv_i)-Q'(\tiv)\big) (\tiv_i-\bar v_k) \Big) \,dx. 
\]
Using \eqref{Q-est1} and $|\bar v_k|\le C$, we have
\[
I_3\le  \int_\bbr|(a_i)_x| Q(\bar v_k|\tilde{v}) dx + C \int_\bbr|(a_i)_x|  \big( |p(\tiv)-p(\tiv_i)|^2 + |\tiv_i-\tiv| \big) dx =:I_{31}+I_{32}.
\]
Using \eqref{rel_Q} and \eqref{pQ-equi0}, we have
\begin{align*}
\begin{aligned}
I_4&\le \sqrt{ \int_\bbr\left(\frac{\eps_i}{\lambda}\right)^2|(a_i)_x|\,dx} \sqrt{ \int_\bbr|(a_i)_x| \Big(|\bar v_k-\tilde v|^2 + |p(\bar v_k)-p(\tilde v)|^2 \Big)\,dx}\\
&\le C \sqrt{\frac{\eps_i^2}{\lambda}}\sqrt{ \int_\bbr|(a_i)_x| Q(\bar v_k|\tilde{v})\,dx}.
\end{aligned}
\end{align*}
Notice that since the definition of $\bar v_k$ implies either $\tilde v\le \bar v_k\le v$ or $v\le\bar v_k\le \tilde v$, it follows from \eqref{Q-sim} that
\[
Q(v|\tilde{v})\ge Q(\bar v_k|\tilde{v}).
\]
This and \eqref{l1} imply
\[
I_1+I_{31}+I_4 \le C \int_\bbr |(a_i)_x| Q(v|\tilde{v})\,dx + C \sqrt{\frac{\eps_i^2}{\lambda}}\sqrt{ \int_\bbr|(a_i)_x| Q(v|\tilde{v})\,dx} \le  C\frac{\eps_i^2}{\lambda} .
\]
For the remaining terms, we use $\tiv, \tiv_i \in(v_-/2, 2v_-)$ to have
\begin{align*}
I_2+I_{32}+I_5 \le C\frac{\lam}{\eps_i} \int_\bbr |(\tiv_i)_x| |\tiv-\tilde v_i| dx.
\end{align*}
Then, using \eqref{vdid1} of the following lemma together with taking $\deltaz<\eps_0$, we have
\[
I_2+I_{32}+I_5 \le C \lam \exp\Big( -C\min(\eps_1,\eps_2) t \Big) .
\]
We now choose $t_0$ big enough such that 
\[
\lam \exp\Big( -C\min(\eps_1,\eps_2) t_0 \Big) \le C\frac{(\min(\eps_1,\eps_2))^2}{\lambda}.
\]
Then, for all $t\ge t_0$,
\[
I_2+I_{32}+I_5 \le C \frac{\eps_i^2}{\lambda} .
\]
Hence, for some $C_2>0$,
\[
|\mathcal{Y}^g_i(\bar v_k)| \le C_2\frac{\eps_i^2}{\lambda}, \qquad \forall k\leq k_0,\quad \forall t\ge t_0. 
\]
\end{proof}

The following lemma provides inequalities on the interaction of waves, which are useful in the proofs of Lemma \ref{lem:essy}, the estimate \eqref{n6} and Lemma \ref{lem:ws}.

\begin{lemma}\label{lem:sep}
For given $v_->0$ and $u_-\in\bbr$, there exist positive constants $\eps_0, C$ such that for any $\eps_1, \eps_2\in(0,\eps_0)$, the following estimates hold.\\
For each $i=1, 2$,
\begin{align}\label{vdid1}
\int_\bbr |(\tiv_i)_x^{X_i} |  |\tiv^{X_1, X_2}-\tiv_i^{X_i}|   dx \le C \eps_1\eps_2 \exp\Big( -C\min(\eps_1,\eps_2) t \Big) ,\qquad t>0,
\end{align}
and
\begin{align}\label{vdid2}
\int_\bbr |(\tiv_1)_x^{X_1} | |(\tiv_2)_x^{X_2} |   dx \le C \eps_1\eps_2 \exp\Big( -C\min(\eps_1,\eps_2) t \Big) ,\qquad t>0.
\end{align}
\end{lemma}
\begin{proof}
{\it Proof of \eqref{vdid1} } 
We first consider $i=1$. 
By \eqref{tdv1}, there exists $C>0$ such that
\beq\label{vder1}
|(\tiv_1)_x^{X_1} | =|\pa_x\tiv_1(x-\s_1t-X_1(t))|\le C \eps_1^2\exp\big( -C\eps_1|x-\s_1t-X_1(t)| \big),\quad \forall x\in\bbr, ~t>0.
\eeq
Since $\tiv^{X_1,X_2}=\tiv_1^{X_1}+\tiv_2^{X_2}-v_m$ (by \eqref{xwave}), it follows from Lemma \ref{lem:vp} that
\beq\label{vv1}
 |\tiv^{X_1, X_2}-\tiv_1^{X_1}|   = |\tiv_2^{X_2}-v_m| \le 
 \left\{ \begin{array}{ll}
   C \eps_2 \exp\big( -C\eps_2|x-\s_2t-X_2(t)| \big) ,\quad &\mbox{if}~ x\le \s_2 t + X_2(t), \\
  C \eps_2, \quad &\mbox{if} ~  x\ge \s_2 t + X_2(t).
       \end{array} \right.
\eeq
Thus, using the above estimates together with the fact that $\s_2 t + X_2(t)>0$ by \eqref{sx12}, we find
\begin{align*}
|(\tiv_1)_x^{X_1} |^{1/2} |\tiv^{X_1, X_2}-\tiv_1^{X_1}|  \le 
 \left\{ \begin{array}{ll}
   C\eps_1 \eps_2  \exp\big( -C\eps_2|x-\s_2t-X_2(t)| \big)  ,\quad &\mbox{if}~ x\le 0, \\
  C\eps_1 \eps_2 \exp\big( -C\eps_1|x-\s_1t-X_1(t)| \big) , \quad &\mbox{if} ~  x\ge 0.
       \end{array} \right.
\end{align*}
Since \eqref{sx12} implies that
\begin{align}
\begin{aligned}\label{lowsep}
&x\le 0\quad\Rightarrow\quad x-(\s_2 t + X_2(t)) \le x -\frac{\s_2}{2}t \le -\frac{\s_2}{2}t <0,\\
&x\ge 0\quad\Rightarrow\quad x-(\s_1 t + X_1(t)) \ge x -\frac{\s_1}{2}t \ge -\frac{\s_1}{2}t >0,
\end{aligned}
\end{align}
we have
\begin{align*}
|(\tiv_1)_x^{X_1} |^{1/2}  |\tiv^{X_1, X_2}-\tiv_1^{X_1}|  \le  C\eps_1 \eps_2  \exp\big( -C\min(\eps_1,\eps_2) t\big),\quad \forall x\in\bbr, ~t>0.
\end{align*}
Therefore, using 
\beq\label{onemass}
\int_\bbr |(\tiv_1)_x^{X_1} |^{1/2} dx \le C \int_\bbr \eps_1\exp\big( -C\eps_1|x-\s_1t-X_1(t)| \big) dx \le C,
\eeq
we have  \eqref{vdid1} for $i=1$.\\
Likewise, for $i=2$, we use Lemma \ref{lem:vp} to have
\beq\label{vder2}
|(\tiv_2)_x^{X_2} | \le C \eps_2^2\exp\big( -C\eps_2|x-\s_2t-X_2(t)| \big),\quad \forall x\in\bbr, ~t>0,
\eeq
and
\beq\label{vv2}
|\tiv^{X_1, X_2}-\tiv_2^{X_2}|   \le 
 \left\{ \begin{array}{ll}
   C \eps_1  ,\quad &\mbox{if}~ x\le \s_1 t + X_1(t), \\
  C \eps_1\exp\big( -C\eps_1|x-\s_1t-X_1(t)| \big), \quad &\mbox{if} ~  x\ge \s_1 t + X_1(t),
       \end{array} \right.
\eeq
which imply
\begin{align*}
|(\tiv_2)_x^{X_2} |^{1/2}|\tiv^{X_1, X_2}-\tiv_2^{X_2}|   \le 
 \left\{ \begin{array}{ll}
   C\eps_1 \eps_2  \exp\big( -C\eps_2|x-\s_2t-X_2(t)| \big)  ,\quad &\mbox{if}~ x\le 0, \\
  C\eps_1 \eps_2 \exp\big( -C\eps_1|x-\s_1t-X_1(t)| \big) , \quad &\mbox{if} ~  x\ge 0.
       \end{array} \right.
\end{align*}
Therefore we have  the desired estimate for $i=2$.\\
{\it Proof of \eqref{vdid2} } 
First, we use \eqref{tdv1} and \eqref{tdv2} to have
\[
 |(\tiv_1)_x^{X_1}|^{1/2} |(\tiv_2)_x^{X_2} | \le C \eps_1\eps_2^2 \exp\big( -C\eps_1|x-\s_1t-X_1(t)| -C\eps_2|x-\s_2 t-X_2(t)| \big),\quad \forall x\in\bbr, ~t>0.
\]
Then using \eqref{lowsep}, we have
\[
 |(\tiv_1)_x^{X_1}|^{1/2} |(\tiv_2)_x^{X_2} | \le C \eps_1\eps_2^2 \exp\big( -C\min(\eps_1,\eps_2) t\big), \quad \forall x\in\bbr, ~t>0.
\]
Therefore, using \eqref{onemass}, we have \eqref{vdid2}.
\end{proof}

\subsection{Estimates for big values of $|p(v)-p(\tilde v^{X_1, X_2})|$}\label{section-finale}
We now fix the constant $\delta_1$ for Proposition \ref{prop:main3} associated to the constant $C_2$ of Lemma \ref{lem:essy}. If needed, we retake $\delta_1$ such that $\delta_1< k_0$ for the constant $k_0$ of \eqref{lbis}.  (since Proposition \ref{prop:main3} is valid for any smaller $\delta_1$). From now on, we set (without confusion)
$$
\bar v:=\bar v_{\deltao}, \qquad  \bar{U}:=(\bar v, h).
$$
In what follows, for simplicity we use the notation:
\[
\Omega:=\{x~|~ (p(v)-p(\tiv^{X_1, X_2}))(x) \leq\deltao\},
\]
and omit the dependence of the waves and weights on shifts without confusion, for example, $\tilde U:=\tilde U^{X_1, X_2}, \tilde U_i:=\tilde U_i^{X_i}$ and $(a_i)_x:=(a_i)_x^{X_i}$, etc.\\

In order to present all terms to be controlled on the region $\{|p(v)-p(\tilde{v}^{X_1, X_2})|\geq\delta_1\}$, we split $Y_i$ into four parts $Y^g_i$,  $Y^b_i$, $Y^l_i$ and $Y^s_i$ as follows: for each $i=1, 2$,
\begin{align}
\begin{aligned}\label{defyi}
Y_i &= -\int_\bbr  (a_i)_x \Big(\frac{|h-\tilde h|^2}{2}  + Q(v|\tilde v) \Big) dx +\int_\bbr a \Big(-(\tiv_i)_x p'(\tilde v)(v-\tilde v) +(\tih_i)_x (h-\tilde h) \Big) dx\\
&= Y^g_i +Y^b_i +Y^l_i + Y^s_i ,
\end{aligned}
\end{align}
where 
\begin{align*}
\begin{aligned}
 Y^g_i &:= -\frac{1}{2\sigma_i^2}\int_\Omega (a_i)_x |p(v)-p(\tilde v)|^2 dx -\int_\Omega (a_i)_x Q(v|\tilde v) dx -\int_\Omega a (\tiv_i)_x p'(\tilde v)(v-\tilde v) dx \\
&\quad+\frac{1}{\sigma_i}\int_\Omega a  (\tih_i)_x\big(p(v)-p(\tilde v)\big)dx,\\
Y^b_i &:= -\frac{1}{2}\int_\Omega  (a_i)_x  \Big(h-\tilde h-\frac{p(v)-p(\tilde v)}{\sigma_i}\Big)^2 dx \\
&\quad -\frac{1}{\sigma_i} \int_\Omega  (a_i)_x  \big(p(v)-p(\tilde v)\big)\Big(h-\tilde h-\frac{p(v)-p(\tilde v)}{\sigma_i}\Big) dx,\\
Y^l_i &:=\int_\Omega a (\tih_i)_x \Big(h-\tilde h-\frac{p(v)-p(\tilde v)}{\sigma_i}\Big)dx,
\end{aligned}
\end{align*}
and
\[
Y^s_i :=-\int_{\Omega^c} (a_i)_x Q(v|\tilde v) dx -\int_{\Omega^c} a (\tiv_i)_x p'(\tilde v)(v-\tilde v)dx 
 -\int_{\Omega^c} (a_i)_x \frac{|h-\tilde h|^2}{2} dx+\int_{\Omega^c} a (\tih_i)_x (h-\tilde h) dx.
\]

Notice that $Y^g_i$ consists of the terms related to $v-\tilde v$,  while $Y^b_i$ and $Y^l_i$ consist of  terms related to $h-\tilde h$, where $Y^b_i$ is quadratic, and $Y^l_i$ is linear in $h-\tilde h$. 
In Proposition \ref{prop_out}, we will show that $Y^g_i(U)-Y^g_i(\bar U)$, $Y^b_i(U)$, $Y^l_i(U)$ and $Y^s_i(U)$ are negligible by the good terms. 

For the bad terms ${B}_{\delta_1}$ of \eqref{bad} with $\delta=\deltao$, we will use the following notations :
\beq\label{bad0}
{B}_{\delta_1}=\sum_{i=1}^2\Big({B}_{1i}+{B}_{2i}^- +{B}_{2i}^+ +{B}_{3i} +{B}_{4i} \Big)+{B}_5 +{B}_6,
\eeq
where
\begin{align*}
\begin{aligned}
&{B}_{1i}:= \s_i \int_\bbr a (\tilde v_i)_x p(v|\tilde v) dx,\\
&{B}_{2i}^- :=\int_{\Omega^c} (a_i)_x  \big(p(v)-p(\tilde v )\big) \big(h-\tilde h \big)dx , \qquad {B}_{2i}^+ := \frac{1}{2\s_i} \int_\Omega (a_i)_x |p(v)-p(\tilde v) |^2  dx,\\
&{B}_{3i}:= -\int_\bbr (a_i)_x  v^\beta \big(p(v)-p(\tilde v )\big)\partial_x \big(p(v)-p(\tilde v )\big)  dx,\\
&{B}_{4i}:= -\int_\bbr (a_i)_x  \big(p(v)-p(\tilde v)\big) (v^\beta - \tilde v^\beta) \partial_{x} p(\tilde v)  dx,\\
&{B}_5:= -\int_\bbr a  \partial_x \big(p(v)-p(\tilde v)\big) (v^\beta - \tilde v^\beta ) \partial_{x} p(\tilde v)  dx,
\end{aligned}
\end{align*}
and 
\begin{align*}
{B}_6 := \int_\bbr a \big(p(v)-p(\tilde v) \big) \tilde E_1 dx -  \int_\bbr a (h-\tilde h)  \tilde E_2 dx
\end{align*}
with
\begin{align*}
&\tilde E_1:=  \partial_x\big( \tilde v^\beta \partial_x p(\tilde v )\big) -\partial_x\big( \tilde v_1^\beta \partial_x p(\tilde v_1 )\big) -\partial_x\big( \tilde v_2^\beta \partial_x p(\tilde v_2 )\big) ,\\
&\tilde E_2:=\partial_x p(\tilde v) -\partial_x p(\tilde v_1 ) -\partial_x p(\tilde v_2)  .
\end{align*}

We also recall the functionals $G_{1i}^{-}, G_{1i}^{+}, G_{2i}, D$ of \eqref{ggd} for the good terms.\\
Note that
\begin{align*}
\begin{aligned}
{D}(U)&=\int_\bbr av^{\beta}  |\partial_\xi (p(v)-p(\tilde v))|^2 dx\\
&=\int_\bbr a v^{\beta} |\partial_\xi (p(v)-p(\tilde v))|^2 ( {\mathbf 1}_{\{|p(v)-\pt |\leq\deltao\}} + {\mathbf 1}_{\{p(v)-\pt >\deltao\}}+ {\mathbf 1}_{\{p(v)-\pt <-\deltao\}}  )dx\\
&={D}(\bar U)+\int_\bbr a v^{\beta} |\partial_\xi (p(v)-p(\tilde v))|^2 (  {\mathbf 1}_{\{p(v)-\pt >\deltao\}}+ {\mathbf 1}_{\{p(v)-\pt <-\deltao\}}  ) dx\\
&\ge {D}(\bar U),
\end{aligned}
\end{align*}
and it follow from $Q(v|\tilde{v})\geq  Q(\bar v|\tilde{v})$ that
\beq\label{gi2}
 {G}_{2i}(U)-{G}_{2i}(\bar U)=|\sigma_i|\int_\bbr |(a_i)_x|\left(Q(v|\tilde{v})-Q(\bar v|\tilde{v})\right)\,dx\geq 0,
\eeq

We now state the following proposition.

\begin{proposition}\label{prop_out}
There exist constants $\deo, C, C^*>0$ (in particular, $C$ depends on the constant $\delta_1$) such that for any $\eps_1, \eps_2, \lam>0$ satisfying $\eps_1/\lambda, \eps_2/\lambda< \delta_0$ and $\lambda<\delta_0$, the following statements hold.\\
1. For each $i=1, 2$, for all $U$ satisfying $(-1)^{i-1} Y_i(U)\le \eps_i^2$, the following estimates hold:
\begin{align}
\begin{aligned} \label{n1}
&|{B}_{1i}(U)-{B}_{1i}(\bar U)| \leq C\deo \left( {D}(U) + \left({G}_{2i}(U)-{G}_{2i}(\bar U) \right)+\sum_{j=1}^2 \frac{\eps_j}{\lambda}  {G}_{2j}(U)  \right) ,\\
&|{B}_{2i}^-(U)|  \le\deo \left({D}(U) +\sum_{j=1}^2 \frac{\eps_j}{\lambda} {G}_{2j}(U)  \right) + (\deltao+C\delta_0)  {G}_{1i}^-(U),\\
&|{B}_{2i}^+(U)-{B}_{2i}^+(\bar U)|   \le \sqrt{\deo}{D}(U),
\end{aligned}
\end{align}
and
\beq\label{ib12}
|{B}_{1i}(\bar U)|+|{B}_{2i}^+(\bar U)| \le C\int_\bbr |(a_i)_x| Q(\bar v|\tilde{v})\,dx\le C^*\frac{\eps_i^2}{\lambda}.
\eeq
2. For all $U$ such that $(-1)^{i-1} Y_i(U)\le \eps_i^2$ for all $i=1, 2$, the following estimates hold:
\begin{align}
\begin{aligned} \label{n3}
&\sum_{i=1}^2   \Big(|{B}_{3i}(U)|+|{B}_{4i}(U)| \Big)+|{B}_{5}(U)| \\
&\qquad \le  C\deo D(U)  
+C\deo \sum_{i=1}^2  \Big(\left({G}_{2i}(U)-{G}_{2i}(\bar U)\right)+ \frac{\eps_i}{\lambda} {G}_{2i}(\bar U)\Big),
\end{aligned}
\end{align}

\begin{align}
\begin{aligned}\label{n6}
|{B}_6(U)| &\le C\deo D(U)  
+C\deo \sum_{i=1}^2  \Big( {G}_{1i}^-(U) +{G}_{1i}^+(U)+ \left({G}_{2i}(U)-{G}_{2i}(\bar U)\right)+ \frac{\eps_i}{\lambda} {G}_{2i}(\bar U)\Big)  \\
&\quad + C \exp\Big( -C\min(\eps_1,\eps_2) t \Big) ,\qquad t>0,
\end{aligned}
\end{align}
and 
\beq\label{n2}
|{B}_{\deltao}(U)| \le  C\sqrt{\delta_0} {D}(U) +C .
\eeq
3. For each $i=1, 2$, for all $U$ satisfying $(-1)^{i-1} Y_i(U)\le \eps_i^2$ and ${D}(U)\leq 2\frac{C^*}{\sqrt\deltaz} \frac{\eps_i^2}{\lambda}$ ,
\begin{align}
\begin{aligned}\label{m1}
& |Y^g_i(U)-Y^g_i(\bar U)|^2 +|Y^b_i(U)|^2 +|Y^l_i(U)|^2+|Y^s_i(U)|^2  \\
&\le C\frac{\eps_i^2}{\lambda}\bigg( \sqrt{\deo}{D}(U)+\left({G}_{2i}(U)- {G}_{2i} (\bar U) \right) +\deo\sum_{j=1}^2 \frac{\eps_j}{\lambda}  {G}_{2j}(U) +\left(\frac{\eps_i}{\lambda}\right)^{1/4} {G}_{2i}(\bar U) \\
&\qquad\qquad +{G}_{1i}^-(U) + \left(\frac{\lambda}{\eps_i}\right)^{1/4} {G}_{1i}^+(U) \bigg).
\end{aligned}
\end{align}
\end{proposition}

\begin{proof}
We will apply Proposition \ref{aprop:out} and Proposition \ref{aprop_out} to the two cases where for each $i=1, 2$, $\bw= (a_i)_x^{X_i}, \bv_1 = a^{X_1,X_2} (\tiv_i)_x^{X_i}, \bv_2 = a^{X_1,X_2} (\tih_i)_x^{X_i}$, $\eps=\eps_i$, $\bar\bv=\bar v$, and $\tilde U_0 = \tilde U^{X_1,X_2}$ for the composite wave $\tilde U^{X_1,X_2}$, in addition, ${\bf\Omega}=\Omega$ . \\
First of all, by Lemma \ref{lem:vp}, \eqref{vhre}, \eqref{d-weight} and \eqref{massa}, the assumptions \eqref{aasunif}-\eqref{aasinte} and \eqref{cweight} of the Propositions hold.
In addition, for each $i=1, 2,$ and any $U$ satisfying $(-1)^{i-1} Y_i(U)\le \eps_i^2$, it follows from \eqref{l1} of Lemma \ref{lemmeC2} that the assumption \eqref{agenl1} holds.\\
Then, by \eqref{an1} and \eqref{aib12} together with $1/2 \le a \le 1$ and $C^{-1} \le |\sigma_j| \le C$, we have the desired estimates \eqref{n1}-\eqref{ib12}, where note that since $(\tiv_0)_x = (\tiv)_x = (\tiv_1)_x + (\tiv_2)_x$, 
\beq\label{tg2}
{\bf\tilde G}_2 (U) \le \sum_{j=1}^2 \frac{\eps_j}{\lambda} \int_\bbr |(a_j)_x| Q(v|\tilde{v}) dx \le \sum_{j=1}^2 \frac{\eps_j}{\lambda} {G}_{2j}(U) .
\eeq

To show \eqref{n3}, we first observe that
\begin{align*}
\begin{aligned}
&{B}_{4i}(U)= -\int_\bbr (a_i)_x  \big(p(v)-p(\tilde v)\big) (v^\beta - \tilde v^\beta) p'(\tilde v)  \big( (\tiv_1)_x + (\tiv_2)_x \big) dx,\\
&{B}_5(U)= -\int_\bbr a  \partial_x \big(p(v)-p(\tilde v)\big) (v^\beta - \tilde v^\beta )p'(\tilde v)  \big( (\tiv_1)_x + (\tiv_2)_x \big) dx ,
\end{aligned}
\end{align*}
which together with $|p'(\tilde v)| \le C$ yield
\begin{align*}
\begin{aligned}
&|{B}_{4i}(U)|\le C\sum_{j=1}^2 \int_\bbr |(a_i)_x| |(\tiv_j)_x| \big|p(v)-p(\tilde v)\big| \big|v^\beta - \tilde v^\beta\big|  dx ,\\
&|{B}_5(U)|\le C\sum_{j=1}^2 \int_\bbr  |(\tiv_j)_x| \big|\partial_x \big(p(v)-p(\tilde v)\big) \big| \big|v^\beta - \tilde v^\beta \big| dx  . 
\end{aligned}
\end{align*}
Then, applying \eqref{an3} together with ${\bf B}_3={B}_{3j}$, $|{B}_{4i}(U)|\le C \sum_{j=1}^2 {\bf B}_{4}(U)$ and $|{B}_{5}(U)|\le C \sum_{j=1}^2 {\bf B}_{5}(U)$ for each $i, j$, and then summing them up together with \eqref{tg2}, we obtain the desired estimate \eqref{n3}. 

For \eqref{m1}, since ${D}(U)\leq 2\frac{C^*}{\sqrt\deltaz} \frac{\eps_i^2}{\lambda}$, applying \eqref{am1} together with ${\bf D}(U) \le 2 D(U)$ (by $1/2\le a$), we have \eqref{m1}.

To show  \eqref{n6}, we set
\begin{align*}
{B}_6 = \underbrace{\int_\bbr a \big(p(v)-p(\tilde v) \big) \tilde E_1 dx }_{=: B_{61}} -  \underbrace{\int_\bbr a (h-\tilde h)  \tilde E_2 dx }_{=: B_{62}} ,
\end{align*}
with
\begin{align*}
&\tilde E_1:=  \partial_x\big( \tilde v^\beta \partial_x p(\tilde v )\big) -\partial_x\big( \tilde v_1^\beta \partial_x p(\tilde v_1 )\big) -\partial_x\big( \tilde v_2^\beta \partial_x p(\tilde v_2 )\big) ,\\
&\tilde E_2:=\partial_x p(\tilde v) -\partial_x p(\tilde v_1 ) -\partial_x p(\tilde v_2)  .
\end{align*}

\noindent{\it Estimate of $B_{61}$ }:
First, using $\pa_x\tiv = \pa_x\tiv_1+ \pa_x\tiv_2$, we observe that by setting $f(y):=y^\beta p'(y)$,
\begin{align*}
\tilde E_1
&= \partial_x\Big( \big( f(\tiv) -f(\tiv_1)\big) \pa_x \tiv_1 +  \big( f(\tiv) -f(\tiv_2)\big) \pa_x \tiv_2 \Big)   \\
&=\big( f(\tiv) -f(\tiv_1)\big) (\tiv_1)_{xx} + \big( f(\tiv) -f(\tiv_2)\big) (\tiv_2)_{xx} \\
&\quad+ \big( f'(\tiv) -f'(\tiv_1)\big) |(\tiv_1)_x |^2  + \big( f'(\tiv) -f'(\tiv_2)\big) |(\tiv_2)_x |^2 + 2 f'(\tiv) (\tiv_1)_x (\tiv_2)_x ,
\end{align*}
which together with \eqref{tdv1}, \eqref{tdv2}, \eqref{twodv} and $C^{-1}<\tiv, \tiv_i<C$ implies
\[
|\tilde E_1|\le C \eps_1|(\tiv_1)_x |  |\tiv-\tiv_1| +C \eps_2|(\tiv_2)_x | |\tiv-\tiv_2| + C 
|(\tiv_1)_x| |(\tiv_2)_x| .
\]
Then, using \eqref{d-weight} and $\eps_1,\eps_2<\deo$, we have
\begin{align*}
&\left| B_{61} \right|=\left| \int_\Omega a \big(p(v)-p(\tilde v) \big) \tilde E_1 dx + \int_{\Omega^c} a \big(p(v)-p(\tilde v) \big) \tilde E_1 dx \right|  \\
&\le C \sum_{i=1}^2 \left( \underbrace{ \deo \frac{\eps_i}{\lam} \int_\Omega  |(a_i)_x |  \big|p(v)-p(\tilde v) \big|^2 dx }_{=:K_1} + \underbrace{\int_\Omega |(\tiv_i)_x |  |\tiv-\tiv_i|^2 dx + \int_\Omega|(\tiv_1)_x| |(\tiv_2)_x| dx }_{=:K_2}  \right) \\
&\quad +C \sum_{i=1}^2  \underbrace{\left(\int_{\Omega^c} |(a_i)_x | \big|p(v)-p(\tilde v) \big|^2 dx \right)^{1/2} \left(\int_{\Omega^c} |(\tiv_i)_x |  |\tiv-\tiv_i|^2 dx + \int_\Omega|(\tiv_1)_x| |(\tiv_2)_x| dx  \right)^{1/2} }_{=:K_3}.
\end{align*}
Using \eqref{n1}, \eqref{ib12} and the definition of ${B}_{2i}^+(U)$ with $\s_i (a_i)_x >0$, we have
\begin{align}
\begin{aligned}\label{k1est}
K_1&\le \deo \frac{\eps_i}{\lam} \left(  \int_\Omega |(a_i)_x| \Big | |p(v)-\pt|^2-  |p(\bar v)-\pt|^2\Big|  \,dx +  \int_\Omega  |(a_i)_x| |p(\bar v)-\pt|^2 dx \right) \\
&\le  \deo \frac{\eps_i}{\lam} \left( |{B}_{2i}^+(U)-{B}_{2i}^+(\bar U)| +|{B}_{2i}^+(\bar U)|  \right)\\
&\le \deo \frac{\eps_i}{\lam} \left(  \sqrt{\deo} {D}(U)+  \int_\Omega  |(a_i)_x|  Q(\bar v|\tilde v) \,d\xi \right) \le \deo \frac{\eps_i}{\lam} \left(  \sqrt{\deo} {D}(U)+  {G}_{2i}(\bar U) \right) .
\end{aligned}
\end{align}
By Lemma \ref{lem:sep} with $ |\tiv-\tiv_i|^2\le C  |\tiv-\tiv_i|$, we have
\[
K_2 \le C \eps_1\eps_2 \exp\Big( -C\min(\eps_1,\eps_2) t \Big) ,\qquad t>0.
\]
For $K_3$, we first see that (as in $K_2$)
\[
K_3  \le C\sqrt{\eps_1\eps_2} \exp\Big( -C\min(\eps_1,\eps_2) t \Big) \left(\int_{\Omega^c} |(a_i)_x | \big|p(v)-p(\tilde v) \big|^2 dx \right)^{1/2} ,\qquad t>0. 
\]
Using $\big|p(\bar v)-p(\tiv) \big|\le\deltao$ and \eqref{massa}, we have
\begin{align*}
\int_{\Omega^c} |(a_i)_x | \big|p(v)-p(\tilde v) \big|^2 dx &\le 2\int_{\Omega^c} |(a_i)_x | \big|p(v)-p(\bar v) \big|^2 dx + 2\int_{\Omega^c} |(a_i)_x | \big|p(\bar v)-p(\tiv) \big|^2 dx\\
 &\le 2\int_{\Omega^c} |(a_i)_x | \big|p(v)-p(\bar v) \big|^2 dx + 2\delta_1^2\lam.
\end{align*}
Applying \eqref{newq} with \eqref{tg2}, we have
\[
\int_{\Omega^c} |(a_i)_x | \big|p(v)-p(\bar v) \big|^2 dx \le C\frac{\eps_i^{2-q}}{\lam}\bigg(D(U) + \sum_{j=1}^2 \frac{\eps_j}{\lambda} {G}_{2j}(U) \bigg)^q,
\]
where $q:=\frac{2\gamma}{\gamma+\alpha}$, and note that $1<q<2$ by $0<\alpha<\gamma$.\\
Therefore, using Young's inequality, we have
\begin{align*}
K_3 &\le C\sqrt{\eps_1\eps_2} \exp\Big( -C\min(\eps_1,\eps_2) t \Big) \left[\sqrt{\frac{\eps_i^{2-q}}{\lam}} \left(D(U) + \sum_{j=1}^2 \frac{\eps_j}{\lambda} {G}_{2j}(U)  \right)^{q/2} + C\sqrt\lam \right]\\
&\le C\sqrt{\frac{\eps_1\eps_2}{\lam}} \left(D(U) + \sum_{j=1}^2 \frac{\eps_j}{\lambda} {G}_{2j}(U)  \right)^{q/2}  \sqrt{\eps_1\eps_2} \exp\Big( -C\min(\eps_1,\eps_2) t \Big) \\
&\qquad + C \exp\Big( -C\min(\eps_1,\eps_2) t \Big) \\
&\le \deo \bigg(D(U) + \sum_{j=1}^2 \frac{\eps_j}{\lambda} {G}_{2j}(U) \bigg) + C \exp\Big( -C\min(\eps_1,\eps_2) t \Big) .
\end{align*} 
Hence,
\begin{align*}
\left|B_{61} \right| &\le  C\deo D(U)  
+C\deo \sum_{i=1}^2  \Big(\left({G}_{2i}(U)-{G}_{2i}(\bar U)\right)+ \frac{\eps_i}{\lambda}{G}_{2i}(\bar U) \Big)  + C \exp\Big( -C\min(\eps_1,\eps_2) t \Big) .
\end{align*}

\noindent{\it Estimate of $B_{62}$ }: Likewise, since
\begin{align*}
|\tilde E_2|\le \big| p'(\tilde v) - p'(\tilde v_1)\big| |(\tiv_1)_x|+  \big| p'(\tilde v) - p'(\tilde v_2)\big| |(\tiv_2)_x| \le C |(\tiv_1)_x |  |\tiv-\tiv_1| +C |(\tiv_2)_x | |\tiv-\tiv_2| ,
\end{align*}
we use the Young's inequality and Lemma \ref{lem:sep} to have
\begin{align*}
|B_{62}| &\le  \sum_{i=1}^2 \bigg( \deo\frac{\eps_i}{\lam} \int_\bbr  |(a_i)_x |  |h-\tih |^2 dx +\frac{C}{\deo}\int_\Omega |(\tiv_i)_x |  |\tiv-\tiv_i|^2 dx \bigg) \\
&\le C \sum_{i=1}^2  \deo\frac{\eps_i}{\lam} \int_\bbr  |(a_i)_x |  |h-\tih |^2 dx 
+ C  \exp\Big( -C\min(\eps_1,\eps_2) t \Big).
\end{align*}
Note that using the good terms $G_{1i}^-$ and $G_{1i}^+$,
\begin{align*}
\int_\bbr  |(a_i)_x |  |h-\tih |^2 dx &=\int_{\Omega^c}  |(a_i)_x |  |h-\tih |^2 dx + \int_\Omega  |(a_i)_x |  |h-\tih |^2 dx \\
&\le C G_{1i}^-(U) + C G_{1i}^+(U)  + C\int_\Omega  |(a_i)_x |  |p(v)-p(\tiv) |^2 dx.
\end{align*}
Therefore, using \eqref{k1est}, we have
\[
|B_{62}| \le  C\deo \frac{\eps_i}{\lam} \left( G_{1i}^-(U) +  G_{1i}^+(U) +  \sqrt{\deo} {D}(U)+  {G}_{2i}(\bar U) \right) + C  \exp\Big( -C\min(\eps_1,\eps_2) t \Big).
\]

\noindent{\bf Proof of \eqref{n2}:}
First, it follows from \eqref{ib12} and \eqref{n1} with \eqref{l1} that
\begin{align*}
G_{1i}^+(U) &\le C \int_{\Omega}  |(a_i)_x |  |h-\tih |^2 dx + \int_{\Omega}  |(a_i)_x |  |p(v)-p(\tiv) |^2 dx \\
& \le C \frac{\eps_i^2}{\lambda} +C |{B}_{2i}^+(U)|  \le C \frac{\eps_i^2}{\lambda}  + C\sqrt{\delta_0} {D}(U).
\end{align*}
This together with using \eqref{n1}-\eqref{n6}, \eqref{l1}, \eqref{gi2} and recalling \eqref{bad0}, implies
\[
|{B}_{\deltao}(U)| \le  C\sqrt{\delta_0} {D}(U) +C \sum_{i=1}^2 \frac{\eps_i^2}{\lambda}  + C  \exp\Big( -C\min(\eps_1,\eps_2) t \Big), \quad t>0.
\]
Therefore, we have the rough bound \eqref{n2}.\\

\end{proof}

\subsection{Estimates for separation of waves}

We define a non-negative Lipschitz monotone function $\phi_{1,t}$ on $\bbr$ as follows: for any fixed $t>0$,
\beq\label{1phi}
\phi_{1,t}(x) = \left\{ \begin{array}{ll}
    1\quad &\mbox{if}~ x<\frac{1}{2} (X_1(t) +\s_1 t), \\
   \mbox{linear}\quad &\mbox{if} ~ \frac{1}{2} (X_1(t) +\s_1 t) \le x \le \frac{1}{2} (X_2(t) +\s_2 t), \\
   0  \quad &\mbox{if}  ~  x> \frac{1}{2} (X_2(t) +\s_2 t) .
       \end{array} \right.
\eeq
Likewise, we define a non-negative Lipschitz monotone function $\phi_{2,t}$ on $\bbr$ such that $\phi_{1,t}(x)+\phi_{2,t}(x)=1$ for all $x\in\bbr$ and $t>0$. Thus, $\phi_{2,t}$ satisfies
\beq\label{2phi}
\phi_{2,t}(x) = \left\{ \begin{array}{ll}
    0\quad &\mbox{if}~ x<\frac{1}{2} (X_1(t) +\s_1 t), \\
   \mbox{linear}\quad &\mbox{if} ~ \frac{1}{2} (X_1(t) +\s_1 t) \le x \le \frac{1}{2} (X_2(t) +\s_2 t), \\
   1 \quad &\mbox{if}  ~  x > \frac{1}{2} (X_2(t) +\s_2 t) .
       \end{array} \right.
\eeq

\begin{lemma}\label{lem:phi}
Let $\phi_{1,t}$ and $\phi_{2,t}$ be the non-negative Lipschitz monotone functions such that \eqref{1phi}-\eqref{2phi}.
For given $v_->0$ and $u_-\in\bbr$, there exist positive constants $\eps_0, C$ such that for any $\eps_1, \eps_2\in(0,\eps_0)$ and for all  $t>0$,
\begin{align*}
\begin{aligned}
&\int_\bbr |(\tiv_1)_x^{X_1} | \phi_{2,t}  dx  \le C \eps_1 \exp\big( -C\eps_1 t \big), \\
&\int_\bbr |(\tiv_2)_x^{X_2} | \phi_{1,t}  dx \le C \eps_2 \exp\big( -C\eps_2 t \big) .
\end{aligned}
\end{align*}
\end{lemma}
\begin{proof}
The proof is similar to the one of Lemma \ref{lem:sep}.\\
First, by \eqref{vder1},
\[
\phi_{2,t} |(\tiv_1)_x^{X_1}|^{1/2}   \le C  \phi_{2,t} \eps_1\exp\big( -C\eps_1|x-\s_1t-X_1(t)| \big) .
\]
Note from \eqref{2phi} that $0\le \phi_{2,t}\le 1$ and
\[
\phi_{2,t} = \phi_{2,t} \mathbf{1}_{\{x\ge (X_1(t)+\s_1t)/2 \}} .
\]
Since \eqref{sx12} implies that
\[
x\ge \frac{X_1(t)+\s_1t}{2}\quad\Rightarrow\quad x-(\s_1 t + X_1(t)) \ge - \frac{X_1(t)+\s_1t}{2} \ge -\frac{\s_1}{4}t >0,
\]
we have
\[
\phi_{2,t} |(\tiv_1)_x^{X_1}|^{1/2}  \le C  \eps_1 \exp\big( -C\eps_1 t \big), \quad \forall x\in\bbr, \quad t>0.
\]
Thus, using \eqref{onemass} again, we have the desired estimate.\\
Likewise, using \eqref{vder2} and
\[
\phi_{1,t} = \phi_{1,t} \mathbf{1}_{\{x\le (X_2(t)+\s_2t)/2 \}} ,
\]
and 
\[
x\le \frac{X_2(t)+\s_2t}{2}\quad\Rightarrow\quad x-(\s_2 t + X_2(t)) \le - \frac{X_2(t)+\s_2t}{2} \le -\frac{\s_2}{4}t <0,
\]
we have the desired estimate.
\end{proof}

\begin{lemma}\label{lem:ws}
Let $\phi_{1,t}$ and $\phi_{2,t}$ be  the non-negative Lipschitz monotone functions such that \eqref{1phi}-\eqref{2phi} and $\phi_{1,t}+\phi_{2,t}=1$. For given $v_->0$ and $u_-\in\bbr$, there exist positive constants $C$ and $\deo$ such that for any $\eps_1, \eps_2, \lam>0$ satisfying $\eps_1/\lambda, \eps_2/\lambda< \delta_0$ and $\lambda<\delta_0$, there exists a constant $C_\eps$ depending on $\eps_1, \eps_2$ such that
the following estimates hold.\\
For each $i=1, 2$ and for all $t>0$,
\begin{align}
\begin{aligned}\label{wsby}
& \left| {B}_{1i}(\bar U) -\mathcal{I}_{1i}(\bar v) \right| +\left|{B}_{2i}^+(\bar U)-\mathcal{I}_{2i}(\bar v) \right| +\left|Y^g_i(\bar U) -\mathcal{Y}^g_i(\bar v) \right|   \le C \delta_0 \exp\Big( -C_\eps t \Big) ,\\
&-\left({G}_{2i} (\bar U)- \mathcal{G}_{2i} (\bar v) \right) \le C  \deltaz  \exp\big( -C_\eps t \big),
\end{aligned}
\end{align}
and
\beq\label{fwsd}
-D(\bar U) \le -  \frac{1}{(1+\deltaz)^2} \sum_{i=1}^2   \mathcal{D}_i (\bar v)+ C \deltaz \left(\exp\big( -C_\eps t\big) + \frac{1}{t^4} \right)+  \frac{C}{\deltaz} \frac{1}{t^2} \int_\bbr  \eta(U|\tilde U) dx.
\eeq
\end{lemma}

\begin{proof}
\noindent{\bf Proof of \eqref{wsby}:} 
For simplicity, we here use the following notations: for each $i=1, 2$,
\begin{align*}
\begin{aligned}
&\overline{{B}_{1i}}:= \s_i \int_\bbr a (\tilde v_i)_x \phi_{i,t}^2 p(\bar v|\tilde v) dx,\\
& \overline{{B}_{2i}^+} := \frac{1}{2\s_i} \int_\bbr (a_i)_x\phi_{i,t}^2 \big|p(\bar v)-p(\tilde v)\big|^2  dx,\\
& \overline{{G}_{2i}}:=\sigma_i  \int_\bbr  (a_i)_x \bigg( \frac{1}{2\gamma}  p(\tilde v)^{-\frac{1}{\gamma}-1} \phi_{i,t}^2\big(p(\bar v)-p(\tilde v)\big)^2 - \frac{1+\gamma}{3\gamma^2} p(\tilde v)^{-\frac{1}{\gamma}-2} \phi_{i,t}^3\big(p(\bar v)-p(\tilde v)\big)^3 \bigg) dx, \\
& \overline{{Y}^g_i} :=-\frac{1}{2\sigma_i^2}\int_\bbr (a_i)_x \phi_{i,t}^2 |p(\bar v)-p(\tilde v)|^2 dx -\int_\bbr (a_i)_x \phi_{i,t}^2 Q(\bar v|\tilde v) dx \\
&\qquad\qquad -\int_\bbr a \partial_x p(\tilde v_i) \phi_{i,t} (\bar v-\tilde v)dx+\frac{1}{\sigma_i}\int_\bbr a  (\tilde h_i)_x \phi_{i,t} \big(p(\bar v)-p(\tilde v)\big)dx .
\end{aligned}
\end{align*}
We will only prove the case when $i=1$, since the other case can be shown in the same way.
To estimate $\left|{B}_{11}(\bar U) -\mathcal{I}_{11}(\bar v) \right| $, we first separate it into two parts:
\[
\left|{B}_{11}(\bar U) -\mathcal{I}_{11}(\bar v) \right| \le \left|{B}_{11}(\bar U) - \overline{{B}_{11}} \right| +  \left|\overline{{B}_{11}}-\mathcal{I}_{11}(\bar v) \right| .
\]
Take $\deo$ such that $\deo<\sqrt\eps_0$ for the constant $\eps_0$ of Lemma \ref{lem:phi}.
Since $\phi_{1,t}+\phi_{2,t}=1$ with $|\phi_{i,t}|\le1$ for any $i=1, 2$, and $p(\bar v|\tilde v) \le C$, we use Lemma \ref{lem:phi} to have
\[
\left|{B}_{11}(\bar U) - \overline{{B}_{11}} \right|  \le C \int_\bbr |(\tilde v_1)_x| \phi_{2,t} p(\bar v|\tilde v) dx  \le C \int_\bbr |(\tilde v_1)_x| \phi_{2,t} dx  
\le C \eps_1 \exp\big( -C\eps_1 t \big)  .
\]
To control $\left|\overline{{B}_{11}}-\mathcal{I}_{11}(\bar v) \right|$, we first use Lemma \ref{lem:tri} to find
\[
p(\bar v|\tilde v)-p(\bar v|\tilde v_1) = -p(\tilde v|\tiv_1) + (p'(\tiv)-p'(\tiv_1)) (\tiv - \bar v).
\]
Then, using $\bar v, \tiv, \tiv_i \in (v_-/2 , 2v_-)$ and \eqref{pressure0}, we have
\begin{align*}
 \left|\overline{{B}_{11}}-\mathcal{I}_{11}(\bar v) \right| \le C\int_\bbr |(\tilde v_1)_x| \Big(p(\tilde v|\tiv_1) +|p'(\tiv)-p'(\tiv_1)| |\tiv - \bar v| \Big) dx  \le C \int_\bbr  |(\tilde v_1)_x| |\tiv-\tiv_1| dx ,
\end{align*}
Therefore, using  \eqref{vdid1}, there exists a constant $C_\eps>0$ such that
\begin{align*}
\left|{B}_{11}(\bar U) -\mathcal{I}_{11}(\bar v) \right| &\le C \eps_1 \exp\big( -C\eps_1 t \big)  + C\eps_1\eps_2 \exp\Big( -C\min(\eps_1,\eps_2) t \Big) \\
&\le  C \delta_0 \exp\Big( -C_\eps t \Big).
\end{align*}
Similarly, we have
\begin{align*}
 \left|{B}_{21}^+(\bar U)-\mathcal{I}_{21}(\bar v)  \right| &\le \left|{B}_{21}^+(\bar U) - \overline{{B}_{21}^+} \right| + \left|\overline{{B}_{2i}^+}-\mathcal{I}_{21}(\bar v) \right| \\
 &\le  C\frac{\lambda}{\eps_1}\int_\bbr |(\tiv_1)_x | \bigg(\phi_{2,t} \big| p(\bar v)-p(\tilde v) \big|^2   +  \Big| \big|p(\bar v)-p(\tilde v) \big|^2 -\big|p(\bar v)-p(\tilde v_1)\big|^2 \Big| \bigg)dx \\
 &\le  C\frac{\lambda}{\eps_1}\int_\bbr |(\tiv_1)_x | \bigg(\phi_{2,t} +  |p(\tiv)-p(\tiv_1)| \bigg) dx\\
 &\le  C\frac{\lambda}{\eps_1}\int_\bbr |(\tiv_1)_x | \bigg(\phi_{2,t} +  |\tiv-\tiv_1| \bigg) dx \le  C \delta_0 \exp\Big( -C_\eps t \Big) .
 \end{align*}
Likewise, using \eqref{vhre} and Lemma \ref{lem:tri}, we have
\begin{align*}
\begin{aligned}
& \left| Y^g_1(\bar U) - \mathcal{Y}^g_1(\bar v) \right| \le  \left| Y^g_1(\bar U) - \overline{{Y}^g_1} \right| +\left|\overline{{Y}^g_1}-\mathcal{Y}^g_1(\bar v) \right| \\
 &\le C\frac{\lambda}{\eps_1}\int_\bbr |(\tiv_1)_x | \phi_{2,t} dx + C\frac{\lam}{\eps_1}  \int_\bbr  |(\tilde v_1)_x| \Big(|p(\tiv)-p(\tiv_1)| + Q(\tiv|\tiv_1) + |Q'(\tiv)-Q'(\tiv_1)| \Big) dx\\
&\quad +  C\int_\bbr  |(\tilde v_1)_x| \Big(|\tiv-\tiv_1| + |p(\tiv)-p(\tiv_1)| \Big) dx\\
&\le  C\frac{\lambda}{\eps_1}\int_\bbr |(\tiv_1)_x | \bigg(\phi_{2,t} +  |\tiv-\tiv_1| \bigg) dx \le  C \delta_0 \exp\Big( -C_\eps t \Big) .
\end{aligned}
\end{align*}
Next, to estimate
 \[
 -\left({G}_{21} (\bar U)- \mathcal{G}_{21} (\bar v) \right) =  -\left({G}_{21} (\bar U)- \overline{{G}_{21}} \right) - \left(\overline{{G}_{21}} - \mathcal{G}_{2i} (\bar v) \right),
 \]
we first observe that \eqref{Q-est11} with $\sigma_i (a_i)_x>0$ for any $i$ implies
\begin{align*}
&{G}_{21} (\bar U)- \overline{{G}_{21}} \\
&= \sigma_1  \int_\bbr  (a_1)_x \bigg[ Q(\bar v|\tilde{v}) -\bigg( \frac{1}{2\gamma}  p(\tilde v)^{-\frac{1}{\gamma}-1} \phi_{1,t}^2\big(p(\bar v)-p(\tilde v)\big)^2 \\
&\qquad \qquad - \frac{1+\gamma}{3\gamma^2} p(\tilde v)^{-\frac{1}{\gamma}-2} \phi_{1,t}^3\big(p(\bar v)-p(\tilde v)\big)^3 \bigg) \bigg] dx \\
&\ge \sigma_1  \int_\bbr  (a_1)_x \bigg[ \frac{1}{2\gamma}  p(\tilde v)^{-\frac{1}{\gamma}-1} (1-\phi_{1,t}^2) \big(p(\bar v)-p(\tilde v)\big)^2  \\
&\qquad \qquad - \frac{1+\gamma}{3\gamma^2} p(\tilde v)^{-\frac{1}{\gamma}-2} (1-\phi_{1,t}^3) \big(p(\bar v)-p(\tilde v)\big)^3 \bigg] dx ,
\end{align*}
and thus,
\[
-\left({G}_{21} (\bar U)- \overline{{G}_{21}} \right) \le C\frac{\lambda}{\eps_1}\int_\bbr |(\tiv_1)_x | \phi_{2,t}  dx \le C  \deltaz  \exp\big( -C\eps_1 t \big)  .
\]
Moreover, since
\begin{align*}
\big|\overline{{G}_{21}}-\mathcal{G}_{21} (\bar v) \big| &\le C \frac{\lam}{\eps_1}  \int_\bbr  |(\tilde v_1)_x| \Big(\left|p(\tiv)^{-\frac{1}{\gamma}-1}-p(\tiv_1)^{-\frac{1}{\gamma}-1}\right| + |p(\tiv)-p(\tiv_1)| \\
&\qquad + \left|p(\tiv)^{-\frac{1}{\gamma}-2}-p(\tiv_1)^{-\frac{1}{\gamma}-2}\right| \Big) dx \le C\frac{\lam}{\eps_1}  \int_\bbr  |(\tilde v_1)_x| |\tiv-\tiv_1| dx,
\end{align*}
we have
 \[
 -\left({G}_{21} (\bar U)- \mathcal{G}_{21} (\bar v) \right) \le C \delta_0 \exp\Big( -C_\eps t \Big) .
 \]

\noindent{\bf Proof of \eqref{fwsd}:} 
First, using the fact that $\phi_{1,t}+\phi_{2,t}=1$ and $1\ge \phi_{i,t}\ge \phi_{i,t}^2\ge0$ for any $i$, we separate $D(\bar U)$ into
\[
D(\bar U) = \int_\bbr a\, \bar v^\beta (\phi_{1,t}+\phi_{2,t}) |\partial_x \big(p(\bar v)-p(\tilde v)\big)|^2 dx \ge \sum_{i=1}^2 \int_\bbr a\, \bar v^\beta \phi_{i,t}^2 |\partial_x \big(p(\bar v)-p(\tilde v)\big)|^2 dx .
\]
Since Young's inequality yields
\begin{align*}
\int_\bbr a\, \bar v^\beta |\partial_x \big(\phi_{i,t}(p(\bar v)-p(\tilde v))\big)|^2 dx &\le (1+\deltaz) \int_\bbr a\, \bar v^\beta \phi_{i,t}^2 |\partial_x \big(p(\bar v)-p(\tilde v)\big)|^2 dx \\
&\quad + \frac{C}{\deltaz}  \int_\bbr a\, \bar v^\beta  |\partial_x \phi_{i,t}|^2 |p(\bar v)-p(\tilde v)|^2 dx ,
\end{align*}
we have
\[
-D(\bar U) \le -  \frac{1}{1+\deltaz} \sum_{i=1}^2 \underbrace{\int_\bbr a\, \bar v^\beta |\partial_x \big(\phi_{i,t}(p(\bar v)-p(\tilde v))\big)|^2 dx}_{=:\overline{{D}_i }}
+  \frac{C}{\deltaz} \sum_{i=1}^2 \int_\bbr a\, \bar v^\beta  |\partial_x \phi_{i,t}|^2 |p(\bar v)-p(\tilde v)|^2 dx .
\]
Note that since \eqref{sx12} yields
\[
\frac{1}{2}\Big( (X_2(t)+\s_2 t) - (X_1(t)+\s_1 t) \Big) \ge \frac{\s_2- \s_1}{4} t >0,
\]
it follows from \eqref{1phi}-\eqref{2phi} that for each $i=1, 2$,
\beq\label{phidecay}
|\partial_x \phi_{i,t}(x)| \le \frac{4}{\s_2- \s_1} \frac{1}{t},\quad\forall x\in\bbr, \quad t>0.
\eeq
This together with $C^{-1}\le a \bar v^\beta \le C$ implies
\beq\label{wsd}
-D(\bar U) \le -  \frac{1}{1+\deltaz} \sum_{i=1}^2  \overline{{D}_i } +  \frac{C}{\deltaz} \frac{1}{t^2} \int_\bbr  \eta(U|\tilde U) dx.
\eeq
To estimate $ \overline{{D}_i }$, since Young's inequality yields
\begin{align*}
\int_\bbr a\, \bar v^\beta |\partial_x \big(\phi_{i,t}(p(\bar v)-p(\tilde v_i))\big)|^2 dx &\le (1+\deltaz)\int_\bbr a\, \bar v^\beta |\partial_x \big(\phi_{i,t}(p(\bar v)-p(\tilde v))\big)|^2 dx \\
&\quad + \frac{C}{\deltaz}\int_\bbr a\, \bar v^\beta |\partial_x \big(\phi_{i,t}(p(\tiv)-p(\tiv_i))\big)|^2 dx  ,
\end{align*}
we have
\[
- \overline{{D}_i }  \le -  \frac{1}{1+\deltaz} \int_\bbr a\, \bar v^\beta |\partial_x \big(\phi_{i,t}(p(\bar v)-p(\tilde v_i))\big)|^2 dx
+ \underbrace{ \frac{C}{\deltaz} \int_\bbr a\, \bar v^\beta |\partial_x \big(\phi_{i,t}(p(\tiv)-p(\tiv_i))\big)|^2 dx}_{=:J_i}.
\]
First, when $i=1$, we observe that
\begin{align*}
J_1&\le \frac{C}{\deltaz} \int_\bbr  |\partial_x\phi_{1,t}|^2 |p(\tiv)-p(\tiv_1)|^2 dx + \frac{C}{\deltaz}\int_\bbr  \phi_{1,t}^2 |\partial_x\big(p(\tiv)-p(\tiv_1)\big)|^2 dx \\
&\le \frac{C}{\deltaz}\underbrace{\int_\bbr  |\partial_x\phi_{1,t}|^2 |\tiv-\tiv_1|^2 dx}_{=:J_{11}}+ \frac{C}{\deltaz}\underbrace{ \int_\bbr \Big( |p'(\tiv)-p'(\tiv_1)|^2 |(\tiv_1)_x|^2 +\phi_{1,t}^2 |p'(\tiv)| |(\tiv_2)_x|^2 \Big) dx }_{=:J_{12}} .
\end{align*}
Using \eqref{vdid1} and Lemma \ref{lem:phi}, we have
\begin{align*}
J_{12} &\le C \int_\bbr \Big( |\tiv-\tiv_1|^2 |(\tiv_1)_x|^2 + |\phi_{1,t} (\tiv_2)_x|^2 \Big) dx \\
&\le C\eps_1\eps_2 \exp\Big( -C\min(\eps_1,\eps_2) t \Big)
+ C \eps_2 \exp\big( -C\eps_2 t\big) .
\end{align*}
Since \eqref{1phi}-\eqref{2phi} implies that for each $i=1, 2$,
\[
\partial_x\phi_{i,t}(x) =\partial_x\phi_{i,t}(x) \mathbf{1}_{\{\frac{1}{2}(X_1(t)+\s_1t) \le x\le \frac{1}{2}(X_2(t)+\s_2t)\}} ,
\]
it follows from \eqref{vv1} and \eqref{phidecay} with \eqref{1phi} that
\beq\label{temphiv}
|\partial_x\phi_{1,t}| |\tiv-\tiv_1| \le \frac{C}{t} \eps_2 \exp\big( -C\eps_2|x-\s_2t-X_2(t)| \big)\mathbf{1}_{\{\frac{1}{2}(X_1(t)+\s_1t) \le x\le \frac{1}{2}(X_2(t)+\s_2t)\}} .
\eeq
Then, using
\[
x\le \frac{X_2(t)+\s_2t}{2}\quad\Rightarrow\quad x-(\s_2 t + X_2(t)) \le - \frac{X_2(t)+\s_2t}{2} \le -\frac{\s_2}{4}t <0,
\]
we have
\[
|\partial_x\phi_{1,t}| |\tiv-\tiv_1| \le \frac{C}{t} \eps_2 \exp\big( -C\eps_2 t\big),
\]
which together with \eqref{temphiv} implies
\[
|\partial_x\phi_{1,t}|^2 |\tiv-\tiv_1|^2 \le \frac{C}{t^2} \eps_2^2 \exp\big( -C\eps_2 t\big) \exp\big( -C\eps_2|x-\s_2t-X_2(t)| \big).
\]
Thus, we have
\[
J_{11}\le \frac{C}{t^2} \eps_2\exp\big( -C\eps_2 t\big).
\]
Therefore, 
\[
J_1\le  C \deltaz \left(\exp\big( -C_\eps t\big) + \frac{1}{t^4} \right).
\]
Likewise, using \eqref{vdid1}, Lemma \ref{lem:phi} and \eqref{vv2}, we have
\[
J_{2} \le   C \deltaz \left(\exp\big( -C_\eps t\big) + \frac{1}{t^4} \right).
\]
Hence we have 
\[
- \overline{{D}_i }  \le -  \frac{1}{1+\deltaz} \mathcal{D}_i (\bar v)
+ C \deltaz \left(\exp\big( -C_\eps t\big) + \frac{1}{t^4} \right).
\]
which together with \eqref{wsd} gives the desired result.

\end{proof}

\subsection{Proof of Proposition \ref{prop:main}}
We first apply Proposition \ref{prop:main3} to the functionals of \eqref{note-in} for each waves, and to the weights of \eqref{weight-a}.
Let us take $\deltaz$ small enough as $\deltaz<\delta_1$. 
Then, $U_-, U_m, U_+ \in B_{\delta_0}(U_*)\subset B_{\deltao}(U_*)$. Note that it follows from \eqref{weight-a} that for each $i=1, 2$,
\[
\partial_x a_i = -\lam \frac{\partial_x p(\tilde v_i)}{\eps_i}, \quad\mbox{where } \eps_1=|p(v_-)-p(v_m)|,\quad \eps_2=|p(v_m)-p(v_+)|,
\]
and the weight $a$ satisfies $\|a-1\|_{L^\infty(\bbr)}\le \|a_1-1\|_{L^\infty(\bbr)}+\|a_2-1\|_{L^\infty(\bbr)} \le 2\lam$.\\
Moreover, it follows from Lemma \ref{lem:essy} that for each $i=1, 2$,
\begin{equation}\label{YC2}
|\mathcal{Y}^g_i(\bar v)| \leq C_2 \frac{\eps_i^2}{\lambda}, \quad \forall t\ge t_0. 
\end{equation}
In addition, using \eqref{vv1}, \eqref{vv2} with $\eps_i\le\deltaz\ll\deltao$ (by taking $\deltaz$ small enough as $\deltaz\ll\delta_1$), we have
\[
|p(\bar v)-p(\tiv_i)|\le |p(\bar v)-p(\tiv)| +|p(\tiv)-p(\tiv_i)| \le \deltao +C|\tiv-\tiv_i|\le \deltao + C\deltaz \le 2\deltao.
\]
Therefore, since the two waves $\tilde U_1, \tilde U_2$ and the $\bar v$ satisfy the hypotheses of Proposition \ref{prop:main3} when $t\ge t_0$, we find that for each $i=1, 2$,
\begin{align}
\begin{aligned}\label{newinside}
\mathcal{R}_{\deltao}^i(\bar v)&:=-\frac{1}{\eps_i\deltao}|\mathcal{Y}^g_i(\bar v)|^2 +\mathcal{I}_{1i}(\bar v)+\deltao|\mathcal{I}_{1i}(\bar v)|\\
&\quad\quad+\mathcal{I}_{2i}(\bar v)+\deltao\left(\frac{\eps_i}{\lambda}\right)|\mathcal{I}_{2i}(\bar v)|-\left(1-\deltao\left(\frac{\eps_i}{\lambda}\right)\right)\mathcal{G}_{2i}(\bar v)-(1-\deltao)\mathcal{D}_i(\bar v) \\
&\le 0,\qquad \forall t\ge t_0.
\end{aligned}
\end{align}

For simplicity, we here use the following notations $G_{Y_i}$ to denote the parts of $Y_i$ in $\mathcal{R}$:
\begin{align*}
\begin{aligned}
G_{Y_1}(U)&:= -\frac{1}{\eps_1^4} |Y_1(U)|^2 {\mathbf 1}_{\{0\le Y_1(U) \le \eps_1^2\}} 
+\frac{\s_1}{2\eps_1^2} |Y_1(U)|^2 {\mathbf 1}_{\{ -\eps_1^2 \le Y_1(U) \le 0 \}} 
-\frac{\s_1}{2} Y_1(U) {\mathbf 1}_{\{Y_1(U) \le -\eps_1^2 \}}, \\
G_{Y_2}(U)&:=-\frac{1}{\eps_2^4} |Y_2(U)|^2 {\mathbf 1}_{\{-\eps_2^2\le Y_2(U) \le 0 \}} 
-\frac{\s_2}{2\eps_2^2}  |Y_2(U)|^2 {\mathbf 1}_{\{0\le Y_2(U) \le\eps_2^2 \}}
-\frac{\s_2}{2} Y_2(U) {\mathbf 1}_{\{Y_2(U) \ge \eps_2^2 \}}.
\end{aligned}
\end{align*}
That is,
\begin{align*}
\mathcal{R}(U) &= G_{Y_1}(U)+G_{Y_2}(U) +  {B}_{\deltao}(U)+\delta_0\frac{\min(\eo,\et)}{\lambda} |{B}_{\deltao}(U)|
-{G}_{11}^{-}(U)-{G}_{11}^{+}(U)  \\
&\quad-{G}_{12}^{-}(U)-{G}_{12}^{+}(U)  -\left(1-\delta_0\frac{\eps_1}{\lambda}\right){G}_{21}(U) -\left(1-\delta_0\frac{\eps_2}{\lambda}\right){G}_{22}(U) -(1-\delta_0){D}(U) .
\end{align*}

\noindent{\bf Step 1)} We first get a rough bound for short time $t\le t_0$ as follows. \\
Using $G_{Y_1}(U)\le 0$ and $G_{Y_2}(U)\le 0$ (by $\s_1<0<\s_2$), and $\deltaz<1/2$, we find that for all $t>0$,
\[
\mathcal{R}(U)\le 2 |{B}_{\deltao}(U)| -(1-\delta_0){D}(U).
\]
Thus, using \eqref{n2} of Proposition \ref{prop_out} and taking $\deltaz\ll1$, we find that for all $U$ satisfying  $Y_1(U)\le\eps_1^2$ and $Y_2(U) \ge -\eps_2^2$, 
\[
\mathcal{R}(U) \le C.
\]

\vskip0.2cm
\noindent{\bf Step 2)} This step is for long time $t\ge t_0$.
Without loss of generality, we assume $\eps_1 < \eps_2$.
We split this step into the following three cases, depending on the strength of the dissipation term ${D}(U)$ :
(i) ${D}(U)> \frac{2 C^*}{\sqrt\deltaz} \frac{\eps_2^2}{\lambda}$; (ii) ${D}(U) \le \frac{2 C^*}{\sqrt\deltaz} \frac{\eps_1^2}{\lambda}$; (iii) $\frac{2 C^*}{\sqrt\deltaz} \frac{\eps_2^2}{\lambda}\ge {D}(U)> \frac{2 C^*}{\sqrt\deltaz} \frac{\eps_1^2}{\lambda}$, where $C^*$ is the constant of Proposition \ref{prop_out}.\\

\noindent{\bf Case i) } 
Assume ${D}(U)> \frac{2 C^*}{\sqrt\deltaz} \frac{\eps_2^2}{\lambda}$.
Then, ${D}(U)> \frac{2 C^*}{\sqrt\deltaz} \frac{\eps_i^2}{\lambda}$ for all $i=1, 2$ (by $\eps_1 < \eps_2$).
We first have
\begin{align*}
\begin{aligned}
\mathcal{R}(U)&\le 2 |{B}_{\deltao}(U)|
-\sum_{i=1}^2  \Big( {G}_{1i}^-(U) +{G}_{1i}^+(U)+ \frac{1}{2} \big({G}_{2i}(U)-{G}_{2i}(\bar U)\big) + \frac{1}{2}{G}_{2i}(\bar U) \Big) -\frac{1}{2} {D}(U).
\end{aligned}
\end{align*}
Since it follows from \eqref{n1}, \eqref{ib12}, \eqref{n3} and \eqref{n6}  with $\eps_i/\lam<\deltaz\le\deltao\ll 1$ that  for some constant $C_\eps$ (depending on $\eps_1, \eps_2$),
\begin{align*}
& |{B}_{\deltao}(U)|\le
2 \sum_{i=1}^2 \Big( |{B}_{1i}(U)-{B}_{1i}(\bar U)| + |{B}_{2i}^+(U)-{B}_{2i}^+(\bar U)|+ |{B}_{2i}^-(U)| + |{B}_{3i}(U)|+ |{B}_{4i}(U)| \Big)  \\
&\quad\qquad + 2\sum_{i=1}^2 \big(|{B}_{1i}(\bar U)|+|{B}_{2i}^+(\bar U)| \big) + 2|{B}_{5}(U)|+ 2|{B}_6(U)| \\
&\quad \le C\sqrt\deo D(U)  
+C\deltao \sum_{i=1}^2  \Big( {G}_{1i}^-(U) +{G}_{1i}^+(U)+ \left({G}_{2i}(U)-{G}_{2i}(\bar U)\right)\Big) +C\deltaz \sum_{i=1}^2\frac{\eps_i}{\lambda} {G}_{2i}(\bar U) \\
&\quad\qquad  +2C^* \sum_{i=1}^2 \frac{\eps_i^2}{\lambda}  +C \exp\Big( -C_\eps t \Big),
\end{align*}
 we find that for any $t>0$,
\begin{align*}
\begin{aligned}
\mathcal{R}(U)&\le 2C^* \sum_{i=1}^2 \frac{\eps_i^2}{\lambda}  -\frac{1}{4} {D}(U)+ C \exp\Big( -C_\eps t \Big) \\
& \leq C \exp\Big( -C_\eps t \Big) .
\end{aligned}
\end{align*}

\noindent{\bf Case ii) } 
We first notice that \eqref{yabs} of Lemma \ref{lemmeC2} implies that 
\begin{align}
\begin{aligned}\label{ylinear}
-\frac{\s_1}{2} Y_1(U) {\mathbf 1}_{\{Y_1(U) \le -\eps_1^2 \}} &\le -\frac{\s_1}{2} Y_1(U) {\mathbf 1}_{\{Y_1(U)\in  [-C_0\eps_1^2/\lam,-\eps_1^2] \}} \\
&\le \frac{\s_1\lam}{2C_0\eps_1^2} |Y_1(U)|^2 {\mathbf 1}_{\{Y_1(U)\in  [-C_0\eps_1^2/\lam,-\eps_1^2]\}},\\
-\frac{\s_2}{2} Y_2(U) {\mathbf 1}_{\{Y_2(U) \ge \eps_2^2 \}} &\le -\frac{\s_2}{2} Y_2(U) {\mathbf 1}_{\{Y_2(U)\in  [\eps_2^2, C_0\eps_2^2/\lam] \}} \le -\frac{\s_2\lam}{2C_0\eps_2^2} |Y_2(U)|^2{\mathbf 1}_{\{Y_2(U)\in  [\eps_2^2, C_0\eps_2^2/\lam] \}}.
\end{aligned}
\end{align}
Since it follows from \eqref{defyi}  that for each $i=1, 2$,
$$
Y_i(U)=\mathcal{Y}^g_i(\bar v) + \big(Y^g_i(U)-Y^g_i(\bar U)\big) +\big(Y^g_i(\bar U)-\mathcal{Y}^g_i(\bar v)\big) + Y^b_i(U) +Y^l_i(U)+Y^s_i(U),
$$
we have
\begin{align*}
|\mathcal{Y}^g_i(\bar v)|^2 &\le 4\Big(|Y_i(U)|^2+\big|Y^g_i(U)-Y^g_i(\bar U)\big|^2+ \big|Y^g_i(\bar U)-\mathcal{Y}^g_i(\bar v)\big|^2 \\
&\quad+ | Y^b_i(U)|^2+| Y^l_i(U)|^2+| Y^s_i(U)|^2\Big),
\end{align*}
and so,
\begin{align}
\begin{aligned}\label{simpyi}
-4|Y_i(U)|^2 &\leq -|\mathcal{Y}^g_i(\bar v)|^2+4\Big(\big|Y^g_i(U)-Y^g_i(\bar U)\big|^2+ \big|Y^g_i(\bar U)-\mathcal{Y}^g_i(\bar v)\big|^2 \\
&\quad+ | Y^b_i(U)|^2+| Y^l_i(U)|^2+| Y^s_i(U)|^2\Big).
\end{aligned}
\end{align}
Then, using \eqref{ylinear}, we find that for any $\eps_i/\lam<\deltaz$,
\begin{align*}
&G_{Y_1}(U) \le -\frac{4}{\eps_1\deltao}|Y_1(U)|^2 {\mathbf 1}_{\{Y_1(U)\in [-C_0\eps_1^2/\lam,\eps_1^2]\}} ,\\
&G_{Y_2}(U) \le -\frac{4}{\eps_2\deltao}|Y_2(U)|^2 {\mathbf 1}_{\{Y_2(U)\in  [-\eps_2^2, C_0\eps_2^2/\lam] \}} .
\end{align*}
Thus, it follows from \eqref{simpyi} and \eqref{yabs} of Lemma \ref{lemmeC2} that for any $U$ satisfying $Y_1(U)\le\eps_1^2$ and $Y_2(U)\ge-\eps_2^2$, 
\begin{align}
\begin{aligned}\label{gyif}
G_{Y_i}(U)& \le -\frac{|\mathcal{Y}^g_i(\bar v)|^2}{\eps_i\deltao} 
+ \frac{4}{\eps_i\deltao} \Big(\big|Y^g_i(U)-Y^g_i(\bar U)\big|^2+ \big|Y^g_i(\bar U)-\overline{{Y}^g_i}\big|^2 \\
&\quad + \big|\overline{{Y}^g_i }-\mathcal{Y}^g_i(\bar v)\big|^2 + | Y^b_i(U)|^2+| Y^l_i(U)|^2+| Y^s_i(U)|^2\Big) .
\end{aligned}
\end{align}
Next, for the diffusion term, we use \eqref{fwsd} to have
\begin{align*}
\begin{aligned}
& -(1-\delta_0){D}(U) \le  -((\deltao/2)-\delta_0){D}(U) -(1-\deltao/2){D}(\bar U)\\
&\quad \le -\Big(\frac{\deltao}{2}-\delta_0\Big){D}(U) -  \frac{1-\deltao/2}{(1+\deltaz)^2} \sum_{i=1}^2 \mathcal{D}_i(\bar v) + C \left(\exp\big( -C_\eps t\big) + \frac{1}{t^4} \right)+  \frac{C}{\deltaz} \frac{1}{t^2} \int_\bbr  \eta(U|\tilde U) dx\\
&\quad \le -\Big(\frac{\deltao}{2}-\delta_0\Big){D}(U) -  (1-\deltao) \sum_{i=1}^2 \mathcal{D}_i(\bar v) + C  \left(\exp\big( -C_\eps t\big) + \frac{1}{t^4} \right)+  \frac{C}{\deltaz} \frac{1}{t^2} \int_\bbr  \eta(U|\tilde U) dx .
\end{aligned}
\end{align*}
Therefore, this and \eqref{gyif} imply that for any  $\eps_i/\lambda<\delta_0$,
\begin{align*}
\begin{aligned}
\mathcal{R}(U)\le \sum_{i=1}^2 \big( \mathcal{R}_{\deltao}^i(\bar v)  +\mathbb{Y}_i+2\mathbb{B}_i  + \mathbb{G}_i \big) +\mathbb{D} + C  \left(\exp\big( -C_\eps t\big) + \frac{1}{t^4} \right)+  \frac{C}{\deltaz} \frac{1}{t^2} \int_\bbr  \eta(U|\tilde U) dx  ,
\end{aligned}
\end{align*}
where $\mathcal{R}_{\deltao}^i$ denotes the functional in \eqref{newinside} as 
\begin{align*}
\mathcal{R}_{\deltao}^i(\bar v)& =-\frac{1}{\eps_i\deltao}|\mathcal{Y}^g_i(\bar v)|^2 +\mathcal{I}_{1i}(\bar v)+\deltao|\mathcal{I}_{1i}(\bar v)|\\
&\quad\quad+\mathcal{I}_{2i}(\bar v)+\deltao\left(\frac{\eps_i}{\lambda}\right)|\mathcal{I}_{2i}(\bar v)|-\left(1-\deltao\left(\frac{\eps_i}{\lambda}\right)\right)\mathcal{G}_{2i}(\bar v)-(1-\deltao)\mathcal{D}_i(\bar v),
\end{align*}
and
\begin{align*}
\mathbb{Y}_i&:=  \frac{4}{\eps_i\deltao} \Big(\big|Y^g_i(U)-Y^g_i(\bar U)\big|^2+ \big|Y^g_i(\bar U)-\mathcal{Y}^g_i(\bar v)\big|^2  + | Y^b_i(U)|^2+| Y^l_i(U)|^2+| Y^s_i(U)|^2\Big)
 ,\\
\mathbb{B}_i&:=   |{B}_{1i}(U)-{B}_{1i}(\bar U)| + \left|{B}_{1i}(\bar U) -\mathcal{I}_{1i}(\bar v) \right| + |{B}_{2i}^+(U)-{B}_{2i}^+(\bar U)|\\
&\qquad + \left|{B}_{2i}^+(\bar U) -\mathcal{I}_{2i}(\bar v) \right| + |{B}_{2i}^-(U)| + |{B}_{3i}(U)|+ |{B}_{4i}(U)| + |{B}_{5}(U)|+ |{B}_6(U)| \\
&\qquad -\left(1-\deltao\frac{\eps_i}{\lambda}\right) \Big({G}_{2i} (\bar U)-\mathcal{G}_{2i} (\bar v)  \Big) ,
\end{align*}
and $\mathbb{G}_i$ and $\mathbb{D}$ contain the good terms as
\begin{align*}
\mathbb{G}_i &:= -{G}_{1i}^{-}(U)-{G}_{1i}^{+}(U)-\left(1-\delta_0\frac{\eps_i}{\lambda}\right)\big({G}_{2i}(U)-{G}_{2i}(\bar U)\big) -(\deltao-\deltaz)\frac{\eps_i}{\lambda}{G}_{2i} (\bar U) ,\\
\mathbb{D} &:= -\Big(\frac{\deltao}{2}-\delta_0\Big){D}(U) .
\end{align*}
Since we consider the case when ${D}(U) \le \frac{2 C^*}{\sqrt\deltaz} \frac{\eps_1^2}{\lambda}$, we have ${D}(U) \le \frac{2 C^*}{\sqrt\deltaz} \frac{\eps_i^2}{\lambda}$ for all $i=1, 2$ (by $\eps_1 < \eps_2$). 
Thus, using Proposition \ref{prop_out} and Lemma \ref{lem:ws}, we find that for any $\eps_i/\lambda<\delta_0$ and $\lambda<\delta_0$,
\begin{align*}
\begin{aligned}
&\sum_{i=1}^2\mathbb{Y}_i  \\
&\le \frac{C}{\deltao}\sum_{i=1}^2 \frac{\eps_i}{\lambda}\bigg( \sqrt{\frac{\eps_i}{\lambda}}{D}(U)+\left({G}_{2i}(U)- {G}_{2i} (\bar U) \right) +\left(\frac{\eps_i}{\lambda}\right)^{1/4} {G}_{2i}(\bar U) +{G}_{1i}^-(U) + \left(\frac{\lambda}{\eps_i}\right)^{1/4} {G}_{1i}^+(U)  \bigg) \\
& \quad+ C_{\eps} \exp\Big( -C_\eps t \Big) \\
&\le \frac{C}{\delta_1}\sum_{i=1}^2  \left(\frac{\eps_i}{\lambda}\right)^{1/4} \left( {D}(U)+\left({G}_{2i}(U)- {G}_{2i} (\bar U) \right)+{G}_{1i}^-(U) +  {G}_{1i}^+(U) +\frac{\eps_i}{\lambda} {G}_{2i}(\bar U) \right) + C_{\eps} e^{ -C_\eps t } ,
\end{aligned}
\end{align*}
which together with taking $\deltaz$ small enough as $\deltaz\ll\delta_1^8$ implies
\begin{align*}
\begin{aligned}
\sum_{i=1}^2\mathbb{Y}_i&\le C\frac{\delta_0^{1/4}}{\delta_1} \sum_{i=1}^2 \left( {D}(U)+\left({G}_{2i}(U)- {G}_{2i} (\bar U) \right)+{G}_{1i}^-(U) +  {G}_{1i}^+(U) +\frac{\eps_i}{\lambda} {G}_{2i}(\bar U) \right)+ C_{\eps} e^{ -C_\eps t } \\
&\le -\frac{1}{4}\Big(\mathbb{D} +\sum_{i=1}^2\mathbb{G}_i \Big)+ C_{\eps} e^{ -C_\eps t }   .
\end{aligned}
\end{align*}
Likewise, we have
\begin{align*}
\begin{aligned}
&2\sum_{i=1}^2\mathbb{B}_i  \\
&\le C\sqrt\deo D(U)  
+C\deltao \sum_{i=1}^2  \Big( {G}_{1i}^-(U) +{G}_{1i}^+(U)+ \left({G}_{2i}(U)-{G}_{2i}(\bar U)\right)\Big) +C\deltaz \sum_{i=1}^2\frac{\eps_i}{\lambda} {G}_{2i}(\bar U) \\
&\quad + C e^{ -C_\eps t } \\
&\le -\frac{1}{4}\Big(\mathbb{D} +\sum_{i=1}^2\mathbb{G}_i \Big) + C e^{ -C_\eps t } .
\end{aligned}
\end{align*}
Therefore,
\begin{align*}
\begin{aligned}
\mathcal{R}(U)&\le \sum_{i=1}^2 \mathcal{R}_{\deltao}^i(\bar v)
+   C_\eps \left(\exp\big( -C_\eps t\big) + \frac{1}{t^4} \right)+  \frac{C}{\deltaz} \frac{1}{t^2} \int_\bbr  \eta(U|\tilde U) dx  .
\end{aligned}
\end{align*}
Hence, it follows from \eqref{newinside} that
\[
\mathcal{R}(U) \le  C_\eps \left(\exp\big( -C_\eps t\big) + \frac{1}{t^4} \right)+  \frac{C}{\deltaz} \frac{1}{t^2} \int_\bbr  \eta(U|\tilde U) dx, \quad \forall t\ge t_0.
\]

\noindent{\bf Case iii) } 
Since $\frac{2 C^*}{\sqrt\deltaz} \frac{\eps_2^2}{\lambda}\ge {D}(U)> \frac{2 C^*}{\sqrt\deltaz} \frac{\eps_1^2}{\lambda}$, we may combine the strategies of the previous cases. We first have
\begin{align*}
\begin{aligned}
\mathcal{R}(U) &\le  \big( \mathcal{R}_{\deltao}^2(\bar v)  +\mathbb{Y}_2+2\mathbb{B}_2  +\mathbb{G}_2 \big) +\mathbb{D} + C \left(\exp\big( -C_\eps t\big) + \frac{1}{t^4} \right)+  \frac{C}{\deltaz} \frac{1}{t^2} \int_\bbr  \eta(U|\tilde U) dx \\
&\quad + 2  \Big( |{B}_{11}(U)-{B}_{11}(\bar U)| + |{B}_{21}^+(U)-{B}_{21}^+(\bar U)|+ |{B}_{21}^-(U)| + |{B}_{31}(U)|+ |{B}_{41}(U)| \Big)  \\
&\quad +  2\big(|{B}_{11}(\bar U)|+|{B}_{21}^+(\bar U)| \big) - \Big( {G}_{11}^-(U) +{G}_{11}^+(U)+ \frac{1}{2} \big({G}_{21}(U)-{G}_{21}(\bar U)\big) + \frac{1}{2}{G}_{21}(\bar U) \Big)
\end{aligned}
\end{align*}
Since ${D}(U)\le \frac{2 C^*}{\sqrt\deltaz} \frac{\eps_2^2}{\lambda}$, using the same estimates as in Case i) and ii) together with \eqref{newinside}, we have
\[
\mathbb{Y}_2 \le -\frac{1}{4}\Big(\mathbb{D} + \mathbb{G}_2 \Big)+ \deo G_{21}(U)  + C_{\eps} e^{ -C_\eps t } ,
\]
and
\begin{align*}
\begin{aligned}
&2 \Big(\mathbb{B}_2 + |{B}_{11}(U)-{B}_{11}(\bar U)| + |{B}_{21}^+(U)-{B}_{21}^+(\bar U)|+ |{B}_{21}^-(U)| + |{B}_{31}(U)|+ |{B}_{41}(U)| \Big) \\
& \le -\frac{1}{4}\Big(\mathbb{D} + \mathbb{G}_2 \Big)+ C_{\eps} e^{ -C_\eps t } +\frac{1}{2} \Big( {G}_{11}^-(U) +{G}_{11}^+(U)+ \frac{1}{2} \big({G}_{21}(U)-{G}_{21}(\bar U)\big) + \frac{1}{2}{G}_{21}(\bar U) \Big).
\end{aligned}
\end{align*}
In addition, using \eqref{ib12} and \eqref{newinside}, we have
\begin{align*}
\begin{aligned}
\mathcal{R}(U) &\le \frac{1}{2} \mathbb{D} +   C_\eps \left(\exp\big( -C_\eps t\big) + \frac{1}{t^4} \right)+  \frac{C}{\deltaz} \frac{1}{t^2} \int_\bbr  \eta(U|\tilde U) dx   \\
&\quad + 2C^* \frac{\eps_1^2}{\lambda} +\sqrt{\deo}{D}(U).
\end{aligned}
\end{align*}
Since $\deltaz$ is small enough (as $\deltaz\ll\delta_1$) such that
\[
2\sqrt{\deo}{D}(U) \le -\frac{1}{2} \mathbb{D},
\]
we have
\[
\mathcal{R}(U) \le   C_\eps \left(\exp\big( -C_\eps t\big) + \frac{1}{t^4} \right)+  \frac{C}{\deltaz} \frac{1}{t^2} \int_\bbr  \eta(U|\tilde U) dx  + 2C^* \frac{\eps_1^2}{\lambda} -\sqrt{\deo}{D}(U).
\]
Therefore, by ${D}(U)> \frac{2 C^*}{\sqrt\deltaz} \frac{\eps_1^2}{\lambda}$, we have
\[
\mathcal{R}(U) \le    C_\eps \left(\exp\big( -C_\eps t\big) + \frac{1}{t^4} \right)+  \frac{C}{\deltaz} \frac{1}{t^2} \int_\bbr  \eta(U|\tilde U) dx  .
\]

Hence we complete the proof of Proposition \ref{prop:main}.\\

\section{Proof of Theorem \ref{thm_inviscid}}\label{sec:main}
\setcounter{equation}{0}
This section basically follows the same argument as in \cite[Section 5]{KV-unique19}. Therefore, we omit the details of the proof, but briefly present non-trivial parts of the proof for completeness.  
Contrary to \cite[Theorem 1.1]{KV-unique19}, we need to show the estimates \eqref{limX12} and  \eqref{X-control}  by using the separation property \eqref{amsepX12}. \\
First, we choose  $\{(v^{\nu}_0, u^{\nu}_0)\}_{\nu>0}$ a sequence of smooth functions satisfying \eqref{ini_conv}. 
Indeed, such a sequence is obtained by using the same construction as in  \cite[Section 5.1]{KV-unique19}, since
\begin{align*}
\mbox{as }\nu\to0,\qquad 
&\tiv^\nu(0,x)=\tilde v_1(x/\nu)+\tilde v_2(x/\nu) -v_m \to 
\left\{ \begin{array}{ll}
       v_-,\quad \mbox{if } x<0, \\
       v_+,\quad \mbox{if } x>0,\end{array}\ \right. \\
 &\tilde u^\nu(0,x)=\tilde u_1(x/\nu)+\tilde u_2(x/\nu) -u_m \to 
\left\{ \begin{array}{ll}
       u_-,\quad \mbox{if } x<0, \\
       u_+,\quad \mbox{if } x>0.\end{array}\ \right.
\end{align*}

\subsection{Proof of \eqref{wconv}}
Consider $\{(v^{\nu}, u^{\nu})\}_{\nu>0}$ a sequence of solutions on $(0,T)$ to \eqref{inveq} corresponding to the initial datum $(v^{\nu}_0, u^{\nu}_0)$. \\
Applying \eqref{acont_main} to the following functions: 
\begin{align*}
\begin{aligned}
&v(t,x)=v^{\nu}(\nu t, \nu x),\quad \tilde v(t,x):= \tilde v^{\nu}(\nu t, \nu x),\quad u(t,x)=u^{\nu}(\nu t, \nu x),\quad \tilde u(t, x):= \tilde u^{\nu}(\nu t,\nu x),
\end{aligned}
\end{align*}
and using  \eqref{d-weight}, we have
\begin{align*}
\begin{aligned}
&\int_{-\infty}^\infty E\big((v,u)(t,x)| (\tilde v, \tilde u)^{X_1, X_2}(t,x) \big) dx \\
&\qquad+\int_{0}^{T/\nu}\int_{-\infty}^{\infty} |(\partial_x\tilde v)^{X_1, X_2}(t,x)| \, Q\left(v(t,x)|\tilde v^{X_1, X_2}(t,x)\right) dx dt \\
&\qquad +\nu\int_{0}^{T/\nu}\int_{-\infty}^{\infty} v^{\gamma-\alpha}\big|\partial_x\big(p(v(t,x))-p(\tilde v^{X_1, X_2}(t,x))\big)\big|^2dxdt \\
&\le   C\int_{-\infty}^{\infty} E\big((v_0^\nu,u_0^\nu)(\nu x)| (\tilde v, \tilde u)(0,x)\big) dx.
\end{aligned}
\end{align*}
Then by the change of variables $t\mapsto t/\nu, x\mapsto x/\nu$, we have
\begin{align*}
\begin{aligned}
&\int_{-\infty}^{\infty}  E_\nu \big((v^{\nu},u^{\nu})(t,x)| (\tilde v^{\nu}, \tilde u^{\nu})^{X_1^\nu, X_2^\nu}(t,x)\big) dx \\
&\qquad+\int_{0}^{T}\int_{-\infty}^{\infty} |(\partial_x\tilde v^\nu)^{X_1^\nu, X_2^\nu}(t,x) | Q\left(v^{\nu}(t,x)|(\tilde v^\nu)^{X_1^\nu, X_2^\nu}(t,x)\right) dx dt \\
&\qquad +\nu\int_{0}^{T}\int_{-\infty}^{\infty} (v^{\nu})^{\gamma-\alpha}(t,x)\big|\partial_x\big(p(v^{\nu}(t,x))-p((\tilde v^\nu)^{X_1^\nu, X_2^\nu}(t,x))\big)\big|^2dxdt\\
&\quad \le C\int_{-\infty}^{\infty} E_\nu\big((v^{\nu}_0,u^{\nu}_0)(x)| (\tilde v^{\nu}, \tilde u^{\nu})(0,x) \big) dx,
\end{aligned}
\end{align*}
where $X^{\nu}_i(t):= \nu X_i(t/\nu)$ for each $i=1, 2$, and 
\beq\label{E_nu}
E_\nu((v_1,u_1)|(v_2,u_2)) :=\frac{1}{2}\left(u_1 +\nu \Big(p(v_1)^{\frac{\alpha}{\gamma}}\Big)_x -u_2 -\nu \Big(p(v_2)^{\frac{\alpha}{\gamma}}\Big)_x  \right)^2 +Q(v_1|v_2).
\eeq
By considering the variables: 
\beq\label{effective}
h^{\nu}:=u^{\nu}+\nu \Big(p(v^\nu)^{\frac{\alpha}{\gamma}}\Big)_x,\quad \tilde h^{\nu}:=\tilde u^{\nu}+\nu \Big(p(\tilde v^\nu)^{\frac{\alpha}{\gamma}}\Big)_x, 
\eeq
which satisfy
\[
E_\nu \big((v^{\nu},u^{\nu})| (\tilde v^{\nu}, \tilde u^{\nu})\big) = \eta\big((v^{\nu},h^{\nu})| (\tilde v^{\nu}, \tilde h^{\nu})\big),
\]
and using \eqref{ini_conv}, we find that
\begin{align}
\begin{aligned}\label{ineq-m}
&\mbox{for any $\delta\in(0,1)$, there exists $\nu_*$ such that for all $\nu<\nu_*$}, \\
&\int_{-\infty}^{\infty}  \eta \big((v^{\nu},h^{\nu})(t,x)| (\tilde v^{\nu}, \tilde h^{\nu})^{X_1^\nu, X_2^\nu}(t,x)\big) dx \\
&\qquad+\int_{0}^{T}\int_{-\infty}^{\infty} |(\partial_x\tilde v^\nu)^{X_1^\nu, X_2^\nu}(t,x) | Q\left(v^{\nu}(t,x)|(\tilde v^\nu)^{X_1^\nu, X_2^\nu}(t,x)\right) dx dt \\
&\qquad +\nu\int_{0}^{T}\int_{-\infty}^{\infty} (v^{\nu})^{\gamma-\alpha}(t,x)\big|\partial_x\big(p(v^{\nu}(t,x))-p((\tilde v^\nu)^{X_1^\nu, X_2^\nu}(t,x))\big)\big|^2dxdt\\
&\quad \le C\mathcal{E}_0 +\delta.
\end{aligned}
\end{align}
Therefore, following \cite[Section 5.2.2]{KV-unique19} together with the uniform estimate \eqref{ineq-m}, 
there exist limits $v_{\infty}\in L^\infty(0,T; L^\infty(\bbr)+\mathcal{M}(\bbr))$ and $u_\infty\in  L^\infty(0,T;L^2_{loc}(\bbr))$ such that
\beq\label{vwc}
v^\nu \rightharpoonup v_{\infty}, \quad h^\nu \rightharpoonup u_{\infty} ,\quad \mbox{in} ~\mathcal{M}_{\mathrm{loc}}((0,T)\times\bbr).
\eeq
In addition,
\[
u^\nu \rightharpoonup u_{\infty} \quad \mbox{in} ~\mathcal{M}_{\mathrm{loc}}((0,T)\times\bbr) .
\]
 
\subsection{Proof of \eqref{limX12}}
As in the proof of \cite[Lemma 5.1]{KV-unique19}, we use \eqref{aest-shift} to have 
\[
\Big|\frac{d}{dt}X_i^{\nu}(t)\Big|=|X_i'(t/\nu)| \le C\left( f_\nu (t)+\int_{-\infty}^{\infty}E_1\big((v_0^1,u_0^1)| (\tilde v, \tilde u)\big) dx +1 \right),
\]
where $f_\nu(t):= f\big(\frac{t}{\nu}\big)$. Since \eqref{aest-shift} implies that 
\begin{align*}
\begin{aligned} 
\|f_\nu \|_{L^1(0,T)}&=\nu\|f \|_{L^1(0,T/\nu)} \le C \nu,
\end{aligned}
\end{align*}
$f_\nu$ is uniformly bounded in $L^1(0,T)$. Therefore, $\frac{d}{dt}X_i^{\nu}$ is uniformly bounded in $L^1(0,T)$.\\
Moreover, since $X_i^\nu (0)=0$ and so,
\beq\label{convx}
|X_i^{\nu}(t)|\le Ct+ C\int_0^t f_\nu (s) ds \le C(t+\nu),
\eeq
$X_i^{\nu}$ are also uniformly bounded in $L^1(0,T)$.\\
Therefore, by the compactness of BV, there exist $X_1^\infty,  X_2^\infty\in BV(0,T)$ such that for each $i=1, 2,$
\beq\label{temxc}
X_i^\nu \to X_i^\infty \quad \mbox{in } L^1(0,T),\quad \mbox{up to subsequence as }\nu\to0.
\eeq
Especially, \eqref{temxc} and the uniform (in $\nu$) bound \eqref{convx} imply that
\beq\label{contx}
|X_i^\infty(t)| \le C t\quad\mbox{for a.e. } t\in [0,T].
\eeq
Since it follows from \eqref{amsepX12} that
\[
X_1^\nu(t) \le -\frac{\s_1}{2} t,\quad X_2^\nu (t) \ge -\frac{\s_2}{2} t,\quad\forall t\in (0,T],
\]
we use \eqref{temxc} to have
\[
X_1^\infty(t) \le -\frac{\s_1}{2} t,\quad X_2^\infty (t) \ge -\frac{\s_2}{2} t\quad\mbox{for a.e. } t\in (0,T].
\]
Thus we have
\beq\label{seplimx}
\sigma_1 t +X_1^\infty(t)\le \frac{\s_1}{2} t <0<\frac{\s_2}{2} t\le \sigma_2 t +X_2^\infty(t)\quad\mbox{for a.e. } t\in (0,T],
\eeq
which proves \eqref{limX12}.\\

\subsection{Proof of \eqref{uni-est}} 
Consider a mollifier
\beq\label{def-time}
\phi_\eps(t):=\frac{1}{\eps}\phi \big(\frac{t}{\eps}\big)\quad\mbox{for any }\eps>0,
\eeq
where $\phi:\bbr\to\bbr$ is a nonnegative smooth function such that  $\int_{\bbr}\phi=1$ and $\mbox{supp } \phi = [-1,1]$.\\
Following \cite[Section 5.2.4]{KV-unique19}, the convergence \eqref{wconv} and \eqref{temxc} together with the uniform estimate \eqref{ineq-m} imply that
\begin{align}
\begin{aligned} \label{conv-hv}
& \liminf_{\nu\to 0}\int_{(0,T)\times\bbr} \phi_\eps(s)  \frac{|h^\nu(s,x)-\bar u^{X_1^\infty,X_2^\infty}(s,x)|^2}{2}  ds dx \\
&\quad+ \liminf_{\nu\to 0} \int_{(0,T)\times\bbr} \phi_\eps(s)  Q(v^\nu(s,x) | \bar v^{X_1^\infty,X_2^\infty}(s,x) ) ds dx \le C(\mathcal{E}_0 +\delta).
\end{aligned}
\end{align}
First, using \eqref{vwc} with the weakly lower semi-continuity of the $L^2$-norm (for example see \cite{evans-w}), we have
\begin{align}
\begin{aligned} \label{conv-h1}
&\int_{(0,T)\times\bbr} \phi_\eps(s)  \frac{|u_\infty(s,x)-\bar u^{X_1^\infty,X_2^\infty}(s,x)|^2}{2}  ds dx \\
&\quad \le  \liminf_{\nu\to 0}\int_{(0,T)\times\bbr} \phi_\eps(s)  \frac{|h^\nu(s,x)-\bar u^{X_1^\infty,X_2^\infty}(s,x)|^2}{2}  ds dx .
\end{aligned}
\end{align}
However, to show the weakly lower semi-continuity for the second term on the left-hand side of \eqref{conv-hv}, we use the generalized relative functional \eqref{dQ} to handle the measure $v_\infty$, as mentioned in Section 1.1.\\
More precisely, we set 
\begin{align}
\begin{aligned}\label{def-OM}
&\Omega_M^- :=\{ (t,x)\in(0,T)\times\bbr~|~ x< \s_1 t + X_1^\infty(t) \} ,\\
&\Omega_M^+ :=\{ (t,x)\in(0,T)\times\bbr~|~ x<\s_2 t + X_2^\infty(t) \} .
\end{aligned}
\end{align}
Since $X_1^\infty, X_2^\infty\in BV(0,T)$,
\beq\label{bdry}
\mbox{$\partial\Omega_M^\pm$ (:= the boundary of $\Omega_M^\pm$) has measure zero in } \bbr^2,
\eeq
and the complement of $\overline{\Omega_M^-}\cup\overline{\Omega_M^+}$ on $(0,T)\times\bbr$ is as follows:
\beq\label{com-Om}
\left(\overline{\Omega_M^-}\cup\overline{\Omega_M^+}\right)^c= \{ (t,x)\in(0,T)\times\bbr~|~ \sigma_1 t +X_1^\infty(t)  <x<\sigma_2 t +X_2^\infty(t)\}.
\eeq
Note that 
\beq\label{def-Omv}
 \bar v^{X_1^\infty,X_2^\infty}(t,x) =\left\{ \begin{array}{ll}
        v_- \qquad \mathrm{for  \ \ } (t,x)\in \Omega_M^-,\\
        v_m \qquad \mathrm{for  \ \ } (t,x)\in \left(\overline{\Omega_M^-}\cup\overline{\Omega_M^+}\right)^c,\\
        v_+ \qquad \mathrm{for  \ \ } (t,x)\in \Omega_M^+ .\end{array} \right.
\eeq
The following lemma is on the weakly lower semi-continuity of the functional 
\[
v\mapsto dQ(v |  \bar v^{X_1^\infty,X_2^\infty}).
\] 
\begin{lemma}\label{lem:mlsc}
Assume \eqref{def-OM}-\eqref{def-Omv}.
Let $\Phi:\bbr^+\times\bbr\to\bbr$ be any compactly supported nonnegative function.\\
Let $\{v^k\}_{k=1}^{\infty}$ be a sequence of positive measures in $\mathcal{M}((0,T)\times\bbr)$ such that 
 $v^k \rightharpoonup v_\infty$ in $\mathcal{M}_{\mathrm{loc}}(\bbr^+\times\bbr)$, and there exists a constant $C_0>0$ (independent of $k$) such that
\[
\int_{(0,T)\times\bbr} \Phi(t,x)~d Q\left(v^k |  \bar v^{X_1^\infty,X_2^\infty} \right) (t,x) \le C_0 ,
\]
where for the decomposition $d v^k (t,x) := v^k_a(t,x) dtdx + dv^k_s (t,x)$,
\[
dQ\left(v^k | \bar v^{X_1^\infty,X_2^\infty}\right)(t,x) := Q\left(v^k_a| \bar v^{X_1^\infty,X_2^\infty}\right) dt dx +  |Q'(\overline V(t,x))| dv^k_s (t,x) ,
\]
where 
\beq\label{def-vom}
\overline{V}(t,x):=\left\{ \begin{array}{ll}
        v_- \qquad \mathrm{for  \ \ } (t,x)\in \overline{\Omega_M^-},\\
        v_m \qquad \mathrm{for  \ \ } (t,x)\in \left(\overline{\Omega_M^-}\cup\overline{\Omega_M^+}\right)^c,\\
        v_+ \qquad \mathrm{for  \ \ } (t,x)\in \overline{\Omega_M^+} .\end{array} \right.
\eeq
Then,
\[
\int_{(0,T)\times\bbr}  \Phi(t,x)~d Q\left(v_\infty |  \bar v^{X_1^\infty,X_2^\infty}\right) (t,x)  \le  C_0.
\]
where 
\beq\label{RNv}
d v_\infty (t,x) = v_a(t,x) dtdx + dv_s (t,x),
\eeq
for some $v_a \in L^1(\bbr)$ and $dv_s$ (singular part of $v_\infty$)
\end{lemma}
The proof of Lemma \ref{lem:mlsc} follows the same argument as that of \cite[Lemma 5.3]{KV-unique19}, since $|Q'(v_\pm)|\le |Q'(v_m)|$. For completeness, we present the proof of Lemma \ref{lem:mlsc} in Appendix \ref{app:mlsc}.\\
To apply Lemma \ref{lem:mlsc}, we consider a smooth function $\psi_0^R$ such that for any $R>0$,
\[
\mathbf{1}_{|x|\le R} \le \psi_0^R (x) \le \mathbf{1}_{|x|\le 2R} .
\]
Then, 
\begin{align*}
\begin{aligned} 
& \int_{(0,T)\times\bbr} \phi_\eps(s)  \psi_0^R (x)   Q(v^\nu(s,x) | ^{X_1^\infty,X_2^\infty}(s,x)) ds dx \\
&\quad \le \int_{(0,T)\times\bbr} \phi_\eps(s)  Q(v^\nu(s,x) | \bar v^{X_1^\infty,X_2^\infty}(s,x) ) ds dx.
\end{aligned}
\end{align*}
Thus, using Lemma \ref{lem:mlsc} together with the weak convergence \eqref{vwc}, we have
\begin{align}
\begin{aligned}\label{lastqn} 
& \int_{(0,T)\times\bbr} \phi_\eps(s)  \psi_0^R (x)  ~ d Q(v_\infty | \bar v^{X_1^\infty,X_2^\infty}) (s,x) \\
&\quad \le \liminf_{\nu\to \infty}\int_{(0,T)\times\bbr} \phi_\eps(s)  Q(v^\nu(s,x) | \bar v^{X_1^\infty,X_2^\infty}(s,x) ) ds dx ,
\end{aligned}
\end{align}
where for the decomposition \eqref{RNv},
\[
dQ\left(v_\infty |\bar v^{X_1^\infty,X_2^\infty} \right)(t,x) = Q\left(v_a|\bar v^{X_1^\infty,X_2^\infty}\right) dt dx +  |Q'(\overline V(t,x))| dv_s (t,x) ,
\]
where $\overline{V}(t,x)$ is defined by \eqref{def-vom} with \eqref{def-OM}.\\
Then, using $\bbr^+\times (-R,R)\nearrow \bbr^+\times\bbr$ as $R\to\infty$, we have
\begin{align*}
\begin{aligned} 
& \int_{(0,T)\times\bbr} \phi_\eps(s)   ~ d Q(v_\infty | \bar v^{X_1^\infty,X_2^\infty}) (s,x)  \\
&\quad \le \liminf_{R\to 0} \int_{(0,T)\times\bbr} \phi_\eps(s)  \psi_0^R (x)  ~ d Q(v_\infty | \bar v^{X_1^\infty,X_2^\infty}) (s,x) .
\end{aligned}
\end{align*}
Therefore, this together with \eqref{lastqn}, \eqref{conv-hv} and \eqref{conv-h1} yields
\begin{align*}
\begin{aligned} 
&\int_{(0,T)\times\bbr} \phi_\eps(s)  \frac{|u_\infty(s,x)-\bar u^{X_1^\infty,X_2^\infty}(s,x)|^2}{2}  ds dx \\
& \quad + \int_{(0,T)\times\bbr} \phi_\eps(s)   ~ d Q(v_\infty | \bar v^{X_1^\infty,X_2^\infty}) (s,x)   \le C(\mathcal{E}_0 +\delta) .
\end{aligned}
\end{align*}
Taking $\eps\to0$ (by \eqref{def-time}), we obtain that
\[
d Q(v_\infty | \bar v^{X_1^\infty,X_2^\infty}) \in L^\infty(0,T;\mathcal{M}(\bbr)),
\]
and, for a.e. $s\in(0,T)$,
\[
\int_{\bbr }  \frac{|u_\infty(s,x)-\bar u^{X_1^\infty,X_2^\infty}(s,x)|^2}{2}   dx +\int_{x\in \bbr} d Q(v_\infty | \bar v^{X_1^\infty,X_2^\infty}) (s,x)  \le C(\mathcal{E}_0 +\delta).
\]
Since $\delta>0$ is arbitrary, we have the desired estimate \eqref{uni-est}.

\subsection{Proof of \eqref{X-control}}
This will be shown by using the fact that it follows from \eqref{wconv} and the first (linear) equation of $\eqref{inveq}_1$ that
the limits $v_\infty, u_\infty$ satisfy $\partial_t v_\infty -\partial_x u_\infty=0$ in the sense of distributions, and by using the stability estimate \eqref{uni-est} together with \eqref{seplimx} on the separation estimate  and \eqref{contx} on the continuity at $t=0$ of the shifts.\\
First, thanks to the fact that $X_1^\infty, X_2^\infty \in BV ((0,T))$, we choose positive constants $r_1=r_1(T)$ and $r_2=r_2(T)$ such that $\|X_i^\infty\|_{L^{\infty}((0,T))}\le r_i$ for each $i=1,2$. \\
Let $t_0\in (0,T)$ be any constant at which \eqref{contx}-\eqref{seplimx} and \eqref{uni-est} hold (Note that \eqref{contx}-\eqref{seplimx} and \eqref{uni-est} hold for almost everywhere on $(0,T)$).
We consider $\psi_1, \psi_2: \bbr \to\bbr$ nonnegative smooth functions defined by
\begin{align*}
\begin{aligned} 
&\psi_1(x)=\left\{ \begin{array}{ll}
       1\quad\mbox{if }  x\in [-r_1+\s_1T, \frac{\s_1 t_0}{2}],\\
       0\quad\mbox{if }  x\le -2r_1+\s_1T\mbox{ or } x\ge 0,\end{array} \right.
\quad \mbox{and}\quad \|\psi_1'\|_{L^\infty(\bbr)}\le \max\left(\frac{2}{|r_1|}, \frac{4}{|\s_1|t_0}\right), \\
& \psi_2(x)=\left\{ \begin{array}{ll}
       1\quad\mbox{if }  x\in [\frac{\s_2 t_0}{2}, r_2 + \s_2 T],\\
       0\quad\mbox{if }  x\le 0\mbox{ or } x\ge 2r_2+\s_2T,\end{array} \right.
\quad \mbox{and}\quad \|\psi_2'\|_{L^\infty(\bbr)}\le \max\left(\frac{2}{r_2}, \frac{4}{\s_2t_0}\right)
\end{aligned}
\end{align*}
Let $\theta:\bbr\to\bbr$ be a nonnegative smooth function such that $\theta (s)=\theta(-s)$, $\int_{\bbr}\theta=1$ and $\mbox{supp } \theta = [-1,1]$, and let 
\[
\theta_\delta(s):=\frac{1}{\delta}\theta \big(\frac{s-\delta}{\delta}\big)\quad\mbox{for any }\delta>0.
\]
For any $t\in(t_0,T)$ at which \eqref{seplimx} and \eqref{uni-est} hold, and 
for any $\delta<\min(t-t_0,T-t)$, we define a nonnegative smooth function
\[
\varphi_{t, \delta} (s) := \int_0^s \Big(\theta_\delta (\tau -t_0) -\theta_\delta (\tau- t)  \Big) d\tau.
\]
Since  $\partial_t v_\infty -\partial_x u_\infty=0$ in the sense of distributions, we have
\beq\label{limit-con}
\int_{(0,T)\times\bbr} \big(  \varphi_{t, \delta}'(s) \psi_i(x) d v_{\infty}(s,x) - \varphi_{t, \delta} (s) \psi_i'(x) u_{\infty}(s,x) dsdx \big)  =0,\quad\mbox{for each } i=1,2.
\eeq
We rewrite the left-hand side above as
\[
I_1^\delta+I_2^\delta+I_3^\delta=0,
\]
where
\begin{align*}
\begin{aligned} 
&I_1^\delta:=\int_{(0,T)\times\bbr}  \theta_\delta (s-t_0) \psi_i(x) d v_{\infty}(s,x)  ,\\
&I_2^\delta:= -\int_{(0,T)\times\bbr}  \theta_\delta (s-t)   \psi_i(x) d v_{\infty}(s,x) ,\\
&I_3^\delta:= -\int_{(0,T)\times\bbr} \varphi_{t, \delta} (s) \psi_i'(x) u_{\infty}(s,x)dxds.
\end{aligned}
\end{align*}
Since $v_\infty$ is weakly continuous in time (as in  \cite[Section 5]{KV-unique19}), we find that as $\delta\to0$ :
\[
I_1^\delta\to \int_\bbr  \psi_i(x) v_{\infty}(t_0,dx),\quad I_2^\delta \to -\int_\bbr  \psi_i(x) v_{\infty}(t,dx) ,
\]
and 
\[
I_3^\delta \to -\int_{t_0}^t \int_\bbr \psi_i'(x) u_{\infty}(s,x) dxds.
\]
Therefore, it follows from \eqref{limit-con} that
\[
\underbrace{\int_\bbr  \psi_i(x) \big(v_{\infty}(t,dx) -v_{\infty}(t_0,dx) \big) }_{=:J_1} + \underbrace{\int_{t_0}^t \int_\bbr \psi_i'(x) u_{\infty}(s,x) dxds}_{=:J_2} =0.
\]
Our strategy is to use the stability estimate \eqref{uni-est} and the Rankine-Hugoniot condition. 
For that, we use the shifted Riemann solution to decompose $J_1$ into three parts:
\[
J_1=J_{11}^i+J_{12}^i+J_{13}^i,
\]
where 
\begin{align*}
\begin{aligned} 
J_{11}^i &=\int_\bbr  \psi_i(x) \big(v_{\infty}(t,dx) -\bar v^{X_1^\infty,X_2^\infty}(t,x)dx\big) , \\
J_{12}^i &= \int_\bbr  \psi_i(x) \big(\bar v^{X_1^\infty,X_2^\infty}(t,x)-\bar v^{X_1^\infty,X_2^\infty}(t_0,x) \big)  dx,\\
J_{13}^i &=\int_\bbr  \psi_i(x) \big(\bar v^{X_1^\infty,X_2^\infty}(t_0,x) dx - v_{\infty}(t_0,dx) \big).
\end{aligned}
\end{align*}
Likewise, we decompose $J_2$ into two parts:
\[
J_2=J_{21}^i+J_{22}^i,
\]
where 
\begin{align*}
\begin{aligned} 
J_{21}^i &=\int_{t_0}^t \int_\bbr  \psi_i'(x) \big(u_{\infty}(s,x) -\bar u^{X_1^\infty,X_2^\infty}(s,x) \big)  dxds, \\
J_{22}^i &= \int_{t_0}^t \int_\bbr  \psi_i'(x) \bar u^{X_1^\infty,X_2^\infty}(s,x) dx ds.
\end{aligned}
\end{align*}
Since it follows from \eqref{seplimx} that
\begin{align*}
&-r_1 + \s_1T \le \sigma_1 t +X_1^\infty(t)\le \frac{\s_1}{2} t_0 <0\\
&\qquad<\frac{\s_2}{2} t_0\le \sigma_2 t +X_2^\infty(t)\le r_2+\s_2T \quad\mbox{for a.e. } t \in [t_0,T),
\end{align*}
we have
\begin{align*}
&J_{12}^i = \left\{ \begin{array}{ll}
       (v_- -v_m)(X_1^\infty(t)-X_1^\infty(t_0)+\s_1 (t-t_0))\quad\mbox{when } i=1,\\
        (v_m -v_+)(X_2^\infty(t)-X_2^\infty(t_0)+\s_2 (t-t_0))\quad\mbox{when } i=2,\end{array} \right. \\
&J_{22}^i = \left\{ \begin{array}{ll}
       (t-t_0)(u_--u_m)\quad\mbox{when } i=1,\\
       (t-t_0)(u_m-u_+)\quad\mbox{when } i=2,\end{array} \right.
\end{align*}
Then using $\s_1= -\frac{u_m-u_-}{v_m-v_-}, \s_2= -\frac{u_+-u_m}{v_+-v_m}$ by the condition \eqref{end-con}, we have
\[
J_{12}^i +J_{22}^i = \left\{ \begin{array}{ll}
       (v_- -v_m)(X_1^\infty(t)-X_1^\infty(t_0))\quad\mbox{when } i=1,\\
        (v_m -v_+)(X_2^\infty(t)-X_2^\infty(t_0))\quad\mbox{when } i=2,\end{array} \right.
\]
Now, it remains to control the remaining terms by the initial perturbation $\mathcal{E}_0$ as follows.
First, recall the (unique) decomposition of the measure $v_\infty$ by
\[
d v_\infty (t,dx) = v_a(t,x) dx + v_s (t,dx) .
\]
Using \eqref{rel_Q}, we have
\begin{align*}
\begin{aligned} 
|J_{11}^1| &\le  \int_{-2r_1+\s_1T}^{0} \big|v_a (t,x) -\bar v^{X_1^\infty,X_2^\infty}(t,x)\big| {\mathbf 1}_{\{v\le 3v_-\}} dx  \\
&\quad + \int_{-2r_1+\s_1T}^{0} \big|v_a (t,x) -\bar v(x-X_\infty(t))\big| {\mathbf 1}_{\{v\ge 3v_-\}} dx +\int_\bbr \psi(x) v_s (t,dx) \\
&\le  \frac{1}{\sqrt{c_1}}  \int_{-2r_1+\s_1T}^{0} \sqrt{Q\big(v_a (t,x)|\bar v(x-X_\infty(t))\big)} dx \\
&\quad + \frac{1}{c_2} \int_{-2r_1+\s_1T}^{0} Q\big(v_a(t,x)|\bar v(x-X_\infty(t))\big) dx\\
&\quad +\frac{1}{|Q'(\max(v_-,v_+))|}\int_\bbr \psi(x)|Q'(\overline V)| v_s (t,dx) ,
\end{aligned}
\end{align*}
where note that $|Q'(\overline V)|\ge |Q'(\max(v_-,v_+))|>0$ by  \eqref{barvd}.\\
Thus, we use the stability estimate \eqref{uni-est} to have
\[
|J_{11}^1|  \le C \sqrt{\mathcal{E}_0} + C\mathcal{E}_0.
\]
Likewise, we have
\[
|J_{11}^2|  \le C \sqrt{\mathcal{E}_0} + C\mathcal{E}_0,
\]
and 
\[
|J_{13}^i|  \le C \sqrt{\mathcal{E}_0} + C\mathcal{E}_0\quad\mbox{for } i=1,2.
\]
Similarly, there exists a constant $C(t_0)$ depending on $t_0$ such that
\begin{align*}
|J_{21}^1| & \le \|\psi_1'\|_{L^\infty(\bbr)} \int_{t_0}^t \int_{[-2r_1+\s_1T,-r_1+\s_1T]\cup[\s_1t_0/2,0]} \big|u_{\infty}(s,x) -\bar u^{X_1^\infty,X_2^\infty}(s,x) \big| dxds \\
&\le  C(t_0)t\sqrt{\mathcal{E}_0},
\end{align*}
and 
\[
|J_{21}^2| \le C(t_0)t\sqrt{\mathcal{E}_0}.
\]
Therefore, we have shown that for each $i=1, 2$,
\[
 | X_i^\infty(t)-X_i^\infty(t_0) | \le C(t_0)\Big( \mathcal{E}_0 + (1+t)\sqrt{\mathcal{E}_0} \Big),  \quad\mbox{for a.e. } t \in (t_0,T).
\]
Since it follows from \eqref{contx} that
\[
|X_i^\infty(t_0)|\le C t_0,
\]
we have
\[
|X_i^\infty(t)| \le C(t_0)\Big( \mathcal{E}_0 + (1+t)\sqrt{\mathcal{E}_0} \Big),  \quad\mbox{for a.e. } t \in (t_0,T).
\]
Hence, this and \eqref{contx} imply the desired estimate \eqref{X-control}.\\

\begin{appendix}
\setcounter{equation}{0}
\section{Useful inequalities}  \label{app-use}
\setcounter{equation}{0}

We here present the useful inequalities developed in \cite[Section 2.4 and Lemma 2.2]{Kang-V-NS17}.\\
First, the following lemma provides some global inequalities on the relative function $Q(\cdot|\cdot)$ corresponding to the convex function $Q(v)=\frac{v^{-\gamma+1}}{\gamma-1}$, $v>0$, $\gamma>1$. 

\begin{lemma}\label{lem-pro}  \cite[Lemma 2.4]{Kang-V-NS17}
For given constants $\gamma>1$, and $v_*>0$, there exists constants $c_1, c_2>0$ such that  the following inequalities hold.\\
1)  For any $w\in (0,2v_*)$,
\begin{align}
\begin{aligned}\label{rel_Q}
& Q(v|w)\ge c_1 |v-w|^2,\quad \mbox{for all } 0<v\le 3v_*,\\
 & Q(v|w)\ge  c_2 |v-w|,\quad  \mbox{for all } v\ge 3v_*.
\end{aligned}
\end{align}

2) Moreover if $0<w\leq u\leq v$ or $0<v\leq u\leq w$ then 
\beq\label{Q-sim}
Q(v|w)\geq Q(u|w),
\eeq
and for any $\delta_*>0$ there exists a constant $C>0$ such that if, in addition, 
$|w-v_*|\le \delta_*/2$ and $|w-u|>\delta_*$, we have
\beq\label{rel_Q1}
Q(v|w)-Q(u|w)\geq C|u-v|. 
\eeq
\end{lemma}

The following lemma provides some global inequalities on the pressure $p(v)=v^{-\gamma}$, $v>0$, $\gamma>1$,  and on the associated relative function $p(\cdot|\cdot)$.

\begin{lemma}\label{lem-pro2}  \cite[Lemma 2.5]{Kang-V-NS17}
For given constants $\gamma>1$, and $v_*>0$, there exist constants $c_3, C>0$ such that  the following inequalities hold.\\
For any $w>v_*/2$,
\beq\label{pressure2}
|p(v)-p(w)| \le c_3 |v-w|,\qquad \forall v\ge v_*/2,
\eeq
\beq\label{pressure0}
p(v|w) \le C |v-w|^2,\qquad \forall v\ge v_*/2 ,
\eeq
\beq\label{pressure4}
p(v|w)\leq C(|v-w|+|p(v)-p(w)|),\qquad \forall v>0.
\eeq
\end{lemma}

The following lemma presents some local estimates on $p(v|w)$ and $Q(v|w)$ for $|v- w|\ll 1$, based on  Taylor expansions.

\begin{lemma}\label{lem:local}
For given constants $\gamma>1$ and $v_*>0$ 
there exist positive constants $C$ and $\delta_*$ such that for any $0<\delta<\delta_*$, the following is true.\\
1) For any $(v, w)\in \bbr_+^2$  
satisfying $|p(v)-p(w)|<\delta$, and  $|p(w)-p(v_*)|<\delta$ the following estimates \eqref{p-est1}-\eqref{Q-est1} hold:
\begin{align}
\begin{aligned}\label{p-est1}
p(v|w)&\le \bigg(\frac{\gamma+1}{2\gamma} \frac{1}{p(w)} + C\delta \bigg) |p(v)-p(w)|^2,
\end{aligned}
\end{align}
\beq\label{Q-est11}
Q(v|w)\ge \frac{p(w)^{-\frac{1}{\gamma}-1}}{2\gamma}|p(v)-p(w)|^2 -\frac{1+\gamma}{3\gamma^2} p(w)^{-\frac{1}{\gamma}-2}(p(v)-
p(w))^3,
\eeq
\beq\label{Q-est1}
Q(v|w)\le \bigg( \frac{p(w)^{-\frac{1}{\gamma}-1}}{2\gamma} +C\delta  \bigg)|p(v)-p(w)|^2.
\eeq
2) For any $(v, w)\in \bbr_+^2$ such that  $|p(w)-p(v_*)|\leq \delta$,  and satisfying either $Q(v|w)<\delta$ or $|p(v)-p(w)|<\delta$,
\beq\label{pQ-equi0}
|p(v)-p(w)|^2 \le C Q(v|w).
\eeq
\end{lemma}

The following lemma presents an estimate based on the inverse of the pressure function, which will be used in Appendix \ref{app-exp}. 
\begin{lemma}\label{lem:press}
For any $r>0$, there exist $\eps_0>0$ and $C>0$ such that the following holds.
For any $p_-, p_+, p>0$ such that $p_-\in (r/2, 2r)$,
$p_+-p_-=:\eps\in (0, \eps_0)$, $p_-\leq p\leq p_+$, and $v, v_-,v_+$ such that 
$p(v)=p, p(v_\pm)=p_\pm$, we have
$$
\left|\frac{v-v_-}{p-p_-}+\frac{v-v_+}{p_+-p}+\frac{1}{2}\frac{p''(v_-)}{p'(v_-)^2}(v_--v_+)\right|\leq C\eps^2.
$$
\end{lemma}

The following identity of \cite[Lemma 4.1]{KVARMA} will be used when splitting the composite wave. 
\begin{lemma}\cite[Lemma 4.1]{KVARMA} \label{lem:tri}
For any function $F$, its relative function $F(\cdot|\cdot)$ satisfies
\[
F(u|w)+F(w|v) = F(u|v) + (F'(w)-F'(v)) (w-u),\quad\forall u, v, w.
\]
\end{lemma}

\section{Proof of Proposition \ref{prop:main3}} \label{app-exp}
\setcounter{equation}{0}

First of all, given $(v_*,u_*)$, fix the value of $\eps_0$ corresponding to the constant $r:=p(v_*)$ in Lemma \ref{lem:press}.
Then we consider any constant $\deltao\in(0,1/2)$ such that
$$
\deltao < \min\left(\frac{\delta_*}{ |p'(v_*/2)|}, \frac{\delta_*}{2}, \frac{p(v_*)}{2}, \sqrt\eps_0\right),
$$
where $\delta_*$ is the constant as in Lemma \ref{lem:local}.  \\
Note that using $|\tilde v_0- v_*|\le \delta_1$, we have
\[
|p(\tilde v_0)-p(v_*)|\le |p'(v_*/2)| |\tilde v_0- v_*| \le   |p'(v_*/2)| \deltao <\delta_*. 
\]
Moreover, since $|p(v)-p(\tilde v_0)|\le 2\deltao<\delta_*$, we can apply the results of Lemma \ref{lem:local} to the case of  $w=\tilde v_0$. 
Since $\eps \le \lam\deltao<\delta_1^2< \eps_0$, 
 we can also apply Lemma \ref{lem:press} to the case of $\{p_-,p_+\}=\{p(v_l),p(v_r)\}$.

Without loss of generality, we assume $v_l<v_r$, and so, $\s_0=\sqrt{-\frac{p(v_r)-p(v_l)}{v_r-v_l}}>0$ (that is, assume the case of $2$-shock).\\
We will rewrite the functionals $\mathcal{Y}^g, \mathcal{I}_{1}, \mathcal{I}_{2}, \mathcal{G}_{2}, \mathcal{D}$ with respect to the following variables:
\beq\label{notw}
w:= \big(p(v)-p(\tilde v_0)\big)\phi,\qquad W:=\frac{\lambda}{\eps} w,\qquad  y:=\frac{p(\tilde v_0)-p(v_r)}{\eps} ,
\eeq
where note that $\eps=p(v_l)-p(v_r)$.\\
Since  $p(\tilde v_0)$ is decreasing in $x$,
we can use the change of variables $x\mapsto y\in[0,1]$. \\
Notice from the assumptions that 
\beq\label{ach}
\partial_x a_0 = -\lam \frac{\partial_x p(\tilde v_0)}{\eps},\qquad   \partial_x y= \frac{\partial_x p(\tilde v_0)}{\eps} ,
\eeq
and
\beq\label{approx1}
\|a-1\|_{L^\infty(\bbr)}\le 2\lam \leq 2\deltao .
\eeq

First of all, note that the functionals $\mathcal{Y}^g, \mathcal{I}_{1}, \mathcal{I}_{2}$  with  $\phi\equiv 1$
are  the same as $Y_g, \mathcal{I}_{1}, \mathcal{I}_{2}$ in \cite[Proposition 4.2]{KV-unique19} (and as $Y_g, \mathcal{B}_{1}, \mathcal{B}_{21}$ in \cite[Proposition 4.2]{Kang-V-NS17}).
 As in \cite[Proposition 3.4]{Kang-V-NS17} and \cite[Appendix A]{KV-unique19}, we use the notations:
\[
\s_* = \sqrt{-p'(v_*)},\qquad \alpha_* := \frac{\gamma \sqrt{-p'(v_*)} p(v_*)} {\gamma+1}.
\]
Note that
\beq\label{sigma-f}
 |\s_0 -\s_*| \le C\deltao ,
\eeq 
and 
\beq\label{sigma-p}
\|\s_*^2+p'(\tilde{v}_0)\|_{\infty} \leq C \deltao,
\qquad
\left\|\frac{1}{\sigma_*^2}-\frac{p(\tilde{v}_0)^{-\frac{1}{\gamma}-1}}{\gamma}\right\|_{\infty} \leq C \deltao.
\eeq
Thus, following the same estimates as in \cite[(3.30), (3.33),  (3.34)]{Kang-V-NS17} together with using Lemma \ref{lem:local} and \eqref{notw}-\eqref{sigma-p} above, we find that 
\begin{align}
&\left|\sigma_*^2 \frac{\lambda}{\eps^2 }\mathcal{Y}^g-\int_0^1W^2\,dy-2\int_0^1 W\,dy\right|\leq C\deltao\left(\int_0^1W^2\,dy+\int_0^1|W|\,dy\right) ,\label{Yfunc}\\
&{2\alpha_*} \frac{\lambda^2}{\eps^3}|\mathcal{I}_{1}|\leq \left(1 +C\deltao\right)\int_0^1 W^2\,dy,\label{B1func}\\
&{2\alpha_*} \frac{\lambda^2}{\eps^3}|\mathcal{I}_{2}|\leq \left(\frac{\alpha_\gamma}{\sigma_*}\left(\frac{\lambda}{\eps}\right)+C\deltao\right)\int_0^1 W^2\,dy. \label{B2func}
\end{align}
Using \eqref{Yfunc} and the assumption \eqref{assYp}, we have
$$
\int_0^1W^2\,dy-2\left|\int_0^1W\,dy\right|\leq C_2\sigma_*^2+C\deltao \left(\int_0^1W^2\,dy+\int_0^1|W|\,dy\right).
$$
Using
$$
\left|\int_0^1W\,dy\right|\leq \int_0^1|W|\,dy\leq \frac{1}{8} \int_0^1W^2\,dy+8 ,
$$
we have
$$
\int_0^1W^2\,dy\leq 2\left|\int_0^1W\,dy\right|+C_2\sigma^2+C\deltao\left(\int_0^1W^2\,dy+\int_0^1|W|\,dy\right)\leq C+24+\frac{1}{2}\int_0^1W^2\,dy,
$$
for $\delta_1$ small enough.
Thus there exists a  constant $C_1>0$  depending on $C_2$ (but not on $\eps$ nor $\lambda$)  such that 
\begin{equation}\label{controlW}
\int_0^1W^2\,dy\leq C_1.
\end{equation}

For an estimate on the $|\mathcal{Y}^g|^2$ terms, we use \eqref{Yfunc} to have
\begin{align*}
\begin{aligned}
-2\alpha_*\left(\frac{\lambda^2}{\eps^3}\right) \frac{|\mathcal{Y}^g|^2}{\eps \delta_1}
& =-\frac{2\alpha_*}{\deltao\sigma_*^4}\left| \frac{\sigma_*^2\lambda}{\eps^2}\mathcal{Y}^g\right|^2 \\
&\leq
  -\frac{\alpha_*}{\delta_1\sigma_*^4}\left| \int_0^1W^2\,dy+2\int_0^1 W\,dy\right|^2 +C\deltao\left(\int_0^1W^2\,dy+\int_0^1|W|\,dy \right)^2 ,
 \end{aligned}
\end{align*}
which together with \eqref{controlW} yields
\beq\label{Yfinal}
-2\alpha_*\left(\frac{\lambda^2}{\eps^3}\right) \frac{|\mathcal{Y}^g|^2}{\eps \delta_1} \le   -\frac{\alpha_*}{\delta_1\sigma_*^4}\left| \int_0^1W^2\,dy+2\int_0^1 W\,dy\right|^2 +C\deltao \int_0^1W^2\,dy .
\eeq

For the $\mathcal{G}_{2}$, we use  \eqref{notw}-\eqref{sigma-p} to have
\begin{align*}
\begin{aligned}
\mathcal{G}_{2} &= - \frac{\sigma_0\lam}{2\gamma}   \int_0^1 p(\tilde v_0)^{-\frac{1}{\gamma}-1} w^2 dy +\sigma_0\lam \frac{1+\gamma}{3\gamma^2}  \int_0^1 p(\tilde v_0)^{-\frac{1}{\gamma}-2} w^3 dy \\
&\ge \Big(  \frac{\lam}{2\sigma_*} -C\eps \deltao \Big)  \int_0^1 w^2 dy - \frac{\lam}{3\alpha_*} \int_0^1 w^3 dy - C\frac{\eps\lam}{\alpha_*} \int_0^1 |w|^3 dy,
 \end{aligned}
\end{align*}
which yields
\begin{align}
\begin{aligned}\label{G2final}
 -2\alpha_* \frac{\lambda^2}{\eps^3}\mathcal{G}_{2}\leq  \left(-\frac{\alpha_*}{\sigma_*}\left(\frac{\lambda}{\eps}\right) +C\deltao    \right)\int_0^1W^2\,dy +\frac{2}{3}\int_0^1W^3\,dy+C\eps\int_0^1|W|^3\,dy.
 \end{aligned}
\end{align}

Therefore, it remains to estimate the diffusion $\mathcal{D}$ as follows.\\
First, by the change of variable, the diffusion $\mathcal{D}$ is written as
\[
\mathcal{D}=\int_0^1 a |\partial_{y} w|^2 v^\beta \Big(-\frac{\partial y}{\partial x}\Big) dy.
\]
Since integrating \eqref{re_shock} over $(-\infty,x]$ yields that
\[
\tilde v_0^\beta \partial_x p(\tilde v_0)=\sigma_0 (\tilde v_0 -v_l) + \frac{p(\tilde v_0)-p(v_l)}{\sigma_0},
\]
we use \eqref{ach} to have
\[
\eps\, \tilde v_0^\beta \frac{\partial y}{\partial x}=\frac{1}{\sigma_0}\Big(\sigma_0^2 (\tilde v_0-v_l) + p(\tilde v_0)-p(v_l)\Big).
\]
Following the proof of  \cite[Lemma 3.1]{Kang-V-NS17}, with $1-y=\frac{p(v_l)-p(\tilde v_0)}{\eps}$, we have
\[
\frac{\tilde v_0^\beta}{y(1-y)}  \Big(-\frac{\partial y}{\partial x}\Big) =\frac{\eps}{\sigma_0 (v_l-v_r)}\left( \frac{\tilde v_0-v_r}{p(\tilde v_0)-p(v_r)}+\frac{\tilde v_0-v_l}{p(v_l)-p(\tilde v_0)} \right).
\]
Then
\begin{eqnarray*}
&&\left|\frac{\tilde v_0^\beta}{y(1-y)}  \Big(-\frac{\partial y}{\partial x}\Big) - \eps \frac{p''(v_*)}{2p'(v_*)^2\sigma_*}\right|\\
&&\qquad \le \underbrace{\left| \frac{\tilde v_0^\beta}{y(1-y)}  \Big(-\frac{\partial y}{\partial x}\Big) - \eps \frac{p''(v_r)}{2p'(v_r)^2\sigma_0} \right|}_{=:I_1} + \underbrace{\frac{\eps}{2}\Big| \frac{p''(v_r)}{p'(v_r)^2\sigma_0} -\frac{p''(v_*)}{p'(v_*)^2\sigma_*} \Big|}_{=:I_2}.
\end{eqnarray*}
Applying Lemma \ref{lem:press} to the case where $p_-=p(v_r)$, $p_+=p(v_l)$ and $p=p(\tilde v_0)$, we have
\begin{align*}
\begin{aligned}
I_1=\frac{\eps}{|\sigma_0 |(v_r - v_l)}\left| \frac{\tilde v_0-v_r}{p(\tilde v_0)-p(v_r)}+\frac{\tilde v_0-v_l}{p(v_l)-p(\tilde v_0)}+\frac{p''(v_r)}{2p'(v_r)^2}(v_r-v_l)  \right|\le C\eps^2.
\end{aligned}
\end{align*}
Using \eqref{sigma-f} and $|v_* -v_r|\le \deltao$, we have
$I_2\le C\eps \deltao$. Thus, we get
\[
\left|\frac{\tilde v_0^\beta}{y(1-y)}  \Big(-\frac{\partial y}{\partial x}\Big) - \eps \frac{p''(v_*)}{2p'(v_*)^2\sigma_*}\right| \le C\eps\deltao.
\]
Since $p(v)=v^{-\gamma}$, we have 
$$
\frac{p''(v_*)}{p'(v_*)^2\sigma_*}=\frac{\gamma+1}{\gamma \sigma_* p(v_*)}=\frac{1}{\alpha_*},
$$
which yields
\[
\left|\frac{\tilde v_0^\beta}{y(1-y)}  \Big(-\frac{\partial y}{\partial x}\Big) -\frac{ \eps}{2\alpha_*}\right| \le C\eps\deltao.
\]
Thus, using $|(v^\beta/\tilde v_0^\beta)-1| \le C\deltao$ and \eqref{approx1}, we have
\begin{align*}
\begin{aligned}
\mathcal{D}&\ge (1-2\delta_1)\int_0^1 |\partial_{y} w|^2 v^\beta \Big(-\frac{\partial y}{\partial x}\Big) dy\\
& =  (1-2\delta_1)\int_0^1 |\partial_{y} w|^2  \frac{v^\beta}{\tilde v_0^\beta}\tilde v_0^\beta \Big(-\frac{\partial y}{\partial x}\Big) dy\\
&\geq(1-2\delta_1) \left(\frac{\eps}{2\alpha_*}-C\eps\deltao \right)   \int_0^1y(1-y)  |\partial_{y} w|^2  \, dy\\
&\geq\frac{\eps}{2\alpha_*}(1-C\deltao)  \int_0^1y(1-y)  |\partial_{y} w|^2  \, dy .
\end{aligned}
\end{align*}
After the normalization, we obtain
\begin{equation}\label{newD}
-2\alpha_* \frac{\lambda^2}{\eps^3}\mathcal{D}\leq -(1-C\deltao) \int_0^1y(1-y)  |\partial_{y} W|^2  \, dy.
\end{equation}

To finish the proof, we first observe that 
for any $\delta\le \deltao$, 
 \begin{eqnarray*}
 &&\mathcal{R}_{\delta} (v) \le -\frac{1}{\eps\deltao}|\mathcal{Y}^g(v)|^2 +\mathcal{I}_{1}(v)+\deltao|\mathcal{I}_{1}(v)|\\
&&\quad\quad+\mathcal{I}_{2}(v)+\deltao\left(\frac{\eps}{\lambda}\right)|\mathcal{I}_{2}(v)|-\left(1-\deltao\left(\frac{\eps}{\lambda}\right)\right)\mathcal{G}_{2}(v)-(1-\deltao)\mathcal{D}(v).
 \end{eqnarray*}
Then,  \eqref{B1func},  \eqref{B2func}, \eqref{Yfinal},  \eqref{G2final}, \eqref{newD} imply that for some constants $C_\gamma, C_*>0$,
\begin{align*}
\begin{aligned}
\mathcal{R}_{\delta} (v) & \leq \frac{ \eps^3}{2\alpha_*\lam^2}\bigg[
-\frac{1}{C_\gamma\deltao}\left(\int_0^1W^2\,dy+2\int_0^1 W\,dy\right)^2+(1+C_* \deltao)\int_0^1 W^2\,dy\\
&\quad+\frac{2}{3}\int_0^1 W^3\,dy +C_*\delta_1 \int_0^1 |W|^3\,dy -(1-C_* \delta_1)\int_0^1 y(1-y)|\partial_{y} W|^2\,dy \bigg].
 \end{aligned}
\end{align*}

To finish the proof, we will use the nonlinear Poincar\'e type inequality \cite[Proposition 3.3]{Kang-V-NS17} as follow:
\begin{proposition}\label{prop:W}{\cite[Proposition 3.3]{Kang-V-NS17}}
For a given $C_1>0$, there exists $\deltat>0$, such that for any $\delta<\deltat$ the following is true.\\ For any $W\in L^2(0,1)$ such that 
$\sqrt{y(1-y)}\partial_yW\in L^2(0,1)$, if $\int_0^1 |W(y)|^2\,dy\leq C_1$, then
\begin{align}
\begin{aligned}\label{Winst}
&-\frac{1}{\delta}\left(\int_0^1W^2\,dy+2\int_0^1 W\,dy\right)^2+(1+\delta)\int_0^1 W^2\,dy\\
&\qquad\qquad+\frac{2}{3}\int_0^1 W^3\,dy +\delta \int_0^1 |W|^3\,dy  -(1-\delta)\int_0^1 y(1-y)|\partial_y W|^2\,dy \leq0.
 \end{aligned}
\end{align}
\end{proposition}

To apply the Proposition, let us fix the value of the $\delta_2$ of Proposition \ref{prop:W} corresponding to  the constant $C_1$ of \eqref{controlW}.\\
Then we retake  $\delta_1$ small enough such that  $\max(C_\gamma, C_*) \delta_1\le\delta_2$. Thus we find that
\begin{align*}
\mathcal{R}_{\delta} (v)
& \leq \frac{ \eps^3}{2\alpha_*\lam^2}\bigg[
-\frac{1}{\delta_2}\left(\int_0^1W^2\,dy+2\int_0^1 W\,dy\right)^2+(1+\delta_2)\int_0^1 W^2\,dy\\
&\quad+\frac{2}{3}\int_0^1 W^3\,dy +\delta_2 \int_0^1 |W|^3\,dy -(1-\delta_2)\int_0^1 y(1-y)|\partial_{y} W|^2\,dy \bigg].
\end{align*}
Therefore, using Proposition  \ref{prop:W}, we have 
\[
\mathcal{R}_{\delta}(v) \le 0,
\]
which completes the proof.

\section{Existence and uniqueness of shifts} \label{app-shift}
\setcounter{equation}{0}

\setcounter{equation}{0}
We here prove the existence and uniqueness of the shift functions defined by the non-autonomous system \eqref{X-def}. For any fixed $\eps_1, \eps_2$ and $U\in\mathcal{X}_T$, let
$F:[0,T]\times\bbr^2\rightarrow\bbr^2$ denote the right-hand side of \eqref{X-def} by
\[
F\big(t, X_1, X_2\big) = {
  \Phi_{\eps_1} (Y_1(U)) \Big(2|\mathcal{J}^{bad}(U)|+1 \Big) -\frac{\s_1}{2}  \Psi_{\eps_1} (Y_1(U)) 
\choose     -\Phi_{\eps_2} (-Y_2(U)) \Big(2|\mathcal{J}^{bad}(U)|+1 \Big) -\frac{\s_2}{2}  \Psi_{\eps_2} (-Y_2(U)) } .
\]
We may show that there exists $a, b\in L^q(0,T)$ (for some $q\ge1$) such that
\[
\sup_{(X_1,X_2)\in\bbr^2 }|F(t,X_1,X_2)|\leq a(t),\qquad \sup_{(X_1,X_2)\in\bbr^2 }|\nabla_{X_1,X_2} F(t,X_1,X_2)|\leq b(t) .
\]
To this end, we use the facts that
\begin{align*}
&\mbox{ $\Phi_{\eps_i}$ and $\Psi_{\eps_i}$ are Lipschitz and bounded};\\
&\mbox{for the solution $(v,h)\in\mathcal{X}_T,$}\quad v, v^{-1} \in L^\infty((0,T)\times\bbr) \quad \mbox{and}\quad h-\tih, v_x \in L^\infty(0,T;L^2(\bbr)); \\
&|a^{X_1,X_2}|+|\tih^{X_1,X_2}|\le C,   C^{-1}\le \tiv^{X_1,X_2} \le C,  \mbox{where the constant $C$ is uniform w.r.t. } X_1, X_2 ; \\
&\mbox{for any $r\in[1,\infty]$},\quad L^r(\bbr)\mbox{-norms of }  (a_i)_x^{X_i}, (a_i)_{xx}^{X_i}, (\tiv_i)_x^{X_i}, (\tiv_i)_{xx}^{X_i}, (\tiv_i)_{xxx}^{X_i},  (\tih_i)_x^{X_i}, (\tih_i)_{xx}^{X_i} \\
&\qquad\qquad\qquad\qquad\mbox{ are uniform w.r.t. } X_1, X_2 .
\end{align*}
Thus, there exists a constant $C_*$ (uniform w.r.t. $ X_1, X_2$) such that
\begin{align*}
|F(t,X_1,X_2)|&\le C_*\sum_{i=1}^2 \bigg[ \|(a_i)_x^{X_i}\|_{L^2(\bbr)} \|h-\tih\|_{L^\infty(0,T;L^2(\bbr))} + \|(a_i)_x^{X_i}\|_{L^1(\bbr)} \|\tih-\tih^{X_1,X_2}\|_{L^\infty(\bbr)} \\
& + \Big(\|(a_i)_x^{X_i}\|_{L^2(\bbr)} + \|(\tiv_i)_x^{X_i}\|_{L^2(\bbr)} \Big)  \Big(  \|v_x\|_{L^\infty(0,T;L^2(\bbr))} +  \|(\tiv_i)_x^{X_i}\|_{L^2(\bbr)} \Big) +1\bigg] \\
&\le C_*,
\end{align*}
which especially implies
\[
\sup_{(X_1,X_2)\in\bbr^2 }|F(t,X_1,X_2)|\leq a(t) \in L^1(0,T).
\]
Likewise, we have
\[
|\nabla_{X_1,X_2} F(t,X_1,X_2)|\le C_*,
\]
which implies
\[
\sup_{(X_1,X_2)\in\bbr^2 }|\nabla_{X_1,X_2} F(t,X_1,X_2)|\leq b(t) \in L^2(0,T).
\]
Therefore, the system \eqref{X-def} has a unique absolutely continuous solution thanks to the following lemma. This lemma is a simple extension of \cite[Lemma A.1]{CKKV} (which is for scalar ODE). So, we omit the proof. 
 
 \begin{lemma}\label{lem_cauchy_lip}
 Let $p>1$, $T>0$ and $n\in\mathbb{N}$. Suppose that a function 
 $F:[0,T]\times\bbr^n\rightarrow\bbr^n$  satisfies that
 \[
\sup_{X\in\bbr^n }|F(t,X)|\leq a(t),\qquad \sup_{X\in\bbr^n }|\nabla_{X} F(t,X)|\leq b(t),
\]
for some functions $a \in L^1(0,T)$ and $\, b\in L^p(0,T)$. Then for any $x_0\in\bbr$, there exists a unique absolutely continuous solution $X:[0,T]\rightarrow \bbr^n$ to the system of ODEs:
\[
\left\{ \begin{array}{ll}
        \dot X(t) = F(t,X(t))\quad\mbox{for \textit{a.e.} }t\in[0,T],\\
       X(0)=x_0 .\end{array} \right.
\]
 \end{lemma}

\section{Proof of Lemma \ref{lem:mlsc}}\label{app:mlsc}
Since $v^k$ are positive measures in $\mathcal{M}((0,T)\times\bbr)$, Radon-Nikodym's theorem implies that there exist positive measures $v^k_a \in L^1(\bbr)$ and $dv^k_s$ (singular part of $v^k$) such that 
\[
d v^k (t,x) = v^k_a(t,x) dtdx + dv^k_s (t,x) .
\]
For simplicity, we set ${\bf{\bar v}}:=\bar v^{X_1^\infty,X_2^\infty}(t,x)$. 
To truncate $ v^k_a$ by a large constant, we first use the definition of the relative functional $Q(\cdot|\cdot)$, and the fact that $Q(v)\to 0$ as $v\to\infty$, which imply 
that for any $\eps>0$, there exists a constant $\xi>0$ with $\xi>\max(2 v_-, 2v_+, 2 v_m^{-1})$ such that for all $v>\xi$,
\beq \label{large-con}
(|Q'( {\bf{\bar v}})| +\eps) v \ge Q(v|{\bf{\bar v}}) \ge (|Q'({\bf{\bar v}})| -\eps) v .
\eeq
For such a constant $\xi$, we define
\[
v^k_\xi := \min (v^k_a,\xi) ,
\]
and 
\[
Q_\xi (v):=\left\{ \begin{array}{ll}
       Q(v),\quad\mbox{if }  v\ge \xi^{-1},\\
        Q'(\xi^{-1}) (v-\xi^{-1}) + Q(\xi^{-1}),\quad\mbox{if } v\le\xi^{-1}.\end{array} \right.
\]
Note that $v\mapsto Q_\xi(v)$ is nonnegative and convex $C^1$-function on $[0,\infty)$, and $Q_\xi'({\bf{\bar v}})=Q'({\bf{\bar v}})$ (by $\xi^{-1}<v_m/2<{\bf{\bar v}}$). 
Then, consider its relative functional: for any nonnegative $v_1,v_2\ge 0$,
\[
Q_\xi(v_1|v_2) := Q_\xi(v_1) -Q_\xi(v_2)-Q_\xi'(v_2)(v_1-v_2) .
\]
Then, using \eqref{large-con}, we have
\beq\label{Q-cut}
d Q_\xi (v^k | {\bf{\bar v}}) \ge Q_\xi (v^k_\xi | {\bf{\bar v}}) dtdx + (|Q_\xi'(\overline V)| -\eps) (dv^k - v^k_\xi dtdx) - 2\eps d v^k ,
\eeq
which means that $d Q_\xi (v^k | {\bf{\bar v}}) -\big[ Q_\xi (v^k_\xi | {\bf{\bar v}}) dtdx + (|Q_\xi'(\overline V)| -\eps) (dv^k - v^k_\xi dtdx) - 2\eps d v^k\big]$ is nonnegative measure.
Indeed, this is verified as follows: If $v^k_a\le\xi$, then $v^k_\xi=v^k_a$, and so
\begin{align*}
\begin{aligned}
\mbox{LHS} &:=Q_\xi (v^k_a | {\bf{\bar v}}) dtdx+ |Q_\xi'(\overline V)| dv^k_s = Q_\xi(v^k_\xi| {\bf{\bar v}}) dtdx+ |Q_\xi'(\overline V)| dv^k_s \ge \mbox{RHS} ,
\end{aligned}
\end{align*}
where the last inequality follows from the facts that (by Radon-Nikodym's theorem) the measure $v^k - v^k_\xi$ is positive and
$v^k - v^k_\xi= (v^k_a - v^k_\xi) + v^k_s = v^k_s$.\\
If $v^k_a>\xi$, then $v^k_\xi=\xi$, and using \eqref{large-con}, $\mbox{LHS} \ge (|Q_\xi'({\bf{\bar v}})| -\eps) v^k_a dtdx + |Q_\xi'(\overline V)| dv^k_s $.\\
Since ${\bf{\bar v}}=\overline V$ $dtdx$-a.e. (by \eqref{def-Omv} and \eqref{def-vom}), we have
\begin{align*}
\begin{aligned}
\mbox{LHS} &\ge (|Q_\xi'(\overline V)| -\eps) dv^k = (|Q_\xi'(\overline V)| +\eps) \xi dtdx + (|Q_\xi'(\overline V)| -\eps) (d v^k- \xi dtdx) - 2\eps\xi dtdx\\
&= (|Q_\xi'({\bf{\bar v}})| +\eps) \xi dtdx + (|Q_\xi'(\overline V)| -\eps) (d v^k- v^k_\xi dtdx) - 2\eps\xi dtdx \\
&\ge  Q_\xi (\xi | {\bf{\bar v}}) dtdx + (|Q_\xi'(\overline V)| -\eps) (d v^k- v^k_\xi dtdx) - 2\eps\xi dtdx .
\end{aligned}
\end{align*}
Thus, using $\xi=v^k_\xi\le v^k_a \le v^k$, we have \eqref{Q-cut}.\\

Therefore, we use \eqref{Q-cut} to have
\begin{align*}
\begin{aligned} 
C_0 &\ge \limsup_{k\to\infty}\int_{(0,T)\times\bbr} \Phi(t,x)~d Q\left(v^k | {\bf{\bar v}}\right) (t,x) \\
&\ge \limsup_{k\to\infty}\int_{(0,T)\times\bbr} \Phi(t,x) \Big[ Q_\xi (v^k_\xi | {\bf{\bar v}})dtdx  + (|Q'_\xi(\overline V)| -\eps) d(v^k - v^k_\xi) -2\eps d v^k \Big] .
\end{aligned}
\end{align*}
We set $\Omega_m:=\left(\overline{\Omega_M^-}\cup\overline{\Omega_M^+}\right)^c$, and define  
\[
\Omega_m^{\delta} := \{ (t,x)\in \Omega_m~|~ d((t,x)|\Omega_m^c)>\delta \},\quad \forall \delta>0.
\]
Then we define a smooth function $\psi_1^{\delta}$ such that
\beq\label{defpsi1}
\psi_1^\delta (t,x) :=\left\{ \begin{array}{ll}
       1,\quad\mbox{on } (\Omega_m^\delta)^c ,\\
        0,\quad\mbox{on } \Omega_m^{2\delta} .\end{array} \right.
\eeq
Then, using  the fact that $|Q'_\xi(v_\pm)|\le |Q_\xi'(v_m)|$ together with \eqref{def-Omv} and \eqref{def-vom},
we have
\begin{align*}
\begin{aligned} 
C_0 &\ge \limsup_{k\to\infty} \Big[ \sum_{\pm}\int_{\Omega_M^\pm} \Phi  Q_\xi (v^k_\xi | v_\pm)dtdx +\int_{\Omega_m} \Phi  Q_\xi (v^k_\xi | v_m)dtdx \\
&\quad + (|Q_\xi'(v_-)| -\eps) \int_{x<0} \Phi \psi_1^\delta d(v^k - v^k_\xi)+ (|Q_\xi'(v_+)| -\eps) \int_{x\ge 0} \Phi \psi_1^\delta d(v^k - v^k_\xi) \\
&\quad  + (|Q_\xi'(v_m)| -\eps) \int \Phi (1-\psi_1^\delta) d(v^k - v^k_\xi) -2\eps \int \Phi d v^k  \Big].
\end{aligned}
\end{align*}
Note that since $|v^k_\xi|\le \xi$ for all $k$, there exists $v_*$ such that
\[
v^k_\xi  \rightharpoonup v_* \quad\mbox{in }~ L^\infty .
\]
Moreover, since the function $v\mapsto Q_\xi(v|c)$ with any fixed $c$ is convex, the weakly lower semi-continuity of convex functions (for example, see \cite{evans-w}) implies
\begin{align*}
\begin{aligned} 
& \liminf_{k\to\infty} \Big[ \sum_{\pm}\int_{\Omega_M^\pm} \Phi  Q_\xi (v^k_\xi | v_\pm)dtdx +\int_{\Omega_m} \Phi  Q_\xi (v^k_\xi | v_m)dtdx  \Big] \\
&\quad \ge  \sum_{\pm}\int_{\Omega_M^\pm} \Phi  Q_\xi (v_*| v_\pm)dtdx +\int_{\Omega_m} \Phi  Q_\xi( v_* | v_m)dtdx .
\end{aligned}
\end{align*}
In addition, since $v^k \rightharpoonup v_\infty$ in $\mathcal{M}_{\mathrm{loc}}(\bbr^+\times\bbr)$, and so
\beq\label{twoconv}
v^k -v^k_\xi \rightharpoonup v_\infty - v_* \quad\mbox{ in }~\mathcal{M}_{\mathrm{loc}}(\bbr^+\times\bbr)\quad\mbox{(by the uniqueness of the decomposition)}, 
\eeq
we have
\begin{align}
\begin{aligned} \label{cest1}
C_0 & \ge  \sum_{\pm}\int_{\Omega_M^\pm} \Phi  Q_\xi (v_*| v_\pm)dtdx +\int_{\Omega_m} \Phi  Q_\xi( v_* | v_m)dtdx \\
&\quad + (|Q_\xi'(v_-)| -\eps) \int_{x<0} \Phi \psi_1^\delta d( v_\infty - v_*)+ (|Q_\xi'(v_+)| -\eps) \int_{x\ge 0} \Phi \psi_1^\delta d( v_\infty - v_*) \\
&\quad  + (|Q_\xi'(v_m)| -\eps) \int \Phi (1-\psi_1^\delta) d( v_\infty - v_*)
-2\eps \int \Phi d v_\infty \\
&=: \mathcal{R} .
\end{aligned}
\end{align}
Note that for the decomposition $d v_\infty (t,x) = v_a(t,x) dtdx + dv_s (t,x)$, since the measure $v^k -v^k_\xi $ is positive, it follows from \eqref{twoconv} and \eqref{RNv} that $v_\infty -v_* $, $v_a - v_* $ and $dv_s$ are all nonnegative.\\
Since $dv_\infty - v_*  dtdx= (v_a-v_* ) dtdx+ dv_s$ (by the uniqueness of the decomposition), we rewrite $\mathcal{R} $ in \eqref{cest1} as
\[
\mathcal{R} = \mathcal{R}_1 +\mathcal{R}_2+\mathcal{R}_3,
\]
where
\begin{align*}
\begin{aligned} 
 \mathcal{R}_1  & :=  |Q_\xi'(v_-)|  \int_{x<0} \Phi \psi_1^\delta dv_s +|Q_\xi'(v_+)|  \int_{x\ge0} \Phi \psi_1^\delta dv_s+  |Q_\xi'(v_m)| \int \Phi (1-\psi_1^\delta) dv_s ,\\
 \mathcal{R}_2 &  :=   \sum_{\pm}\int_{\Omega_M^\pm} \Phi  Q_\xi (v_*  | v_\pm) dtdx+ |Q'_\xi(v_-)| \int_{x<0}\Phi \psi_1^\delta (v_a - v_*  ) dtdx  \\
&\quad + |Q'_\xi(v_+)| \int_{x\ge0}\Phi \psi_1^\delta (v_a - v_*  ) dtdx 
    +\int_{\Omega_m} \Phi   Q_\xi (v_*  | v_m) dtdx\\
 &\quad   + |Q'_\xi(v_m)|  \int \Phi (1-\psi_1^\delta)  (v_a - v_*  )  dtdx ,\\
\mathcal{R}_3 &  := - 3\eps \int \Phi  d v_\infty +\eps \int \Phi v_*  dtdx.
\end{aligned}
\end{align*}
Using $\overline{\Omega_M}\subset (\Omega_m^\delta)^c$ and \eqref{defpsi1}, we have
\begin{align*}
\begin{aligned} 
 \mathcal{R}_1  & \ge \sum_\pm |Q'_\xi(v_\pm)|  \int_{\overline{\Omega_M^\pm}} \Phi  dv_s +  |Q'_\xi(v_m)| \int_{\Omega_m^{2\delta}} \Phi dv_s .
\end{aligned}
\end{align*}
Since $\Phi dv_s$ is a positive measure, and
\beq\label{conv-om}
\Omega_m^{2\delta} \nearrow \Omega_m = \left(\overline{\Omega_M^-}\cup\overline{\Omega_M^+}\right)^c,
\eeq
we have
\[
\lim_{\delta\to 0} \int_{\Omega_m^{2\delta}} \Phi dv_s = \int_{\left(\overline{\Omega_M^-}\cup\overline{\Omega_M^+}\right)^c} \Phi dv_s .
\]
Thus,
\[
 \mathcal{R}_1  \ge \int   |Q'_\xi(\overline V)| \Phi dv_s .
\]
For $ \mathcal{R}_2$, we use \eqref{defpsi1} to have
\begin{align*}
\begin{aligned} 
 \mathcal{R}_2 &  \ge\sum_\pm \int_{\Omega_M^\pm} \Phi \Big[ Q_\xi (v_*  | v_\pm) + |Q'_\xi(v_\pm)|  (v_a - v_*  ) \Big] dtdx  \\
&\quad    +\int_{\Omega_m^{2\delta}} \Phi \Big[   Q_\xi (v_*  | v_m) + |Q'_\xi(v_m)|  (v_a - v_*  ) \Big] dtdx .\\
\end{aligned}
\end{align*}
Then, we have
\begin{align*}
\begin{aligned} 
 \mathcal{R}_2 &  \ge \sum_\pm \int_{\Omega_M^\pm}  \Phi Q_\xi (v_a | v_\pm)  dtdx    +\int_{\Omega_m^{2\delta}} \Phi   Q_\xi (v_a | v_m)  dtdx ,
\end{aligned}
\end{align*}
where we used the inequality that for any $w_1, w_2\ge 0$ and any $c>0$,
\[
Q_\xi (w_1+w_2 | c) \le Q_\xi (w_1 | c) + |Q'_\xi (c)|w_2.
\]
Indeed, it follows from $Q_\xi'\le 0$ and the definition of $Q_\xi(\cdot|\cdot)$ that
\[
Q_\xi (w_1+w_2 | c) -Q_\xi (w_1 | c) - |Q'_\xi (c)|w_2 = Q_\xi (w_1+w_2)-Q_\xi(w_1) \le 0 .
\]
Since \eqref{conv-om} imiplies
\[
\lim_{\delta\to 0} \int_{\Omega_m^{2\delta}} \Phi   Q_\xi (v_a | v_+)  dtdx= \int_{\left(\overline{\Omega_M^-}\cup\overline{\Omega_M^+}\right)^c} \Phi   Q_\xi (v_a | v_m)  dtdx ,
\]
we use \eqref{bdry} to have
\[
 \mathcal{R}_2 \ge \int \Phi   Q_\xi (v_a | \bar v^{X_1^\infty,X_2^\infty})  dtdx .
\]
Therefore, we have
\[
\mathcal{R} \ge \int \Phi   Q_\xi (v_a | \bar v^{X_1^\infty,X_2^\infty})  dtdx +  \int  \Phi  |Q'_\xi(\overline V)| dv_s - 3\eps \int \Phi  d v_\infty ,
\]
that is,
\[
\int \Phi   Q_\xi (v_a | \bar v^{X_1^\infty,X_2^\infty})  dtdx +  \int  \Phi  |Q'_\xi(\overline V)| dv_s \le \mathcal{R}  + 3\eps \int \Phi  d v_\infty.
\]
Therefore, taking $\xi\to\infty$ with Fatou's lemma, we have
\[
\int \Phi   Q (v_a | \bar v^{X_1^\infty,X_2^\infty})  dtdx +  \int  \Phi  |Q'(\overline V)| dv_s \le \mathcal{R}  + 3\eps \int \Phi  d v_\infty.
\]
Then taking $\eps\to 0$, we have
\[
\int \Phi   Q (v_a |\bar v^{X_1^\infty,X_2^\infty})  dtdx +  \int  \Phi  |Q'(\overline V)| dv_s \le \mathcal{R} 
\]
This completes the proof.

\end{appendix}

\bibliography{Kang-Vasseur2019}
\end{document}